\documentclass[11pt]{article}
\RequirePackage[OT1]{fontenc}
\RequirePackage{amsthm,amsmath,natbib}
\RequirePackage[colorlinks]{hyperref}
\RequirePackage{hypernat}
\usepackage{amsfonts}
\usepackage{amssymb}
\usepackage{amstext}
\usepackage{parskip}
\usepackage{fullpage}
\bibpunct{(}{)}{;}{a}{,}{,}

\usepackage{times}
\usepackage{graphicx}
\usepackage{epsfig}

\newcommand{\R}{\ensuremath{\mathbb R}}

\newcommand{\Sc}{\ensuremath{S^c}}

\newcommand{\prob}[1]{\ensuremath{{\mathbb P}\left(#1\right)}}

\newcommand{\expct}[1]{\ensuremath{{\mathbb E}#1}}

\newcommand{\size}[1]{\ensuremath{\left|#1\right|}}

\newcommand{\argmin}{\operatorname{argmin}}

\newcommand{\e}{\epsilon}

\newcommand{\silent}[1]{}

\newcommand{\ip}[1]{\;\langle{\,#1\,}\rangle\;}

\newcommand{\T}{{\mathcal T}}

\newcommand{\Id}{\ensuremath{\textsf{Id}}}

\newcommand{\basepen}{\ensuremath{\lambda_{\sigma,a,p}}}

\newcommand{\lars}{\textsf{LARS}}

\newcommand{\init}{\text{\rm init}}
\newcommand{\drop}{{\mathcal D}}
\newcommand{\dropS}{\ensuremath{{S_{\drop}}}}

\newcommand{\inv}[1]{\frac{1}{#1}}
\newcommand{\abs}[1]{\left\lvert#1\right\rvert}
\newcommand{\twonorm}[1]{\left\lVert#1\right\rVert_2}

\newcommand{\shtwonorm}[1]{\lVert#1\rVert_2}
\newcommand{\norm}[1]{\left\lVert#1\right\rVert}

\newcommand{\bens}{\begin{eqnarray*}}
\newcommand{\eens}{\end{eqnarray*}}
\newcommand{\ben}{\begin{eqnarray}}
\newcommand{\een}{\end{eqnarray}}

\newcommand{\beq}{\begin{equation}}
\newcommand{\eeq}{\end{equation}}

\def\E{{\textbf{E}}}
\def\supp{\mathop{\text{\rm supp}\kern.2ex}}
\def\argmin{\mathop{\text{arg\,min}\kern.2ex}}

\let\hat\widehat
\let\tilde\widetilde

\numberwithin{equation}{section}
\theoremstyle{plain}
\newtheorem{theorem}{Theorem}[section]
\newtheorem{assumption}{Assumption}[section]
\newtheorem{lemma}[theorem]{Lemma}
\newtheorem{proposition}[theorem]{Proposition}

\newtheorem{remark}[theorem]{Remark}
\newtheorem{corollary}[theorem]{Corollary}

{\end{eqnarray}\end{subequations}\hskip-4.0pt}%
\newenvironment{proofof}[1]{\hspace*{20pt}{\it Proof}{ of #1}.\hskip10pt}{\qed\vskip5pt}
\newenvironment{proofof2}{\hskip10pt}{\qed\vskip5pt}

\begin{document}

\title{Thresholded Lasso for high dimensional variable selection and
statistical estimation \footnote{
A preliminary version of this paper with title:
Thresholding Procedures for High Dimensional Variable Selection and
Statistical Estimation, has appeared in 
Proceedings of Advances in Neural Information Processing Systems 22, (NIPS 2009).
This research was supported by the Swiss National Science Foundation (SNF)
Grant 20PA21-120050/1. 
}
}

\author{Shuheng Zhou \\
  \vspace{-.3cm}\\
Seminar f\"{u}r Statistik, Department of Mathematics, ETH Z\"{u}rich, CH-8092, Switzerland
}

\date{Feburary 8, 2010} 

\maketitle

\begin{abstract}
\noindent 
Given $n$ noisy samples with $p$ dimensions, where $n \ll p$, we show 
that the multi-step thresholding procedure based on the 
Lasso -- we call it the {\it Thresholded Lasso}, can accurately estimate a 
sparse vector $\beta \in \R^p$ in a linear model $Y = X \beta + \epsilon$, 
where $X_{n \times p}$ is a design matrix normalized to have column 
$\ell_2$ norm $\sqrt{n}$, and $\epsilon \sim N(0, \sigma^2 I_n)$.
We show that under the restricted eigenvalue (RE) condition 
(Bickel-Ritov-Tsybakov 09), it is possible to achieve the $\ell_2$ loss 
within a logarithmic factor of the ideal mean square error one would 
achieve with an {\em oracle }  while selecting a sufficiently sparse model -- 
hence achieving {\it sparse oracle inequalities}; the oracle would supply perfect 
information about which coordinates are non-zero and which are above the 
noise level. In some sense, the Thresholded Lasso recovers the choices 
that would have been made by the $\ell_0$ penalized least squares 
estimators, in that it selects a sufficiently sparse model without sacrificing 
the accuracy in estimating $\beta$ and in predicting $X \beta$. We also 
show for the Gauss-Dantzig selector (Cand\`{e}s-Tao 07), if $X$ obeys a
uniform uncertainty principle and if the true parameter is sufficiently sparse,
one will achieve the sparse oracle inequalities as above, while
allowing at most $s_0$ irrelevant variables in the model  in the worst case,
where $s_0 \leq s$ is the smallest integer such that for $\lambda = \sqrt{2 \log
p/n}$, $\sum_{i=1}^p \min(\beta_i^2, \lambda^2 \sigma^2) \leq s_0 \lambda^2
\sigma^2$.  Our simulation results on the Thresholded Lasso match our
theoretical analysis excellently.
\end{abstract}


\section{Introduction}
\label{sec:introduction}
In a typical high dimensional setting, the number of variables $p$ is much 
larger than the number of observations $n$. This challenging setting appears
in linear regression, signal recovery, covariance selection in graphical 
modeling, and sparse approximations.
In this paper, we consider recovering $\beta \in \R^p$ in the following 
linear model:
\beq
\label{eq::linear-model}
Y = X \beta + \epsilon,
\eeq
where $X$ is an $n \times p$ design matrix, $Y$ is a vector
of noisy observations and $\epsilon$ is the noise term. 
%The design matrix is treated as either fixed or random. 
We assume throughout this paper that $p \ge n$ (i.e. high-dimensional), 
$\epsilon \sim N(0, \sigma^2 I_n)$, and the columns of $X$ are 
normalized to have $\ell_2$ norm $\sqrt{n}$.
Given such a linear model, two key tasks are to identify the relevant
set of variables and to estimate $\beta$ with bounded $\ell_2$ loss.
In particular, recovery of the sparsity pattern 
$S = \supp(\beta) := \left\{j \,:\, \beta_j \neq 0\right\}$, 
also known as variable (model) selection, 
refers to the task of correctly identifying the support set
(or a subset of ``significant'' coefficients in $\beta$) 
based on the noisy observations.

Even in the noiseless case, recovering $\beta$ (or its support) 
from $(X, Y)$ seems impossible when $n \ll p$. 
However, a line of recent research shows that
when $\beta$ is sparse: when it has a relatively small number of 
nonzero coefficients and when the design matrix $X$ is also sufficiently 
nice,  it becomes possible~\cite{Donoho:cs,CRT06,CT05,CT06}.
One important stream of research,  which we also adopt here, requires 
computational feasibility for the estimation methods, among which the 
Lasso and the Dantzig selector are both well studied and shown with 
provable nice statistical properties; see 
for example~\cite{MB06,GR04,Wai08,ZY07,CT07,vandeG08,MY09,BRT09,RWL08}.
For a chosen penalization parameter $\lambda_n \geq 0$, regularized 
estimation with the $\ell_1$-norm penalty, 
also known as the Lasso \citep{Tib96} or 
Basis Pursuit~\citep{Chen:Dono:Saun:1998} 
refers to the following convex optimization problem  
\begin{eqnarray}
\label{eq::origin} \; \; 
\hat \beta = \arg\min_{\beta} \frac{1}{2n}\|Y-X\beta\|_2^2 + 
\lambda_n \|\beta\|_1,
\end{eqnarray}
where the scaling factor $1/(2n)$ is chosen by convenience; The Dantzig 
selector~\citep{CT07} is defined as,
\begin{gather}
\label{eq::DS-func}
(DS) \; \; \arg\min_{\hat{\beta} \in \R^p} \norm{\hat{\beta}}_1
 \;\; \text{subject to} \;\; 
\norm{\inv{n} X^T (Y - X \hat{\beta})}_{\infty} \leq \lambda_n.
\end{gather}
Our goal in this work is to recover $S$ 
%:=\supp(\beta)$ 
as accurately as possible: we wish to obtain $\hat{\beta}$ such that 
$|\supp(\hat{\beta}) \setminus S|$ 
(and sometimes $|S \triangle \supp(\hat{\beta})|$ also) is small, 
with high probability, while at the same time 
$\shtwonorm{\hat{\beta} - \beta}^2$ is bounded within logarithmic factor of 
the ideal mean square error one would achieve with an oracle which would 
supply perfect information about which coordinates are non-zero and which 
are above the noise level (hence achieving the {\it oracle inequality} as 
studied in~\cite{Donoho:94,CT07}); We deem the bound on $\ell_2$-loss 
as a natural criteria for evaluating a sparse model when it is not exactly $S$.
Let {\bf $s = |S|$}.

Given $T \subseteq \{1, \ldots, p\}$, 
let us define $X_T$ as the $n \times |T|$ 
submatrix obtained by extracting columns of $X$ indexed by $T$; similarly, 
let $\beta_T \in \R^{|T|}$, be a subvector of $\beta \in \R^p$ confined to $T$.
Formally, we propose and study a {\bf Multi-step Procedure}: 
First we obtain an initial estimator $\beta_{\init}$ using the Lasso as 
in~\eqref{eq::origin} or the Dantzig selector as in~\eqref{eq::DS-func}, 
with $\lambda_n = d \sigma \sqrt{2\log p/n}$, for some constant $d > 0$.
\begin{enumerate}
\item
We then threshold the estimator $\beta_{\init}$ with $t_0$,
with the general goal such that, we get a set $I_1$ with cardinality at 
most $2s$; in general, we also have $|I_1 \cup S| \leq 2s$, where 
$I_1 =  \left\{j \in \{1, \ldots, p\}: \beta_{j, \init} \geq t_0 \right\}$ 
for some $t_0$ to be specified. Set $I = I_1$.
\item
%In the second stage, we assume that we have a set $I$ of at most $2s$ number 
%of variables (as guaranteed by step 1).
We then feed $(Y, X_{I})$ to either the Lasso estimator 
as in~\eqref{eq::origin} or the ordinary least squares (OLS) estimator to
obtain $\hat{\beta}$, where we set
$\hat\beta_{I} = (X_I^T X_{I})^{-1} X_{I}^T Y$ and 
$\hat{\beta}_{I^c} = 0$.
\item
Possibly threshold $\hat{\beta}_{I_1}$ with 
$t_1 = 4 \lambda_{n} \sqrt{|I_1|}$ to obtain $I_2$, and
repeat step~$2$ with $I = I_2$ to obtain $\hat\beta_{I}$; set other 
coordinates to zero and return $\hat{\beta}$.
\end{enumerate}
Our algorithm is constructive in that it relies neither on the unknown 
parameters $s$ and {\bf $\beta_{\min} := \min_{j \in S} |\beta_{j}|$},
nor the exact knowledge of those that characterize the incoherence 
conditions on $X$; instead, our choice of $\lambda_n$ and thresholding 
parameters only depends on $\sigma, n$, and $p$, and some crude 
estimation of certain parameters, which we will explain in later 
sections.
In our experiments, we apply only the first two steps with the Lasso as an 
initial estimator, which we refer to as the {\it Thresholded Lasso} estimator; 
the Gauss-Dantzig selector is a two-step procedure with the Dantzig selector 
as $\beta_{\init}$~\cite{CT07}. We apply the third step only when 
$\beta_{\min}$ is sufficiently large, so as to get a very sparse model 
$I \supset S$ (cf. Theorem~\ref{THM:RE}).
We now formally define some incoherence conditions in 
Section~\ref{sec:cond-intro} and elaborate on our goals in
Section~\ref{sec:oracle-intro}, where we also outline the rest of this section.

\subsection{Incoherence conditions}
\label{sec:cond-intro}
For a matrix $A$, let $\Lambda_{\min}(A)$ and $\Lambda_{\max}(A)$
denote the smallest and the largest eigenvalues respectively. 
We refer to a vector $\upsilon \in \R^{p}$ with at most 
$s$ non-zero entries, where $s \leq p$, as a {\bf $s$-sparse} vector.
Occasionally, we use $\beta_T \in \R^{|T|}$, where  
$T \subseteq \{1, \ldots, p\}$, 
to also represent its $0$-extended version $\beta' \in \R^p$
such that $\beta'_{T^c} = 0$ and $\beta'_{\T} = \beta_T$; 
for example in~\eqref{eq::beta-diamond} below.
%Throughout this paper,  we let 
%$T_0$ denote the set of largest coefficients in $\beta$ in absolute values;
%cf.~\eqref{eq::T0-define}. 
We assume
\beq
\label{eq::eigen-admissible-s}
\Lambda_{\min}(2s) \;
\stackrel{\triangle}{=} \;
\min_{\upsilon \not= 0; 2s-\text{sparse}} \; \;
\frac{\twonorm{X \upsilon}^2}{n \twonorm{\upsilon}^2} > 0,
\eeq
where $n \geq 2s$ is necessary, as any submatrix with more than $n$ 
columns must be singular. In general, we also assume that 
\ben
\label{eq::eigen-max}
\Lambda_{\max}(2s) \; \stackrel{\triangle}{=} \;
\max_{\upsilon \not=0; 2s-\text{sparse}} \;
% \supp(\upsilon) \leq s} \;
\frac{\twonorm{X \upsilon}^2}{n \twonorm{\upsilon}^2} < \infty.
\een
\cite{CT05} define the $s$-{\em restricted isometry constant} 
$\delta_s$ of $X$ to be the smallest quantity such that
for all $T \subseteq \{1,\ldots, p\}$ with $|T| \leq s$ and coefficients 
sequences $(\upsilon_j)_{j \in T}$, it holds that
\ben
\label{eq::RIP-define}
(1 - \delta_s) \twonorm{\upsilon}^2 \leq \twonorm{X_T \upsilon}^2/n \leq 
(1 + \delta_s) \twonorm{\upsilon}^2;
\een
The $(s, s')$-{\em restricted orthogonality constant} $\theta_{s, s'}$ is the 
smallest quantity such that for all disjoint sets $T,T' \subseteq \{1, \ldots, p\}$ of cardinality $|T| \leq s$ and $|T'| \leq s'$,  
\ben
\label{label:correlation-coefficient}
\frac{\abs{{\ip{X_T c, X_{T'} {c'}}}}}{n} 
\leq \theta_{s, s'} \twonorm{c} \twonorm{c'}
\een
holds, where $s + s' \leq p$. 
Note that $\theta_{s, s'}$ and $\delta_s$ are non-decreasing in $s, s'$ and 
small values of $\theta_{s, s'}$ indicate that disjoint subsets 
covariates in $X_T$ and $X_{T'}$ span nearly orthogonal subspaces 
(See Lemma~\ref{lemma:parallel} for a general bound on $\theta_{s,s'}$.)
For $\delta_s$,  it holds that
$1 - \delta_s \leq \Lambda_{\min}(s) \leq \Lambda_{\max}(s) \leq 1 + \delta_s$.
Hence $\delta_{2s} < 1$ implies that 
condition~\eqref{eq::eigen-admissible-s} holds.
As a consequence of these definitions, for any subset $I$, we have
\begin{eqnarray}
\label{eq::eigen-cond}
& & 
\Lambda_{\max}(|I|) 
\geq \Lambda_{\max}  \left({X_I^T X_I}/{n}\right) \geq 
\Lambda_{\min}  \left({X_I^T X_I}/{n}\right) \geq \Lambda_{\min}(|I|)
\end{eqnarray}
where $\Lambda_{\min}(|I|)  \geq \Lambda_{\min}(2s) > 0$ and 
$\Lambda_{\max}(|I|)  \leq \Lambda_{\max}(2s)$ for $|I| \leq 2s$.
We next introduce some conditions on the design, namely,
the Restricted Eigenvalue (RE) condition by~\cite{BRT09} and
the Uniform Uncertainly Principle by~\cite{CT07}
which we use throughout this paper.
\begin{assumption}
\textnormal{(\bf{Restricted Eigenvalue Condition $RE(s, k_0, X)$}
~\citep{BRT09})} 
\label{def:BRT-cond}
For some integer $1\leq s \leq p$ and a number $k_0 >0$, 
it holds for all $\upsilon \not=0$,
\beq
\label{eq::admissible}
\inv{K(s, k_0)} \stackrel{\triangle}{=}
\min_{\stackrel{J_0 \subseteq \{1, \ldots, p\},}{|J_0| \leq s}}
\min_{\norm{\upsilon_{J_0^c}}_1 \leq k_0 \norm{\upsilon_{J_0}}_1}
\; \;  \frac{\norm{X \upsilon}_2}{\sqrt{n}\norm{\upsilon_{J_0}}_2} > 0.
\eeq
\end{assumption}
\begin{assumption}
\textnormal{(\bf{A Uniform Uncertainly Principle})~\citep{CT07}} 
\label{def:CT-cond}
For some integer $1 \leq s < n/3$, assume $\delta_{2s} + \theta_{s, 2s} < 1$,
which implies that $\lambda_{\min}(2s) > \theta_{s, 2s}$ 
given that $1 - \delta_{2s} \leq \Lambda_{\min}(2s)$.
\end{assumption}
If $RE(s, k_0, X)$ is  satisfied with $k_0 \geq 1$,
then~\eqref{eq::eigen-admissible-s} must hold; 
Bounds on prediction loss and $\ell_p$ loss, 
where $1 \leq p \leq 2$, for estimating the parameters are derived for both 
the Lasso and the Dantzig selector in both linear and nonparametric 
regression models;  see~\cite{BRT09}.
We now define oracle inequalities in terms of $\ell_2$ loss as 
explored in~\cite{CT07}, where  they show such inequalities hold for
the Dantzig selector under the UUP (cf. Proposition~\ref{prop:DS-oracle}).
% the square submatrices of size $\leq 2s$ of $X^TX$ are 
%necessarily positive definite.
%We describe it here only briefly. 
\subsection{Oracle inequalities}
\label{sec:oracle-intro}
Consider the least squares estimators 
$\hat{\beta}_I = (X_I^T X_{I})^{-1} X_{I}^T Y$, where $|I| \leq s$.
Consider the {\it ideal} least-squares estimator $\beta^{\diamond}$
\begin{eqnarray}
\label{eq::beta-diamond}
& & \beta^{\diamond} = \argmin_{I \subseteq \{1, \ldots, p\},\; |I| \leq s} 
\E \twonorm{\beta- \hat{\beta}_I}^2
\end{eqnarray}
which minimizes the expected mean squared error.
It follows from~\cite{CT07} that for $\Lambda_{\max}(s) < \infty$ 
%cf. Section~\ref{section:append-diamond}),
\begin{eqnarray}
\label{eq::ME-diamond}
\E \twonorm{\beta- \beta^{\diamond}}^2 \geq 
\min \left(1, {1}/{\Lambda_{\max}(s)} \right)
\sum_{i=1}^p \min(\beta_i^2, \sigma^2/n).
\end{eqnarray}
%Hence it holds for $\delta_s <1$ that
%$\E \twonorm{\beta- \beta^{\diamond}}^2 \geq 1/2
%\sum_{i=1}^p \min(\beta_i^2, \sigma^2/n)$ given that 
%$\Lambda_{\max}(s) \leq 1 + \delta_s < 2$.
Now we check if for $\Lambda_{\max}(s) < \infty$, 
it holds with high probability that
\begin{eqnarray}
\label{eq::log-MSE}
\shtwonorm{\hat{\beta} - \beta}^2 & = &  
O(\log p) \sum_{i=1}^p \min(\beta_i^2, \sigma^2/n), \; \text{ so that }\\
\label{eq::log-MSE-E}
\shtwonorm{\hat{\beta} - \beta}^2 
& = & O(\log p) \max(1, \Lambda_{\max}(s))
\E\twonorm{\beta^{\diamond}- \beta}^2
\end{eqnarray}
holds in view of ~\eqref{eq::ME-diamond}.
These bounds are meaningful since
$$\sum_{i=1}^p \min(\beta_i^2, \sigma^2/n) = \min_{I \subseteq \{1, \ldots, p\}} 
\twonorm{\beta - \beta_{I}}^2 + \frac{|I| \sigma^2}{n}$$
represents the squared bias and variance.
%\leq \max(1, \Lambda_{\max}(s)) \E\twonorm{\beta^{\diamond}- \beta}^2,
%It is well known that thresholding rules with threshold level at about 
%$\sqrt{2 \log p} \sigma$ achieve the ideal MSE to within a multiplicative 
%factor proportional to $\log p$~\cite{Donoho:94}.
Define $s_0$ as the smallest integer such that 
\ben
\label{eq::define-s0}
\sum_{i=1}^p \min(\beta_i^2, \lambda^2 \sigma^2) \leq 
s_0 \lambda^2 \sigma^2, \text{ where } \; \lambda = \sqrt{ 2 \log p/n}.
\een
A consequence of this definition is: $|\beta_j| < \lambda \sigma$ for all 
$j > s_0$, if we order $|\beta_1| \geq |\beta_2| ... \geq |\beta_p|$
(cf.~\eqref{eq::beta-2-small}). 
We define a quantity $\basepen$ for each 
$a >0$, by which we bound the 
maximum correlation between the noise and covariates of $X$, which 
we only apply to $X$ with column $\ell_2$ norm bounded by $\sqrt{n}$; 
For each $a \geq 0$, let
\ben
\label{eq::low-noise}
&  &
\; \; \; \; {\T_a} := 
\biggl \{\e: \norm{{X^T \e}/{n}}_{\infty} \leq \basepen, 
\text{ where }
\basepen = \sigma \sqrt{1 + a} \sqrt{{2\log p}/{n}}\biggr \},
\een
we have (see~\cite{CT07}) $\prob{\T_a}  \geq 1 - (\sqrt{\pi \log p} p^a)^{-1}.$ 

The main theme of our paper is to explore oracle inequalities of 
the thresholding procedures under conditions as described above.
For the Lasso estimator and the Dantzig selector,
under the sparsity constraint, such oracle results have been 
obtained in a line of recent work for either the prediction error or the 
$\ell_p$ loss,  where $1 \leq p \leq 2$; see for example~\cite{BTW07a, BTW07b, BTW07c,CT07,CP09,BRT09,CWX09,Kol09,Kol07,vandeG08,ZH08,Zhang09,GB09,GBZ10} under  conditions stated above, or other variants.

Along  this line, we prove new results for both the Lasso as an initial 
estimator and for the thresholded estimators.
In Section~\ref{sec:exact-intro} and~\ref{sec:sparse-oracle},
we show oracle results for the Thresholded Lasso and the Gauss-Dantzig 
selector in terms of achieving the {\em sparse oracle inequalities} which we 
shall formally define in Section~\ref{sec:sparse-oracle}.
While the focus  of the present paper is on variable selection and oracle 
inequalities in terms of $\ell_2$ loss, prediction errors are also explicitly 
derived in Section~\ref{sec:intro-pred}; there we introduce the oracle 
inequalities in terms of prediction error and show a natural interpretation 
for the Thresholded Lasso estimator when relating to the $\ell_0$ penalized 
least squares estimators, in particular, ones that have been studied  
by~\cite{FG94}; see also~\cite{Barr:Birg:Mass:1999,BM97,BM01} for 
subsequent developments.  In Section~\ref{sec:type-II-intro}, 
we discuss recovery of a subset of strong signals.

\subsection{Variable selection under the RE condition} 
\label{sec:exact-intro}
Our first result in Theorem~\ref{THM:RE} shows 
that consistent variable selection is possible under the RE condition.
We do not impose any extra constraint on $s$ besides what is 
allowed in order for~\eqref{eq::admissible} to hold. 
Note that when 
$s > n/2$, it is impossible for the restricted eigenvalue assumption to 
hold as $X_{I}$ for any $I$ such that $|I| = 2s$ becomes singular in this case. 
Hence our algorithm is especially relevant if one would like 
to estimate a parameter $\beta$ such that $s$ is very close to $n$;
See Section~\ref{sec:linear-sparsity} for such examples.
Our analysis builds upon the rate of convergence bounds for $\beta_{\init}$
derived in~\cite{BRT09}. 
The first implication of this work and also one of the motivations for  analyzing 
the thresholding methods is: under Assumption~\ref{def:BRT-cond}, 
one can obtain consistent variable selection for very significant values of 
$s$, if only a few extra variables are allowed to be included 
in the estimator $\hat{\beta}$. Note that we did not optimize the lower bound 
on $s$ as we focus on cases when the support $S$ is large.
\begin{theorem}
\label{THM:RE}
Suppose that $RE(s, k_0, X)$ holds with $K(s, k_0)$, where $k_0 = 1$ for 
the Dantzig selector and $= 3$ for the Lasso.
Suppose $\lambda_n \geq f \basepen$ for $\basepen$ as 
in~\eqref{eq::low-noise}, where $f = 1$ for the Dantzig selector,
and $= 2$ for the Lasso. Let $s \geq K^4(s, k_0)$. 
Suppose $\beta_{\min} := \min_{j \in S} |\beta_{j}| \geq B_4 \lambda_n \sqrt{s},$
where $B_4 = 4 \sqrt{2} \max(K(s, k_0), 1) + \max\left(4 K^2(s, k_0), 
{\sqrt{2}}/{f\Lambda_{\min}(2s)} \right)$. 
Then on $\T_a$, the multi-step procedure returns $\hat{\beta}$ such that
for $B_3 = (1+a) (1+  1/(16 f^2\Lambda_{\min}^2(2s)))$,
\bens
& & 
S \subseteq I := \supp(\hat{\beta}), \; \text{ where } \;
|I \setminus S| <  1/(16 f^2\Lambda_{\min}^2(2s)) \; \; \; \text {and } \\
\label{eq::meaningful-bound}
& & 
\shtwonorm{\hat{\beta} - \beta}^2 \; \leq \;  
{\basepen^2 |I|}/{\Lambda_{\min}^2(|I|) } \; \leq \;
B_3 (2\log p/ n) s \sigma^2 /(\Lambda_{\min}^2(|I|)).
\eens
\end{theorem}
In Section~\ref{sec:experiments}, our simulation results using the 
Thresholded Lasso show that the exact recovery rate of the support
is very high for a few types of random matrices once the number of samples 
passes a certain threshold. 
We note that the oracle inequality as in~\eqref{eq::log-MSE}
is also achieved given that $\beta_{\min} \geq \sigma/\sqrt{n};\; \text{ hence } \; 
\sum_{i=1}^p \min(\beta_i^2, \sigma^2/n) = s \sigma^2/n.$
We next extend {\it model selection consistency}  beyond the notion of 
exact recovery of the support set $S$ as we introduced earlier, which has
been considered in~\cite{MB06,ZY07,Wai08}; Instead of having to make 
strong assumptions on either the signal strength, for example, on 
$\beta_{\min}$, or the incoherence conditions (or both), we focus on 
defining a meaningful criteria for  
{\it model selection consistency}  when both are relatively weak.

\subsection{Thresholding that achieves sparse oracle inequalities} 
\label{sec:sparse-oracle}
The natural question upon obtaining Theorem~\ref{THM:RE}
is: is there a good thresholding rule that enables us to obtain a
sufficiently {\it sparse} estimator $\hat{\beta}$ which 
satisfies the {\it oracle inequality} as in~\eqref{eq::log-MSE},
when some components of $\beta_S$  (and hence $\beta_{\min}$) are well
 below $\sigma/\sqrt{n}$?
Theorem~\ref{thm:ideal-MSE-prelude} answers this question positively: 
under a uniform uncertainty principle (UUP), thresholding of an initial 
Dantzig selector $\beta_{\init}$ at the level of $C_1 \sqrt{2 \log p/n} \sigma$ 
for some constant $C_1$, identifies a sparse model $I$  of cardinality at most 
$2 s_0$ such that its corresponding least-squares estimator $\hat{\beta}$ 
based on the model $I$ 
achieves the oracle inequality as in~\eqref{eq::log-MSE}.
This is accomplished without any knowledge of the 
significant coordinates or parameter values of $\beta$.
Theorem~\ref{thm:RE-oracle-main} shows that exactly the same type of 
sparse oracle inequalities hold for the Thresholded Lasso under the 
RE condition, which is both surprising but also mostly anticipated; 
this is also the key contribution of this paper. 
For simplicity, we always aim to bound 
$|I| < 2 s_0$ while achieving the oracle inequality 
as in~\eqref{eq::log-MSE}; 
One could aim to bound $|I| < c s_0$ for some other constant $c > 0$.
We refer to estimators that satisfy both constraints as estimators that 
achieve the {\it sparse oracle inequalities}.
Moreover, we note that thresholding of an initial estimator 
$\beta_{\init}$ which achieves $\ell_2$ loss as in~\eqref{eq::log-MSE} 
at the level of  $c_1 \sigma \sqrt{2 \log p/n}$ for some constant $c_1 > 0$, 
will always select nearly the best subset of variables in the spirit of 
Theorem~\ref{thm:ideal-MSE-prelude} and~\ref{thm:RE-oracle-main}; 
Formal statements of such results are omitted.
\begin{theorem}\textnormal{(\bf{Variable selection under UUP})}
\label{thm:ideal-MSE-prelude}
Choose $\tau, a > 0$ and set $\lambda_n = \lambda_{p, \tau} \sigma$,
where $\lambda_{p,\tau} := (\sqrt{1+a} + \tau^{-1}) \sqrt{2 \log p/n}$,
in~\eqref{eq::DS-func}.
Suppose $\beta$ is $s$-sparse with $\delta_{2s} + \theta_{s,2s} < 1 - \tau$.
Let threshold $t_0$ be chosen from the range 
$(C_1\lambda_{p,\tau} \sigma, C_4 \lambda_{p, \tau} \sigma]$ for some
constants $C_1, C_4$ to be defined.
Then with probability at least
$1 - (\sqrt{\pi \log p} p^a)^{-1}$, the 
Gauss-Dantzig selector $\hat{\beta}$ selects a model 
$I := \supp(\hat{\beta})$ such that
we have 
\ben
\label{eq::ideal-size}
|I| & \leq &  2s_0 \; \text{ and } \; |I \setminus S| \leq s_0 \leq s \;
\text{ and } \\ 
\label{eq::ideal-MSE}
\shtwonorm{\hat\beta - \beta}^2 & \leq & 
2 C_3^2 \log p
\left(\sigma^2/n + \sum_{i=1}^p \min(\beta_i^2, \sigma^2/n)\right)
\een
where $C_1$ is defined in~\eqref{eq::DS-constants-1} and 
$C_3$ depends on $a, \tau$, $\delta_{2s}$, $\theta_{s, 2s}$ and $C_4$; see~\eqref{eq::DS-constants-3}.
\end{theorem}
\begin{theorem}\textnormal{(\bf{Ideal model selection for the Thresholded Lasso})}
\label{thm:RE-oracle-main}
Suppose $RE(s_0, 6, X)$ holds with $K(s_0, 6)$, and
conditions~\eqref{eq::eigen-admissible-s} and~\eqref{eq::eigen-max} hold.
Let $\beta_{\init}$ be an optimal solution to~\eqref{eq::origin} with 
$\lambda_{n} = d_0 \sqrt{2 \log p/n} \sigma \geq 2 \basepen$, where 
 $a \geq 0$ and $d_0 \geq 2 \sqrt{1 + a}$.
Suppose that we choose $t_0 = C_4 \lambda \sigma$, for some 
constant $C_4 \geq D_1$, where $D_1= \Lambda_{\max}(s-s_0) + 9 K^2(s_0, 6)/2$; set $I =  \left\{j \in \{1, \ldots, p\}: \beta_{j, \init} \geq t_0 \right\}$. 
Then for $\drop := \{1, \ldots, p\} \setminus I$ and 
$\hat\beta_{I} = (X_I^T X_{I})^{-1} X_{I}^T Y$, we have on $\T_a$:
\bens
|I| & \leq  & s_0 (1 + D_1/C_4)   < 2 s_0, \; \; |I \cup S|   \leq  s + s_0 \; \text { and } \\
\shtwonorm{\hat\beta - \beta}^2 & \leq &  
2 D_3^2 \log p (\sigma^2/n + \sum_{i=1}^p \min(\beta_i^2, \sigma^2/n))
\eens
where $D_3$ depends on $a$,  
$K(s_0, 6)$, $D_0$ and $D_1$ as in~\eqref{eq::D0-define} 
and~\eqref{eq::D1-define}, $\Lambda_{\min}(|I|)$, $\theta_{s, 2s_0}$, 
and $C_4$; see~\eqref{eq::D3-constant}. 
%All constants are defined in Section~\ref{sec:oracle-lasso}. 
\end{theorem}
Our analysis for Theorem~\ref{thm:ideal-MSE-prelude} builds 
upon~\cite{CT07}, which show that so long as $\beta$ is sufficiently
sparse the Dantzig selector as in~\eqref{eq::DS-func} achieves 
the oracle inequality as in~\eqref{eq::log-MSE}.
Note that allowing $t_0$ to be  chosen from a range 
(as wide as one would like, with the cost of increasing 
the constant $C_3$ in~\eqref{eq::ideal-MSE}), saves us from having 
to estimate $C_1$, which indeed depends on $\delta_{2s}$ and 
$\theta_{s, 2s}$.
The same comment applies to Theorem~\ref{thm:RE-oracle-main} 
for $D_3$.
Assumption~\ref{def:CT-cond} implies that 
Assumption~\ref{def:BRT-cond} holds for $k_0 = 1$ with  $K(s, k_0) = 
\sqrt{\Lambda_{\min}(2s)}/(\Lambda_{\min}(2s) - \theta_{s, 2s}) \leq 
{\sqrt{\Lambda_{\min}(2s)}}/(1 - \delta_{2s} -  \theta_{s, 2s})$ 
(see~\cite{BRT09}). For a more comprehensive comparison
between these conditions, we refer to~\cite{GB09}.
We note that $RE(s_0, 6)$ is imposed on $X$ with {\it sparsity} 
fixed at $s_0$ (rather than $s$) and $k_0 = 6$ in
Theorem~\ref{thm:RE-oracle}. 
Important consequences of this result is shown in 
Section~\ref{sec:intro-pred}. 
The term {\it sparsity oracle inequalities} has also been used in the literature, 
which is targeted at bounding prediction errors of the estimators with the 
best sparse approximation of the regression function known by an oracle; 
see~\cite{BRT09} and more references therein. It would be interesting
to explore such properties for the Thresholded Lasso under the RE 
conditions.

\subsection{Connecting to the $\ell_0$ penalized least squares estimators}
\label{sec:intro-pred}
Now why is the bound of $|I| \le 2 s_0$ interesting? We wish to point out that 
this would make the behavior of the Thresholded Lasso procedure somehow 
mimic that of the $\ell_0$ penalized estimators, which is computational 
inefficient, as we introduce next. 
It is clear that for the least squares estimator based on $I$, 
$\hat\beta_{I} = (X_I^T X_{I})^{-1} X_{I}^T Y$, it holds that 
\ben
%\label{eq::projected-error}
&  &\shtwonorm{X \hat{\beta}_I - X \beta}^2
 = \twonorm{P_I (X \beta + \e) - X \beta}^2
 =  \twonorm{(P_I - \Id) X_{I^c} \beta_{I^c} + P_I \e}^2 \\
\label{eq::MSE-pred}
& & \text{ and hence } \; \expct{\shtwonorm{X \hat{\beta}_I - X \beta}}^2/n
 =  \twonorm{(P_I - \Id) X_{I^c} \beta_{I^c}}^2 +  |I| \sigma^2,
\een
which again shows the typical bias and variance tradeoff. 
Consider the best model $I_0$ upon which 
$\hat{\beta}_{I_0} =  (X_{I_0}^T X_{I_0})^{-1} X_{I_0}^T Y$
achieves the minimum in~\eqref{eq::MSE-pred}:
\bens
\label{eq::MSE-pred-best}
I_0 = \argmin_{I \subset \{1, \ldots, p\}} 
\twonorm{(P_I - \Id) X_{I^c} \beta_{I^c}} +  |I| \sigma^2.
\eens
Now the question is: can one do nearly as well as $\hat{\beta}_{I_0}$ 
in the sense of achieving mean square error within $\log p$ factor 
of $\expct{\shtwonorm{X \hat{\beta}_{I_0} - X \beta}}^2$?
It turns out that the answer is yes, if one 
solves the following $\ell_0$ penalized least squares estimator with 
$\lambda_0 = \sqrt{\log p/n}$, as proposed in the RIC
procedure~\citep{FG94}:
\begin{eqnarray}
\label{eq::ell0} \; \; 
\hat \beta = \arg\min_{\beta} \shtwonorm{Y-X\beta}^2 /(2n) + 
\lambda_0^2 \sigma^2  \|\beta\|_0,
\end{eqnarray}
where $\|\beta\|_0$ is the number of nonzero components in $\beta$.
This is shown in a series of papers 
in~\cite{FG94,Barr:Birg:Mass:1999,BM97,BM01}. We refer 
to~\cite{FG94,Barr:Birg:Mass:1999} for other procedures related 
to~\eqref{eq::ell0}. Note that $\shtwonorm{Y-X\beta}^2 \leq 
2 \shtwonorm{ X \hat\beta  - X \beta}^2 + 2 \shtwonorm{\e}^2$; 
hence we only need to look at the tradeoff between
$\shtwonorm{ X \hat\beta  - X \beta}^2$ and $\log p |I|$.
Note that $\shtwonorm{ X \hat\beta  - X \beta}^2$ would be $0$ 
if $\hat{\beta} = \beta$, but  $|I|$ would be large.
Theorem~\ref{THM:PREDITIVE-ERROR} shows that 
$(a)$ the thresholded estimators achieve a balance between 
the ``complexity'' measure $\log p |I|$ and 
$\shtwonorm{ X \hat\beta  - X \beta}^2$ which now have the 
same order of magnitude; $(b)$ and in some sense, variables in model $I$ 
are essential in predicting $X \beta$.
\begin{theorem}
\label{THM:PREDITIVE-ERROR}
Let $I$ be the model selected by thresholding an initial estimator
$\beta_{\init}$, under conditions as described in 
Theorem~\ref{thm:ideal-MSE-prelude} or Theorem~\ref{thm:RE-oracle-main}.
Let $\drop := \{1, \ldots, p\} \setminus I$.
Let $s_0$ be as defined in~\eqref{eq::define-s0} and 
$\lambda = \sqrt{2 \log p/n}$.
For $\hat\beta_{I} = (X_I^T X_{I})^{-1} X_{I}^T Y$ and some constant $C$, 
we have on $\T_a$, 
\bens
\frac{\twonorm{X \hat{\beta}_I - X \beta}}{\sqrt{n}} \leq
\sqrt{\Lambda_{\max}(s)} \twonorm{\beta_{\drop}} + 
\frac{\sqrt{|I| \Lambda_{\max}(|I|)} \basepen}
{\Lambda_{\min}(|I|)} \leq C \lambda \sigma \sqrt{s_0}.
\eens
\end{theorem}
Comparing~\eqref{eq::ell0}  and~\eqref{eq::origin}, it is clear that 
for entries $\beta_{j, \init} < \lambda_0 \sigma$ in a Lasso estimator,  their 
contributions to the optimization function in~\eqref{eq::ell0} will be larger 
than that  in~\eqref{eq::origin} if $\lambda_n = \lambda_0 \sigma$; 
hence removing these entries from the initial estimator in some sense 
recovers the choices that would have been made by the complexity-based 
function as in~\eqref{eq::ell0}. Put in another way, getting rid of variables 
$\{j : \beta_{j, \init} < \lambda_0 \sigma \}$ from the solution 
to~\eqref{eq::origin} with $\lambda_n \asymp \lambda_0 \sigma$ is
in some way restoring the behavior of~\eqref{eq::ell0} in a brute-force
manner. 
Proposition~\ref{PROP:COUNTING-S0} (by setting $c' = 1$) shows that
the number of variables in $\beta$ at above and around 
$\sqrt{\log p/n} \sigma$ in magnitude is bounded by $2s_0$
(One could choose another target set: for example,
$\{j: |\beta_j| \geq  \sqrt{\log p/(c' n)} \sigma\}$, for some $c' > 1/2$.)
Roughly speaking, we wish to include most of them by leaving $2s_0$ 
variables in the model $I$. 
Such connections will be made precise in our future work.
\subsection{Controlling Type II errors}
\label{sec:type-II-intro}
In Section~\ref{SEC:GENERAL-RE} 
(cf. Theorem~\ref{thm:threshold-general}), 
we show that we can recover a subset $S_L$ of variables accurately,
where $S_L: = \{j: |\beta_j| >  \sqrt{2 \log p/n} \sigma\}$, under 
Assumption~\ref{def:BRT-cond}
when $\beta_{\min,S_L} := \min_{j \in S_L} |\beta_{j}|$ is large enough 
(relative to the $\ell_2$ loss of an initial estimator under the RE condition
on the set $S_L$); in addition, a small number of extra variables from 
$\{1, \ldots, p \} \setminus T_0 =: T_0^c$ 
are possibly also included in the model $I$, where $T_0$ 
denotes positions of the $s_0$ largest coefficients of $\beta$ in 
absolute values.
In this case, it is also possible to get rid of variables from $T_0^c$ entirely
by increasing the threshold $t_0$ while making the lower bound on 
$\beta_{\min,S_L}$ a constant times stronger. We omit such details from the 
paper. Hence compared to Theorem~\ref{THM:RE}, we have relaxed the 
restriction on 
$\beta_{\min}$: rather than requiring all non-zero entries to be large, 
we only require those in a subset $S_L$ to be recovered to be large.
In addition, we believe that our analysis can be extended to cases
when $\beta$ is not exactly sparse, but has entries decaying like a power 
law, for example, as studied by~\cite{CT07}; 
%this will be addressed in our future work. 
We end with Proposition~\ref{PROP:COUNTING-S0}. For a set 
$A$, we use $|A|$ to denote its cardinality.
\begin{proposition}
\label{PROP:COUNTING-S0}
Let $T_0$ denote 
positions of the $s_0$ largest coefficients of $\beta$ in absolute values.
where $s_0$ is defined in~\eqref{eq::define-s0}.
Let $a_0 = \size{S_L}$  (cf.~\eqref{eq::define-a0}). Then $\forall c' > 1/2$, \; \;
we have
$\size{ \{j \in T_0^c: |\beta_j| \geq \sqrt{\log p/(c' n)} \sigma \}}
\leq  (2c' -1) (s_0  - a_0)$.
\end{proposition}

\subsection{Previous work} 
We briefly review related work in multi-step procedures and the role 
of sparsity for high-dimensional statistical inference.
Before this work, hard thresholding idea has been shown in~\cite{CT07}
(via Gauss-Dantzig selector) as a method to correct the bias of 
the initial Dantzig selector. The empirical success of the Gauss-Dantzig 
selector in terms of improving the statistical accuracy is strongly 
evident in their experimental results. Our theoretical analysis
on the oracle inequalities, which hold for the Gauss-Dantzig selector 
under a uniform uncertainty principle, builds upon their theoretical 
analysis of the initial Dantzig selector under the same condition.
For the Lasso,~\cite{MY09} has also shown in theoretical analysis that 
thresholding is effective in obtaining a two-step estimator $\hat\beta$ 
that is consistent in its support with $\beta$ when $\beta_{\min}$ is 
sufficiently large; As pointed out by~\cite{BRT09}, a weakening of 
their condition is still sufficient for Assumption~\ref{def:BRT-cond} to hold. 

The sparse recovery problem under arbitrary noise is also well studied, 
see~\cite{CRT06,NV09,NT08}.
Although as argued in~\cite{CRT06} and \cite{NT08}, the best accuracy under 
arbitrary noise has essentially been achieved in both work, their bounds 
are worse than that in~\cite{CT07} (hence the present paper) under 
the stochastic noise as discussed in the present paper;
% see more discussions in~\cite{CT07}.  
Moreover, greedy algorithms in~\cite{NV09,NT08} require $s$ to be part of 
the input, while algorithms in the present paper do not have such a 
requirement, and hence adapt to the unknown level of sparsity well. 
A more general framework on multi-step variable selection was 
studied by~\cite{WR08}. They control the probability of false positives at 
the price of false negatives, similar to what we aim for here; their analysis is 
constrained to the case when $s$ is a constant. 
Recently, another two-stage procedure that is also relevant has been 
proposed in~\cite{Zhang09}, where in the second stage 
``selective penalization'' is being applied to the set of {\it irrelevant features}  
which are defined as those below a certain threshold in the initial Lasso 
estimator; Incoherence conditions there are sufficiently different 
from the RE condition as we study in this paper for the Thresholded Lasso.
Under conditions similar to Theorem~\ref{THM:RE},~\cite{ZGB09} requires 
$s = O(\sqrt{n/\log p})$ in order to achieve variable selection consistency 
using the adaptive Lasso~\citep{Zou06} (see also~\cite{HMZ08}), 
as the second step procedure. 
Concurrent with the present work,  the authors have revisited the 
adaptive Lasso and derived bounds in terms of prediction 
error~\cite{GBZ10}; there the number of false positives is also 
aimed at being in the same order as that of the set of {\it significant} 
variables which predicts $X \beta$ well; in addition, the adaptive Lasso 
method is compared with thresholding methods, under a  stronger 
incoherence condition than the RE condition studied in the present paper.
While the focus  of the present paper 
is on variable selection and oracle inequalities for the $\ell_2$ loss,  
prediction errors of the OLS estimators $\hat\beta$ are also explicitly 
derived; We also compare the performance in 
terms of variable selections  between the adaptive and the thresholding
methods in our simulation study, which is reported in 
Section~\ref{sec:experiments}.

Parts of this work was presented in a 
conference paper~\cite{Zhou09a}. The current version expands the original 
idea and elaborates upon the conceptual connections between 
the Thresholded Lasso and $\ell_0$ penalized methods;
in addition, we provide new results on the sparse oracle inequalities under 
the RE condition (cf. Theorem~\ref{thm:RE-oracle-main},
Theorem~\ref{thm:RE-oracle} and Theorem~\ref{thm:threshold-general}).

\subsection{Organization of the paper}
Section~\ref{sec:linear-sparsity} briefly discusses the relationship between 
linear sparsity and random design matrices, while highlighting the role
thresholding plays in terms of recovering the best subset of variables,
when $s$ is a linear fraction of $n$, which in turn is 
a nonnegligible fraction of $p$.
We prove Theorem~\ref{THM:RE} essentially in Section~\ref{eq::large-beta}.
A thresholding framework for the general setting is described in
Section~\ref{sec:DS-threshold}, which also sketches the proof of 
Theorem~\ref{thm:ideal-MSE-prelude}.
The proof of Theorem~\ref{thm:RE-oracle-main} is shown in 
Section~\ref{sec:oracle-lasso}, where oracle inequalities for the original 
Lasso estimator is also shown.
In Section~\ref{SEC:GENERAL-RE}, we show conditions under which
one recovers a subset of strong signals.
Section~\ref{sec:experiments} includes simulation results showing 
that the Thresholded Lasso is consistent with our theoretical analysis on 
variable selection and on estimating $\beta$. 
Most of the technical proofs are included in the Appendix.

\section{Linear sparsity and random matrices}
\label{sec:linear-sparsity} 
A special case of design matrices that satisfy the 
Restricted Eigenvalue assumption are the random design matrices. 
This is shown in a large body of work, 
for example~\cite{Szarek91,CRT06,CT05,Donoho06,CT07,BDDW08,MPT08},
which shows that the UUP holds for ``generic'' or random design matrices 
for very significant values of $s$. It is well known that for a random matrix 
the UUP holds for $s \asymp n/\log(p/n)$ with i.i.d. Gaussian random 
variables, subject to normalizations  of columns, 
% (that is, Gaussian random ensemble, 
the Bernoulli, and in general the subgaussian random 
ensembles~\cite{BDDW08,MPT08};~\cite{ALPT09} show that UUP 
holds for $s \asymp n/\log^2(p/n)$ when $X$ is a random matrix 
composed of columns that are independent isotropic vectors with 
log-concave densities. Hence this setup only requires $C s$ observations 
per nonzero value in $\beta$,  where $C$ is a small constant, when $n$ is 
a nonnegligible fraction of  $p$, in order to recover $\beta$; we 
call this level of sparsity the linear sparsity.
Our simulation results in Section~\ref{sec:experiments} show that 
once $n \geq C s \log (p /n)$, where $C$ is a small constant, exact recovery 
rate of the sparsity pattern is very high for Gaussian (and Bernoulli) 
random ensembles, when $\beta_{\min}$ is sufficiently large;
this shows a strong contrast with the ordinary Lasso, for which the 
probability of success in terms of exact recovery of the sparsity pattern 
tends to zero when $n < 2 s \log (p -s)$~\citep{Wai08}. 

A series of recent papers~\cite{ZGB09,RWY09,Zhou09c} show that a 
broader class of subgaussian random matrices also satisfy the 
Restricted Eigenvalue condition; In particular,~\cite{Zhou09c} shows that for 
subgaussian random matrices $\Psi$ which are now well known to satisfy the 
UUP condition under linear sparsity, RE condition holds for 
$X := \Psi \Sigma^{1/2}$ with overwhelming probability with 
$n \asymp s \log (p /n)$ number of samples, where $\Sigma$ is assumed to 
satisfy the follow condition:
Suppose $\Sigma_{jj} = 1, \forall j = 1, \ldots, p$, and for some integer 
$1\leq s \leq p$ and a positive number $k_0$, the following condition holds
for all $\upsilon \not= 0$:
\bens
\label{eq::admissible-random}
\inv{K(s, k_0, \Sigma)} := \min_{\stackrel{J_0 \subseteq \{1, \ldots,
    p\},}{|J_0| \leq s}} 
\min_{{\norm{\upsilon_{J_0^c}}_1 \leq k_0 \norm{\upsilon_{J_0}}_1}}
\; \;  {\norm{\Sigma^{1/2} \upsilon}_2}/{\norm{\upsilon_{J_{0}}}_2} > 0.
\eens
Thus the additional covariance structure $\Sigma$ is explicitly
introduced to the columns of $\Psi$ in generating $X$. 
%In particular, one can apply thresholding procedures analyzed in this 
%paper to such random matrices 
%to recover a nearly ideal sparse model using nearly a constant number of 
%measurements per non-zero component despite the presence of 
%stochastic noise, when $n$ is a nonnegligible fraction of $p$. 
We believe similar results can be extended to other cases: for example, 
when $X$ is the composition of a random Fourier ensemble, or randomly 
sampled rows of orthonormal matrices, see for 
example~\cite{CT06,CT07,RV06}, where the UUP holds for 
$s = O(n/\log^c p)$ for some constant $c > 0$.

\section{Thresholding procedure when $\beta_{\min}$ is large}
\label{eq::large-beta}
In this section,
we use a penalization parameter $\lambda_n \geq B \basepen$ and assume 
$\beta_{\min} > C \lambda_n \sqrt{s}$ for some constants $B, C$; 
we first specify the thresholding parameters in this case.
%The thresholding parameter at stage-$i$ is defined as:  
%$t_i := 4 \lambda_n \sqrt{|\hat{S}_{i}|}, \forall i = 0, 1$, 
%where $\hat{S}_{i}$ are to be  defined. 
%Section~\ref{sec:proof-thm-RE}.
We then show in Theorem~\ref{THM:GENERIC} that our algorithm works 
under any condition so long as the rate of convergence of the initial 
estimator obeys the bounds in~\eqref{eq::2-normS-generic}. 
%and~\eqref{eq::1-norm-generic}. 
Theorem~\ref{THM:RE} is a corollary of Theorem~\ref{THM:GENERIC} 
under Assumption~\ref{def:BRT-cond}, given the rate of convergence 
bounds for $\beta_{\init}$ following derivations in~\citep{BRT09}.

{{\bf The Iterative Procedure.}}
We obtain an initial estimator $\beta_{\init}$ using the Lasso
or the Dantzig selector.
Let $\hat{S}_0 =  \left\{j: \beta_{j, \init} > 4 \lambda_{n} \right\}$, 
and $\hat\beta^{(0)} := \beta_{\init}$; Iterate through the following 
steps twice, for  $i = 0, 1$: 
(a) Set $t_i = 4 \lambda_n \sqrt{|\hat{S}_i|}$;
(b) Threshold $\hat \beta^{(i)}$ with $t_i$ to obtain
$I := \hat{S}_{i+1}$, where 
\ben
\label{eq::final-estimator}
\hat{S}_{i+1} = \left\{j \in \hat{S}_i:
\hat{\beta}^{(i)}_j \geq 4 \lambda_{n} \sqrt{|\hat{S}_{i}|} \right\}
\een
and compute
$\hat\beta^{(i+1)}_I = (X_I^T X_I)^{-1} X_I^T Y$.
Return the final set of variables in $\hat{S}_2$ and output 
$\hat{\beta}$ such that 
$\hat{\beta}_{\hat{S}_2} = \hat{\beta}^{(2)}_{\hat{S}_2}$ and   
$\hat{\beta}_{j} = 0, \forall j \in \hat{S}_2^c$.
\begin{theorem}
\label{THM:GENERIC}
Let $\lambda_n \geq B \basepen$, where $B \geq 1$ is a constant suitably 
chosen such that the initial estimator $\beta_{\init}$ satisfies on 
some event $Q_b$, for $\upsilon_{\init} = \beta_{\init} - \beta$,
\begin{eqnarray}
\label{eq::2-normS-generic}
\norm{\upsilon_{\init,S}}_{2} & \leq & B_0 \lambda_n \sqrt{s} 
\; \text{ and } \;
%\norm{\upsilon_{\init,\Sc}}_1 \; = \;  
\norm{\beta_{\init,\Sc}}_1 \leq B_1 \lambda_n s
\end{eqnarray}
where $B_0, B_1$ are some constants. Suppose for $B_2 = 1/(B \Lambda_{\min}(2s))$,
\ben
\label{eq::beta-min-specific}
\beta_{\min} & \geq & \left(\max\left(\sqrt{B_1}, 2 \right) 2 \sqrt{2}  + 
\max\left(B_0, \sqrt{2} B_2 \right) \right) \lambda_n \sqrt{s}.
\een
Then for $s \geq {B_1^2}/{16}$, it holds on $\T_a \cap Q_b$ that, 
(a): $\forall i = 1, 2, \; |\hat{S}_i | \leq 2s$; and (b):
\ben
\label{eq::2-norm-generic}
\shtwonorm{\hat{\beta}^{(i)} - \beta} & \leq & 
\basepen \sqrt{|\hat{S}_i|} /{\Lambda_{\min}(|\hat{S}_i|)}
\; \leq \; \lambda_n B_2 \sqrt{2s}
\een
where $\forall i = 1, 2$, 
$\hat{\beta}^{(i)}$ are the OLS estimators based on $\hat{S}_i$;
Moreover, the Iterative Procedure includes the set of relevant variables 
in $\hat{S}_2$ such that $S  \subseteq \hat{S}_2 \subseteq \hat{S}_1$ and
\ben
\label{eq::extra-variables}
\size{\hat{S}_2 \setminus S} := \size{\supp(\hat{\beta}) \setminus S} \leq
1/(16 B^2 \Lambda^2_{\min}(|\hat{S}_1|)) \leq B_2^2/16.
\een
\end{theorem}
The proof of Theorem~\ref{THM:GENERIC} appears in 
Section~\ref{appendix:thm:generic}. 
We now discuss its relationship to theorems in the subsequent sections.
We first note that in order to obtain $\hat{S}_1$ such that 
$|\hat{S}_1| \leq 2s$ and $\hat{S}_1 \supseteq S$ as above, we only 
need to threshold $\beta_{\init}$ at $t_0 = B_1 \lambda_n$;
here instead of having to estimate the unknown $B_1$, 
we can use $t_0 = c_0 \lambda_n \sqrt{s}$ for some constant $c_0$ to 
threshold $\beta_{\init}$. In the general setting, we require 
that $t_0$ be chosen from the range $(C_1 \lambda_n, C_4 \lambda_n]$ for
some constants $C_1, C_4 $ to be specified; see 
Section~\ref{sec:DS-threshold} (Lemma~\ref{lemma:threshold-DS}) 
for example.
We note that without the knowledge of $\sigma$, one could use 
$\hat{\sigma} \geq \sigma$ 
in $\lambda_n$; this will put a stronger requirement on $\beta_{\min}$, 
but all conclusions of Theorem~\ref{THM:GENERIC} hold. 
When $\beta_{\min}$ does not satisfy the constraint as in 
Theorem~\ref{THM:GENERIC}, we cannot really guarantee that all
variables in $S$ will be chosen. Hence~\eqref{eq::2-normS-generic}
will be replaced by requirements on $T_0$,
which denotes
locations of the $s_0$ largest coefficients of $\beta$ in absolute values:
ideally, we wish to have 
\begin{eqnarray}
\label{eq::1-norm-general}
\norm{(\beta_{\init} - \beta)_{T_0}}_{2} \leq C_0 \lambda_n \sqrt{|T_0|} 
\; \text{ and }  \; \norm{\beta_{\init,T_0^c}}_1 \leq C_1 \lambda_n |T_0|;
\end{eqnarray}
for some constants $C_0, C_1$,
so that~\eqref{eq::ideal-size} and~\eqref{eq::ideal-MSE} hold
under suitably chosen thresholding rules.
This is the content of Theorem~\ref{thm:RE-oracle} and 
Theorem~\ref{thm:threshold-general}.

\section{Nearly ideal model selections under the UUP}
\label{sec:DS-threshold}
In this section, we wish to derive a meaningful criteria for consistency in 
variable selection, when $\beta_{\min}$ is well below the noise level.
Suppose that we are given an initial estimator $\beta_{\init}$ that achieves 
the oracle inequality as in~\eqref{eq::log-MSE},
which adapts nearly ideally not only to the uncertainty in the support set 
$S$ but also the ``significant'' set.
We show that although we cannot guarantee the presence of 
variables indexed by
$S_R = \{j: |\beta_j| \leq \sigma \sqrt{2 \log p/n} \}$ to be 
included in the final set $I$ (cf.~\eqref{eq::beta-2-small}) due to their 
lack of strength, we wish to include in $I$ most variables in
$S_L = S \setminus S_R$ such that the OLS estimator based on $I$ 
achieves~\eqref{eq::log-MSE}
%an almost ideal rate of convergence as $\beta_{\init}$ does,
even though some non-zero variables are missing from $I$. 
Here we pay a price for the missing variables in order to 
obtain a sufficiently sparse model $I$. 
%when some entries in $S$ are well below $\sqrt{2 \log p/n} \sigma$.
Toward this goal, we analyze the following algorithm.
% under Assumption~\ref{def:CT-cond}.  

{{\bf The General Two-step Procedure}}: Assume $\delta_{2s} + \theta_{s,2s} < 1 - \tau$, 
where $\tau > 0$;
\begin{enumerate}
\item
\label{step::thresh}
First obtain an initial estimator $\beta_{\init}$ using the Dantzig 
selector in~\eqref{eq::DS-func} with 
$\lambda_n = (\sqrt{1+a} + \tau^{-1}) \sqrt{2 \log p/n} \sigma$, where 
$a \geq 0$; then threshold $\beta_{\init}$ with $t_0$, chosen
%as specified as in Proposition~\eqref{prop:DS-oracle};
from the range 
$(C_1 \lambda_{p, \tau} \sigma, C_4 \lambda_{p, \tau} \sigma]$, 
for $C_1$ as defined in~\eqref{eq::DS-constants-1},
to obtain a set $I$ of cardinality at most $2s_0$ 
(cf. Lemma~\ref{lemma:threshold-DS}): \\
set $I :=  \left\{j \in \{1, \ldots, p\}: \beta_{j, \init} \geq t_0  \right\}.$
\item
Given a set $I$ as above, run the OLS regression to 
obtain $\hat\beta_{I} = (X_I^T X_{I})^{-1} X_{I}^T Y$ and set
$\hat\beta_j = 0, \forall j \not\in I.$
\end{enumerate}
In Section~\ref{sec:oracle-lasso}, we analyze the Thresholded Lasso,
where we obtain $\beta_{\init}$ via the Lasso under the RE condition
and follow the same steps as above; see Theorem~\ref{thm:RE-oracle} and 
Lemma~\ref{lemma:threshold-RE} for the new $\lambda_n$ and $t_0$  to be 
specified.
Under the UUP,~\cite{CT07} have shown that the Dantzig selector 
achieves nearly the ideal level of $\ell_2$ loss. 
We then show in Lemma~\ref{lemma:threshold-DS} that thresholding at the 
level of $C_1 \lambda \sigma$  at Step~\ref{step::thresh} selects a set 
$I$ of at most $2s_0$ variables, among which at most $s_0$ are from $\Sc$. 
\begin{proposition}
\textnormal{\citep{CT07}}
\label{prop:DS-oracle}
Let $Y = X \beta + \e$, for $\e$ being i.i.d. $N(0, \sigma^2)$ and
$\twonorm{X_j}^2 = n$.
Choose $\tau, a > 0$ and set $\lambda_n = (\sqrt{1+a} + \tau^{-1}) \sigma \sqrt{2 \log p/n}$ in~\eqref{eq::DS-func}. 
Then if $\beta$ is $s$-sparse with $\delta_{2s} + \theta_{s,2s} < 1 - \tau$, 
the Dantzig selector obeys
% on $\T_a$,
with probability at least $1 - (\sqrt{\pi \log p} p^a)^{-1}$,
%$\prob{\T_a} \geq 
$\twonorm{\hat\beta - \beta}^2  \leq  
2 C_2^2 (\sqrt{1+a} + \tau^{-1})^2 \log p
\left({\sigma^2}/{n} + 
\sum_{i=1}^p \min\left(\beta_i^2, {\sigma^2}/{n}\right)\right).$
\silent{large probability; for concreteness, if one chooses 
$\lambda_{p,\tau} := (\sqrt{1+a} + \tau^{-1}) \sqrt{2 \log p/n}$ for each 
$a \geq 0$, the bounds hold with} 
\end{proposition}
From this point on we let 
$\delta := \delta_{2s}$ and $\theta := \theta_{s, 2s}$;
Analysis in~\cite{CT07} (Theorem 2) and the current paper yields the 
following constants,
\begin{eqnarray}
\label{eq::DS-constants}
C_2 & = & 2 C_0' + \frac{1 + \delta}{1 - \delta - \theta} 
\; \text { where } 
C'_0 = \frac{C_0}{1 - \delta - \theta} + 
\frac{\theta(1 + \delta)}{(1 - \delta - \theta)^2}, 
\end{eqnarray}
where 
$C_0 =  2 \sqrt{2}
\left(1 + \frac{1 - \delta^2}{1 - \delta - \theta}\right) 
+ (1 + 1/\sqrt{2})\frac{(1 + \delta)^2}{1 - \delta - \theta}$; 
We now define
\begin{eqnarray}
\label{eq::DS-constants-1}
C_1 & = & C_0' + \frac{1+ \delta}{1-\delta-\theta}  \text{ and } \\
\label{eq::DS-constants-3}
C_3^2 & = & 3 (\sqrt{1+a} + \tau^{-1})^2 ((C_0' + C_4)^2 +1) 
%\left(1 + \frac{2 \theta^2_{s, 2s_0}}{\Lambda_{\min}^2(2s_0)}\right) 
+ {4(1+a)}/{\Lambda_{\min}^2(2s_0)}
\end{eqnarray}
where $C_3$ has not been optimized.
Recall that $s_0$ is the smallest integer such that
$\sum_{i=1}^p 
\min(\beta_i^2, \lambda^2 \sigma^2) \leq s_0 \lambda^2 \sigma^2,$
where $\lambda = \sqrt{2 \log p/n}$.
We order the $\beta_j$'s in decreasing order of magnitude
\ben
\label{eq::beta-order}
|\beta_1| \geq |\beta_2| ... \geq |\beta_p|.
\een
Thus by definition of $s_0$, the fact $0 \leq s_0 \leq s$, we have 
for $s < p$, 
\ben
\label{eq::s0-upper-bound}
& & \; \; \; 
s_0 \lambda^2 \sigma^2 \leq
 \lambda^2 \sigma^2 + \sum_{i=1}^p \min(\beta_i^2, \lambda^2 \sigma^2) 
% =  2 \log p \left(\frac{\sigma^2}{n} + 
%\sum_{i=1}^p \min\left(\frac{\beta_i^2}{2 \log p}, \frac{\sigma^2}{n}\right)
%\right) \\
 \leq
2 \log p \left(\frac{\sigma^2}{n} + \sum_{i=1}^p 
\min\left(\beta_i^2, \frac{\sigma^2}{n}\right)\right) \\
\label{eq::s0-lower-bound}
& & \; \;  \; s_0 \lambda^2 \sigma^2  \geq \sum_{j=1}^{s_0 + 1} 
\min(\beta_j^2, \lambda^2 \sigma^2)
\geq (s_0 + 1) \min(\beta_{s_0 +1}^2, \lambda^2 \sigma^2)
\een
which implies that (as shown in~\cite{CT07}) that
$\min(\beta_{s_0 +1}^2, \lambda^2 \sigma^2) < \lambda^2 \sigma^2$ and 
hence by~\eqref{eq::beta-order}, it holds that
\ben
\label{eq::beta-2-small}
|\beta_j| < \lambda \sigma\; \; \; \text{ for all } j > s_0.
\een
\begin{lemma}
\label{lemma:threshold-DS}
Choose $\tau > 0$ such that $\delta_{2s} + \theta_{s,2s} < 1 - \tau$.
Let $\beta_{\init}$ be the solution to~\eqref{eq::DS-func} with 
$\lambda_n =  \lambda_{p,\tau} \sigma 
:= (\sqrt{1+a} + \tau^{-1}) \sqrt{2 \log p/n} \sigma$.
Given some constant $C_4 \geq C_1$, for $C_1$ as 
in~\eqref{eq::DS-constants-1}, choose a thresholding parameter $t_0$ 
such that 
$C_4 \lambda_{p,\tau} \sigma \geq t_0  > C_1 \lambda_{p,\tau} \sigma$
and set $I = \{j: \size{\beta_{j, \init}} \geq t_0\}.$
Then with probability at least $\prob{\T_a}$, 
as detailed in Proposition~\ref{prop:DS-oracle}, we have
~\eqref{eq::ideal-size}, and for $C_0'$ as in~\eqref{eq::DS-constants},
$\twonorm{\beta_{\drop}} \leq \sqrt{(C_0' + C_4)^2 + 1} 
\lambda_{p,\tau} \sigma \sqrt{s_0}$, where
$\drop := \{1, \ldots, p\} \setminus I.$
\end{lemma}
It is clear by Lemma~\ref{lemma:threshold-DS} that we cannot 
cut too many ``significant'' variables; in particular, for those that
are $> \lambda \sigma \sqrt{s_0}$,  we can cut at most a constant number of 
them. Next we show that even if we miss some 
columns of $X$ in $S$, we can still hope to get the $\ell_2$ loss
as required in Theorem~\ref{thm:ideal-MSE-prelude} so long as 
$\twonorm{\beta_{\drop}}$ 
is bounded, for example, as bounded in Lemma~\ref{lemma:threshold-DS},
and $I$ is sufficiently sparse.
Now Theorem~\ref{thm:ideal-MSE-prelude} is an immediate 
corollary of Lemma~\ref{lemma:threshold-DS} and~\ref{prop:MSE-missing} 
in view of~\eqref{eq::s0-upper-bound}. See Section~\ref{sec:append-guass}
for its proof.
We note that Lemma~\ref{prop:MSE-missing} yields a general result on 
the $\ell_2$ loss for the OLS estimator, when a subset of relevant variables 
is missing from the chosen model $I$; this is also an important technical 
contribution of this paper. 
\begin{lemma}{\textnormal{({\bf OLS estimator with missing variables})}}
\label{prop:MSE-missing}
Suppose that~\eqref{eq::eigen-admissible-s} and~\eqref{eq::eigen-max}
hold. Let $\drop := \{1, \ldots, p\} \setminus I$ and $S_\drop = \drop \cap S$ 
such that $I \cap S_{\drop} = \emptyset$.
Suppose $|I \cup S_{\drop}| \leq 2s$. Then, 
%for the least squares estimator based on $I$, 
for $\hat\beta_{I} = (X_I^T X_{I})^{-1} X_{I}^T Y$, it holds on $\T_a$ that 
$$
\twonorm{\hat{\beta}_I - \beta}^2 \leq 
{\left(\theta_{|I|, |\dropS|} \twonorm{\beta_{\drop}} + 
\basepen \sqrt{|I|}\right)^2}/{\Lambda_{\min}^2(|I|)} + 
\twonorm{\beta_{\drop}}^2.$$
\end{lemma}
We note that Lemma~\ref{prop:MSE-missing} applies to $X$
so long as conditions~\eqref{eq::eigen-admissible-s} and~\eqref{eq::eigen-max}
hold, which guarantees that $\theta_{|I|, |S_{\drop}|}$ is bounded within a 
reasonable constant, when $|I| + |S_{\drop}| \leq 2s$ (cf. Lemma~\ref{lemma:parallel}).
It is clear from Lemma~\ref{prop:MSE-missing} 
and Theorem~\ref{THM:PREDITIVE-ERROR}
that, except for the constants that appear before each term, 
namely, $\twonorm{\beta_{\drop}}$ and $\sqrt{|I|}\sqrt{2 \log p} \sigma$,
the bias and variance tradeoffs for the prediction error and the 
$\ell_2$ loss  follow roughly the same trend in their upper bounds.
It will make sense to take a look at the bound on prediction error for the 
Gauss-Dantzig selector stated in Corollary~\ref{cor:uup-pred}, 
which follows immediately from Theorem~\ref{THM:PREDITIVE-ERROR} 
and Lemma~\ref{lemma:threshold-DS}. 
\begin{corollary}
\label{cor:uup-pred}
Under conditions in Theorem~\ref{thm:ideal-MSE-prelude},
the Gauss-Dantzig selector chooses $I$, where $\size{I} \leq 2s_0$,
such that for the OLS estimator $\hat{\beta}$ based on $I$,  we have
$\twonorm{X \hat{\beta}_I - X \beta}/\sqrt{n} \leq 
C_5 \sqrt{s_0} \lambda \sigma$, 
where $C_5 = \sqrt{\Lambda_{\max}(s)} 
(\sqrt{(C_0' + C_4)^2 + 1} (\sqrt{1+a} + \tau^{-1})) + f(I)$, where
$f(I) := {\sqrt{ 2 (1+ a) \Lambda_{\max}(|I|)}}/{ \Lambda_{\min}(|I|)}$.
\end{corollary}
%Before we apply Lemma~\ref{prop:MSE-missing} for the Thresholded Lasso 
%in Section~\ref{sec:oracle-lasso},
\silent{
where $\drop := \{1, \ldots, p\} \setminus I$,
where from~\eqref{eq::MSE-missing} and 
Theorem~\ref{THM:PREDITIVE-ERROR}, we have
\bens
\twonorm{\hat{\beta}_I - \beta}^2
& \leq &  
(\frac{2 \theta_{|I|, |\dropS|}^2}{\Lambda_{\min}(|I|)^2} + 1)
\twonorm{\beta_{\drop}}^2 +  
\frac{2 \basepen^2  |I|}{\Lambda_{\min}(|I|)^2} \text{ while } \\
\twonorm{ X \hat\beta  - X \beta}^2/n 
& \leq &  
2 \Lambda_{\max}(s) 
\twonorm{\beta_{\drop}}^2 +  
\frac{2 \Lambda_{\max}(|I|)^2 \basepen^2  |I|}{\Lambda_{\min}(|I|)^2}
\eens
}
\section{On sparse oracle inequalities of the Lasso under the RE condition}
\label{sec:oracle-lasso}
In this section, in order to prove Theorem~\ref{thm:RE-oracle-main},
we first show in Theorem~\ref{thm:RE-oracle} that under the RE condition, 
the Lasso estimator achieves essentially the same type of oracle properties 
as the Dantzig selector (under UUP). This result is new to the best of our 
knowledge; it improves upon a result in~\cite{BRT09}  (cf. Theorem 7.2)
under slightly different RE conditions, and thus may be of independent 
interests.
The sparse oracle properties of the Thresholded Lasso in terms of 
variable selection, $\ell_2$ loss, and prediction error then all follow naturally 
from Theorem~\ref{thm:RE-oracle}, Lemma~\ref{lemma:threshold-RE} and 
Lemma~\ref{prop:MSE-missing} as derived in Section~\ref{sec:DS-threshold}.
The proof of Theorem~\ref{thm:RE-oracle} draws upon techniques from a 
concurrent work in~\cite{GBZ10}, where a stronger condition is required, 
while deriving bounds similar to the present  paper.
\begin{theorem}\textnormal{\bf (Oracle inequalities of the Lasso)}
\label{thm:RE-oracle}
Let $Y = X \beta + \e$, for $\e$ being i.i.d. $N(0, \sigma^2)$ and
$\twonorm{X_j} = \sqrt{n}$. 
Let $s_0$ be as in~\eqref{eq::define-s0} and $T_0$
denote
locations of the $s_0$ largest coefficients of $\beta$ in absolute values.
Suppose that $RE(s_0, 6, X)$ holds with $K(s_0, 6)$,
and~\eqref{eq::eigen-admissible-s} and~\eqref{eq::eigen-max} hold.
Let $\beta_{\init}$ be an optimal solution to~\eqref{eq::origin} with 
$\lambda_{n} = d_0 \lambda \sigma \geq 2 \basepen$, where 
 $a \geq 0$ and $d_0 \geq 2 \sqrt{1 + a}$. 
Let $h = \beta_{\init} - \beta_{T_0}$. Then on $\T_a$ as in~\eqref{eq::low-noise},
we have for $\Lambda_{\max} :=\Lambda_{\max}(s- s_0)$,
\bens
\twonorm{\beta_{\init} - \beta}^2 
& \leq &  2 \lambda^2 \sigma^2  s_0(D_0^2 + D_1^2 + 1), \\
\norm{h_{T_0}}_1 + \norm{\beta_{\init, T_0^c}}_1 
& \leq &  \left(\frac{2 \Lambda_{\max}}{d_0} + 
\max\left\{8K^2(s_0, 6) d_0, 
\frac{ \Lambda_{\max}}{ 3d_0}\right\} \right) \lambda \sigma  s_0, \\
\twonorm{ X \beta_{\init} - X \beta }/\sqrt{n}
& \leq & 
\lambda \sigma \sqrt{s_0}  (\sqrt{\Lambda_{\max}}  + 3 d_0 K(s_0, 6))
\eens
where $D_0, D_1$ are defined  in~\eqref{eq::D0-define} 
and~\eqref{eq::D1-define}.
Moreover, for any subset $I_0 \subset S$, 
by assuming that $RE(|I_0|, 6, X)$ holds with $K(|I_0|, 6)$, we have
\ben
\label{eq::pred-error-gen}
\twonorm{ X \beta_{\init} - X \beta}^2/n
\leq 2 \twonorm{X \beta - X \beta_{I_0}}^2/n  + 9 \lambda_n^2  |I_0| K^2(|I_0|, 6).
\een
\end{theorem}
Let $T_1$ denote the $s_0$ largest positions of $h$  in absolute values 
outside of $T_0$; Let $T_{01} := T_0 \cup T_1$. 
The proof of Theorem~\ref{thm:RE-oracle} yields the following bounds:
for $K :=  K(s_0, 6)$, 
$\twonorm{h_{T_{01}}} \leq D_0 \lambda \sigma \sqrt{s_0}$
and
$\norm{h_{T_0^c}}_1 \leq D_1 \lambda \sigma s_0$ where
\ben
\label{eq::D0-define}
& & D_0 = \max \{D, K \sqrt{2} (2 \sqrt{\Lambda_{\max}(s- s_0)}  + 3 d_0 K)\}, \\
\nonumber
& & \text{ where } D  =(\sqrt{2} + 1) \frac{\sqrt{\Lambda_{\max}(s - s_0)}}
{\sqrt{\Lambda_{\min}(2s_0) }}
+ \frac{\theta_{s_0, 2s_0} \Lambda_{\max}(s - s_0)}
{\Lambda_{\min}(2s_0)} \text{ and } \\
\label{eq::D1-define}
& & D_1 = 2 \Lambda_{\max}(s- s_0)/d_0 + 9 K^2 d_0/2.
\een
%Hence by Lemma~\ref{lemma::h01-bound-CT}, we have
%$\twonorm{h_{T_{01}^c}}
% \leq \norm{h_{T_0^c}}_1/\sqrt{s_0} \leq  D_1 \lambda \sigma \sqrt{s_0}$
%and
%$\twonorm{h} \leq \sqrt{D_0^2 + D_1^2}  \lambda \sigma \sqrt{s_0}$.
The proof of Lemma~\ref{lemma:threshold-RE} follows exactly that 
of  Lemma~\ref{lemma:threshold-DS}, and hence omitted.
We then state the bound on prediction error for $\hat{\beta}$ for the 
Thresholded Lasso, which follows immediately from 
Theorem~\ref{THM:PREDITIVE-ERROR} 
and Lemma~\ref{lemma:threshold-RE}.
\begin{lemma}
\label{lemma:threshold-RE}
%Let $Y = X \beta + \e$, for $\e$ being i.i.d. $N(0, \sigma^2)$ and
%$\twonorm{X_j} = \sqrt{n}$. 
%Let $s_0$ and  $T_0$ be as in~\eqref{eq::define-s0} and~\eqref{eq::T0-redefine}.
Suppose that $X$ obeys $RE(s_0, 6, X)$, 
and conditions~\eqref{eq::eigen-admissible-s} 
and~\eqref{eq::eigen-max} hold.
Let $\beta_{\init}$ be an optimal solution to~\eqref{eq::origin} with 
$\lambda_{n} = d_0 \lambda \sigma \geq 2 \basepen$, where 
$a \geq 0$, $d_0 \geq 2 \sqrt{1 + a}$, and $\lambda := \sqrt{2 \log p/n}$
as in Theorem~\ref{thm:RE-oracle}.
Suppose that we choose $t_0 = C_4 \lambda \sigma$ for some positive 
constant $C_4$.
Let $I = \{j: |\beta_{j, \init} \geq t_0\}$ and $\drop := \{1, \ldots, p\} \setminus I$.
Then we have on $\T_a$,
\begin{eqnarray}
\nonumber
|I| & \leq &  s_0 (1 + D_1/C_4)    \text{ and } \; \; |I \cup S|   \leq  s + D_1 s_0/C_4
\; \text{ and } 
\end{eqnarray}
$\twonorm{\beta_{\drop}} \leq \sqrt{(D_0 + C_4)^2 + 1} 
\lambda \sigma \sqrt{s_0}$, where $D_0, D_1$ are as defined 
in~\eqref{eq::D0-define} and~\eqref{eq::D1-define}.
\end{lemma}
\begin{corollary}
\label{eq::RE-pred}
Under conditions in Theorem~\ref{thm:RE-oracle-main}, the
Thresholded Lasso  chooses $I$, where   $\size{I} \leq 2s_0$, 
such that for the OLS estimator $\hat{\beta}$ based on $I$, 
it holds that
$\twonorm{X \hat{\beta}_I - X \beta}/\sqrt{n} \leq 
C_6 \sqrt{s_0} \lambda \sigma$,
where 
$C_6 = \sqrt{\Lambda_{\max}(s)}\sqrt{(D_0 + C_4)^2 + 1}  + f(I)$,
for $f(I)$ as defined in Corollary~\ref{cor:uup-pred} and $D_0$ is defined 
in~\eqref{eq::D0-define}.
\end{corollary}
We now state Lemma~\ref{lemma:parallel}, which follows from~\cite{CT05} (Lemma 1.2); we then prove Theorem~\ref{thm:RE-oracle-main}, where 
we give an explicit expression for $D_3$.
\begin{lemma}
\textnormal{\citep{CT05}}
\label{lemma:parallel}
Suppose that~\eqref{eq::eigen-admissible-s} and~\eqref{eq::eigen-max} hold.
Then for all disjoint sets $I,S_{\drop} \subseteq \{1, \ldots, p\}$ 
of cardinality $|S_{\drop}| < s$ and $|I| + |S_{\drop}| \leq 2s$, 
$$\theta_{|I|, |S_{\drop}|} \leq (\Lambda_{\max}(2s) - \Lambda_{\min}(2s))/{2};$$
In particular, if $\delta_{2s} < 1$, we have
$\theta_{|I|, |\dropS|} \leq \delta_{|I| + |\dropS|} \leq \delta_{2s} < 1.$
\end{lemma}
\begin{proofof}{\textnormal{Theorem~\ref{thm:RE-oracle-main}}}
It holds by definition of $S_{\drop}$ that $I \cap S_{\drop} = \emptyset$.
It is clear by Lemma~\ref{lemma:threshold-RE} that
for $C_4 \geq D_1$, 
$|I| \leq 2s_0$ and $|I \cup S_{\drop}| \leq |I \cup S| \leq s + s_0 \leq 2s$,
given that $|\dropS| < s$.
We have by Lemma~\ref{prop:MSE-missing}
\bens
\lefteqn{
\twonorm{\hat{\beta}_I - \beta}^2 \leq 
\twonorm{\beta_{\drop}}^2
\left(1 + \frac{2 \theta^2_{|I|, |\dropS|}}{\Lambda_{\min}^2(|I|)}\right) +
\frac{2 |I|}{\Lambda_{\min}^2(|I|)}\basepen^2 } \\
%& \leq &
%\twonorm{\beta_{\drop}}^2
%\left(1 + \frac{2 \theta^2_{s, |I|}}{\Lambda_{\min}^2(|I|)}\right) +
%\frac{2 s_0 + 2 D_1/C_4 }{\Lambda_{\min}^2(|I|)}\basepen^2 \\
%& \leq &
%\lambda^2 \sigma^2 s_0
%\left((D_0 + C_4)^2 + 1) 
%\left(1 + \frac{2 \theta^2_{s, |I|}}{\Lambda_{\min}^2(|I|)}\right)
%+ \frac{2(1 + D_1/C_4) (1+a)}{\Lambda_{\min}^2(|I|)}\right) \\
& \leq &
D_3^2 \lambda^2 \sigma^2 s_0 \leq 
2 D_3^2 \log p \left(\sigma^2/n + \sum_{i=1}^p \min(\beta_i^2, \sigma^2/n)
\right) \text{ where }
\eens
$D_3^2 = ((D_0 + C_4)^2 + 1) 
\left(1 + {2 \theta_{|I|, |\dropS|}^2}/{\Lambda_{\min}^2(|I|)} \right) + 
{4(1+a)}/{\Lambda_{\min}^2(|I|)}$.
\end{proofof}
It is clear by  Lemma~\ref{lemma:parallel} that 
\ben
\label{eq::D3-constant}
& &  \; \; \; D_3^2  \leq ((D_0 + C_4)^2 + 1) 
\left(1 + \frac{(\Lambda_{\max}(2s) - \Lambda_{\min}(2s))^2}{2\Lambda_{\min}^2(|I|)} \right) + \frac{4(1+a)}{\Lambda_{\min}^2(|I|)}.
\een
\section{Controlling Type-II errors}
\label{SEC:GENERAL-RE}
In this section, we derive results that are parametrized based on the 
performance of an initial estimator, the smallest magnitude of variables in 
$\{j: |\beta_j| > \lambda \sigma\},$ where $\lambda  := \sqrt{2 \log p/n}$, 
and the choice of the thresholding parameter $t_0$. 
We emphasize that we do not necessarily require that 
$t_0 > \lambda \sigma$.
We first introduce some more notation.
Again order the $\beta_j$'s in decreasing order of magnitude: 
$|\beta_1| \geq |\beta_2| ... \geq |\beta_p|$.
Let  $T_0 = \{1, \ldots, s_0\}$. In view of~\eqref{eq::beta-2-small},
we decompose $T_0 = \{1, \ldots, s_0\}$ into two sets: 
$A_0$ and $T_0 \setminus A_0$, 
where  $A_0$ contains the set of coefficients of $\beta$ strictly larger than
$\lambda \sigma$, for which we define a constant:
\ben
\label{eq::define-a0}
& & \;\;\; A_0 = \{j: |\beta_j| > \lambda \sigma \} =: \{1, \ldots, a_0\};\ \; 
\text{ Let } \; \beta_{\min,A_0}:= \min_{j \leq a_0} |\beta_{j}| > \lambda \sigma.
\een
Our goal is to show when $\beta_{\min,A_0}$ is sufficiently large,
we have $A_0 \subset I$ while achieving the sparse oracle inequalities;
This is shown in Theorem~\ref{thm:threshold-general} under the RE condition,
which is stated as a corollary of  Lemma~\ref{lemma:threshold-general-II}.
First note that changing the coefficients of $\beta_{A_0}$ will not change 
the values of $s_0$ or $a_0$, so long as their absolute values stay strictly 
larger than $\lambda \sigma$. 
Thus one can increase $t_0$ as $\beta_{\min,A_0}$ increases in order to 
reduce false positives while not increasing false 
negatives from the set $A_0$.
In Lemma~\ref{lemma:threshold-general-II},
we impose a lower bound on $\beta_{\min, A_0}$~\eqref{eq::betaA-min-cond}  in order to recover the 
subset of variables in $A_0$, while achieving the  nearly ideal $\ell_2$
loss with a sparse model $I$.

We now show In Lemma~\ref{lemma:threshold-general} that 
under no restriction on $\beta_{\min}$,  we achieve an oracle bound on the 
$\ell_2$ loss, which depends only on 
the $\ell_2$ loss of the initial estimator on the set $T_0$.
Bounds in Lemma~\ref{lemma:threshold-DS}
and~\ref{lemma:threshold-RE}
are special cases~\eqref{eq::off-beta-norm-bound-2}  as we state now.
\begin{lemma}
\label{lemma:threshold-general}
Let $\beta_{\init}$ be an initial estimator. 
Let  $h = \beta_{\init} - \beta_{T_0}$ and $\lambda := \sqrt{2 \log p/n}.$ 
Suppose that we choose a thresholding parameter $t_0$ and set
$$I = \{j: \size{\beta{j, \init}} \geq t_0\}.$$
Then for $\drop := \{1, \ldots, p\} \setminus I$, we have for 
$\drop_{11} := \drop \cap A_0$ and $a_0 = \size{A_0}$,
\begin{eqnarray}
\label{eq::off-beta-norm-bound-2}
\twonorm{\beta_{\drop}}^2 & \leq  & (s_0 - a_0) \lambda^2 \sigma^2 + 
(t_0 \sqrt{a_0} + \twonorm{h_{\drop_{11}}})^2.
\end{eqnarray}
Suppose that $t_0 < \beta_{\min,A_0}$ as defined in~\eqref{eq::define-a0}.
Then~\eqref{eq::off-beta-norm-bound-2} can be replaced by
\begin{eqnarray}
\label{eq::off-beta-norm-bound-alt}
\twonorm{\beta_{\drop}}^2 & \leq  & (s_0 - a_0) \lambda^2 \sigma^2 + 
\twonorm{h_{\drop_{11}}}^2 
\left({\beta_{\min,A_0}}/{(\beta_{\min,A_0} - t_0)}\right)^2.
\end{eqnarray}
\end{lemma}
\begin{lemma}{\textnormal{\bf(Oracle Ideal MSE with $\ell_{\infty}$ bounds)}}
\label{lemma:threshold-general-II}
Suppose that~\eqref{eq::eigen-admissible-s} and~\eqref{eq::eigen-max} hold.
Let $\beta_{\init}$ be an initial estimator. Let $h =   \beta_{\init} - \beta_{T_0}$
and $\lambda := \sqrt{2 \log p/n}$. 
Suppose on some event $Q_c$, for $\beta_{\min,A_0}$ as 
defined in~\eqref{eq::define-a0}, it holds that
\ben
\label{eq::betaA-min-cond}
\beta_{\min, A_0}  \geq \norm{h_{A_0}}_{\infty} + 
\min\left\{(s_0)^{1/2} \twonorm{h_{T_0^c}}, \;
(s_0)^{-1} \norm{h_{T_0^c}}_1 \right\}.
\een
Now we choose a thresholding parameter $t_0$ such that  on $Q_c$,
for some $\breve{s}_0 \geq s_0$,
\ben
\label{eq::ideal-t0}
& & \beta_{\min, A_0}  -  \norm{h_{A_0}}_{\infty}
\geq t_0  \geq
\min\left\{(\breve{s}_0)^{-1/2} \twonorm{\beta_{\init, T_0^c}},
(\breve{s}_0)^{-1} \norm{\beta_{\init, T_0^c}}_1\right\}
\een
holds and set  $I = \{j: \size{\beta_{j, \init}} \geq t_0\}$; 
Then we have on  $\T_a \cap Q_c$,
\ben
& & 
\label{eq::lemma-last1}
A_0 \subset I \; \text{ and } \; |I \cap T_0^c| \leq \breve{s}_0; \; 
\text{ and hence} \;
|I| \leq s_0 + \breve{s}_0; \\
\label{eq::lemma-last2}
 \text{ and } & & 
\twonorm{\beta_{\drop}}^2 \leq (s_0 - a_0) \lambda^2 \sigma^2.
\een
For $\hat\beta_{I}$ being the OLS estimator based on $(X_I, Y)$
and  $\breve{s}_0 \leq s$, we have on  $\T_a \cap Q_c$, 
\ben
\label{eq::threshold-general-II}
\twonorm{\hat{\beta}_I - \beta}^2 & \leq &  
{C_7 \breve{s}_0 \lambda^2 \sigma^2}/{\Lambda_{\min}^2(|I|)} 
\een
where $C_7$ depends on $\theta_{|I|, |\dropS|}$ which is upper bounded by 
$(\Lambda_{\max}(2s) - \Lambda_{\min}(2s))/{2}.$
\end{lemma}
By introducing $\breve{s}_0$, the dependency of $t_0$ on the knowledge of 
$s_0$ is relaxed; in particular, it can be used to express a desirable level of 
sparsity for the model $I$ that one wishes to select. 
We note that implicit in the statement of 
Lemma~\eqref{lemma:threshold-general-II}, we assume the knowledge of the 
bounds on various norms of $\beta_{\init} - \beta$ 
(hence the name of ``oracle'').
Theorem~\ref{thm:threshold-general} is an immediate corollary of 
Lemma~\ref{lemma:threshold-general-II}, with the difference
being: we now let $\breve{s}_0 = s_0$ everywhere and assume
having an upper estimate $\breve{D}_1$ of $D_1$, so as not to depend on 
an ``oracle'' telling us an exact value.
\begin{theorem}
\label{thm:threshold-general}
Suppose that $RE(s_0, 6, X)$ condition holds.
Choose $\lambda_n \geq b \basepen$, where $b \geq 2$.
Let $\beta_{\init}$ be the Lasso estimator as in~\eqref{eq::origin}.
Suppose that for some constants $\breve{D}_1 \geq D_1$, and for
$D_0, D_1$ as in~\eqref{eq::D0-define} and~\eqref{eq::D1-define},
it holds that
\bens
\label{eq::betaA-min-RE}
\beta_{\min, A_0}  \geq  D_0 \lambda \sigma \sqrt{s_0} 
+ \breve{D}_1 \lambda \sigma, \text{ where }
\; \lambda := \sqrt{2 \log p/n},
\eens
Choose a thresholding parameter $t_0$ and set 
\bens
I = \{j: \size{\beta_{j, \init}} \geq t_0\}, \; \text{ where }
t_0 \geq \breve{D}_1 \lambda \sigma.
\eens
Then on $\T_a$,~\eqref{eq::lemma-last1},
~\eqref{eq::lemma-last2}, and~\eqref{eq::threshold-general-II}
all hold with $\breve{s}_0 = s_0$ everywhere and
$C_7 \leq {\Lambda_{\min}^2(|I|)} + {(\Lambda_{\max}(2s) - \Lambda_{\min}(2s))^2}/{2} + 4(1+ a)$;
Moreover, the OLS estimator $\hat{\beta}$ based on $I$ achieves on $\T_a$,
for $f(I)$ as defined in Corollary~\ref{cor:uup-pred},
where  $\size{I} \leq 2s_0$,
\bens
\twonorm{X \hat{\beta}_I - X \beta}/\sqrt{n} \leq C_8 \sqrt{s_0} \lambda \sigma
\text{  where } \; C_8 = \sqrt{\Lambda_{\max}(s)}  + f(I).
\eens
\end{theorem}
\subsection{Discussions}
Compared to Theorem~\ref{THM:RE}, we now put a lower bound on
$\beta_{\min, A_0}$ rather than on the entire set $S$ in 
Theorem~\ref{thm:threshold-general}, with the hope to recover
$A_0$.
Choosing the set $A_0$ is rather arbitrary; one could for example, 
consider the set of variables that are strictly above 
$\lambda \sigma/2$ for instance.
Bounds on $\norm{h_{A_0}}_{\infty}$ are in general 
harder to obtain than $ \twonorm{h_{A_0}}$;
Under stronger incoherence conditions, such bounds can be obtained; 
see for example~\cite{Lou08,Wai08,CP09}. In general, we can still hope to 
bound $\norm{h_{A_0}}_{\infty}$ by $\twonorm{h_{A_0}}$.
Having a tight bound on  $\twonorm{h_{T_0}}$ (or 
$\norm{h_{T_0}}_{\infty}$) and $\shtwonorm{h_{T_0^c}}$ naturally
helps relaxing the requirement on $\beta_{\min, A_0}$ for 
Lemma~\ref{lemma:threshold-general-II}, while 
in Lemma~\ref{lemma:threshold-general}, such tight upper bounds
will help us  to control both the size of $I$ and $\norm{\beta_{\drop}}$ and 
therefore achieve a tight bound on the $\ell_2$ loss in the expression 
of Lemma~\ref{prop:MSE-missing}. 
In general, when the strong signals are close to each other in their strength,
then a small $\beta_{\min, A_0}$ implies that we are in a situation with low 
signal to noise ratio (low SNR); one needs to carefully tradeoff false positives 
with false negatives; this is shown in our experimental results
in Section~\ref{sec:experiments}. 
We refer to~\cite{Wai09b} and references therein for discussions on 
information theoretic limits 
on sparse recovery where the particular estimator is not specified.

\silent{
Note that in order for $|I| + |S_{\drop}| \leq 2s$ to hold, 
as required by Lemma~\ref{prop:MSE-missing} and~\ref{lemma:parallel},  
we only need to guarantee that
$$ |I \cap \Sc| \leq s, \;  \text{ so that } \; 
|I| + |S_{\drop}| := | I \cup S_{\drop} | \leq |S| + |I \cap \Sc| \leq 2s.$$
}

\silent{It is clear that if there exists $t_0$ such that
$$\beta_{\min, A_0} - \twonorm{\upsilon_{A_0}} \geq t_0 
\geq \twonorm{ \beta^{(12)} + \beta^{(2)}} + \twonorm{\upsilon_{A_0^c}},$$ 
holds, for which the following inequality is sufficient:
$$\beta_{\min, A_0} \geq \twonorm{\upsilon_{A_0}} + 
\sqrt{s_0 - a_0} \lambda \sigma + \twonorm{\upsilon_{A_0^c}}$$
then we have exact recovery of $A_0$.}

\def\sleft{\hskip-5pt}
\def\lleft{\hskip-25pt}
\begin{figure}
\begin{center}
\begin{tabular}{cc}
\begin{tabular}{c}
\includegraphics[width=0.35\textwidth,angle=270]{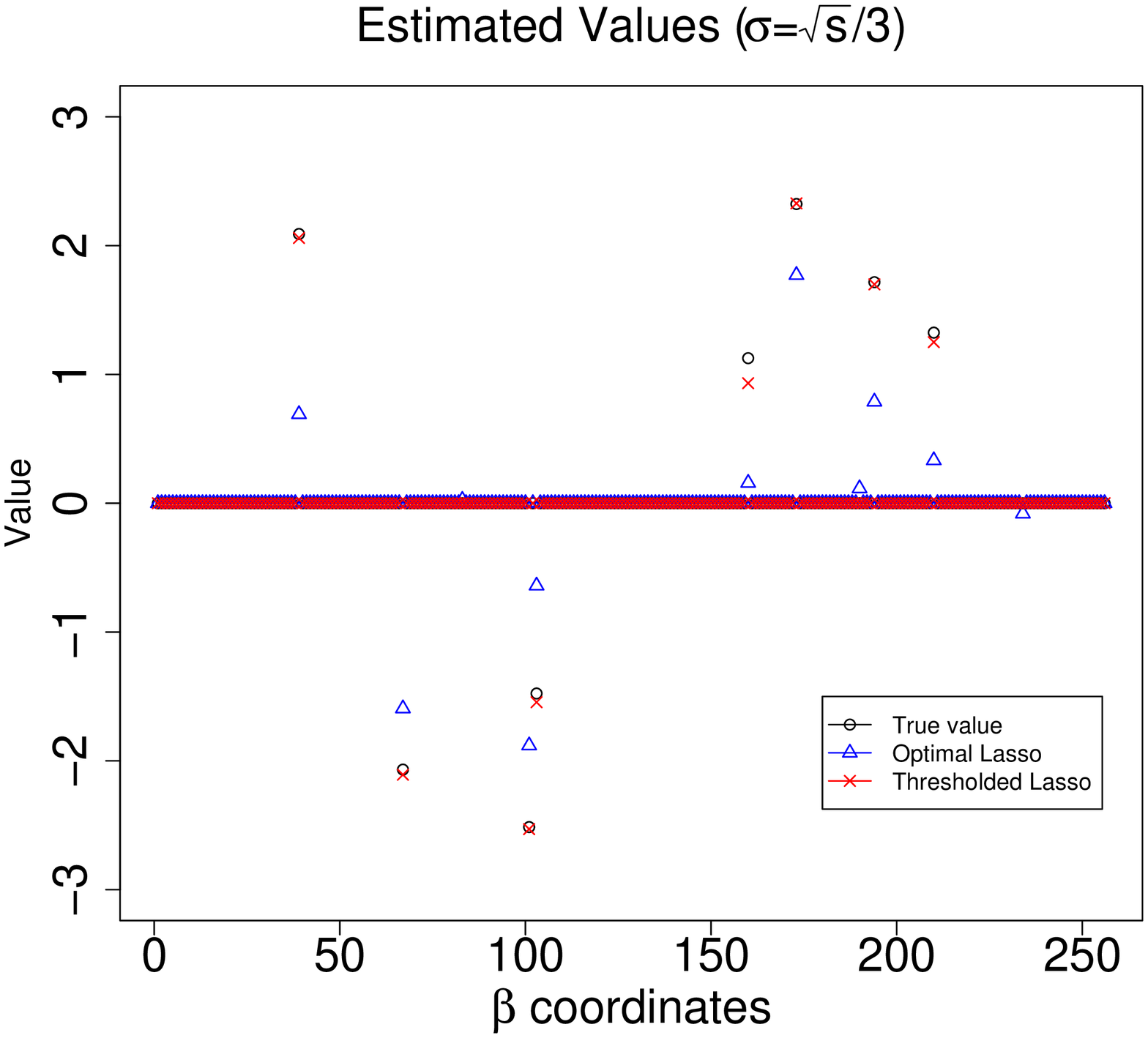}
\end{tabular}& 
\begin{tabular}{c}
\includegraphics[width=0.35\textwidth,angle=270]{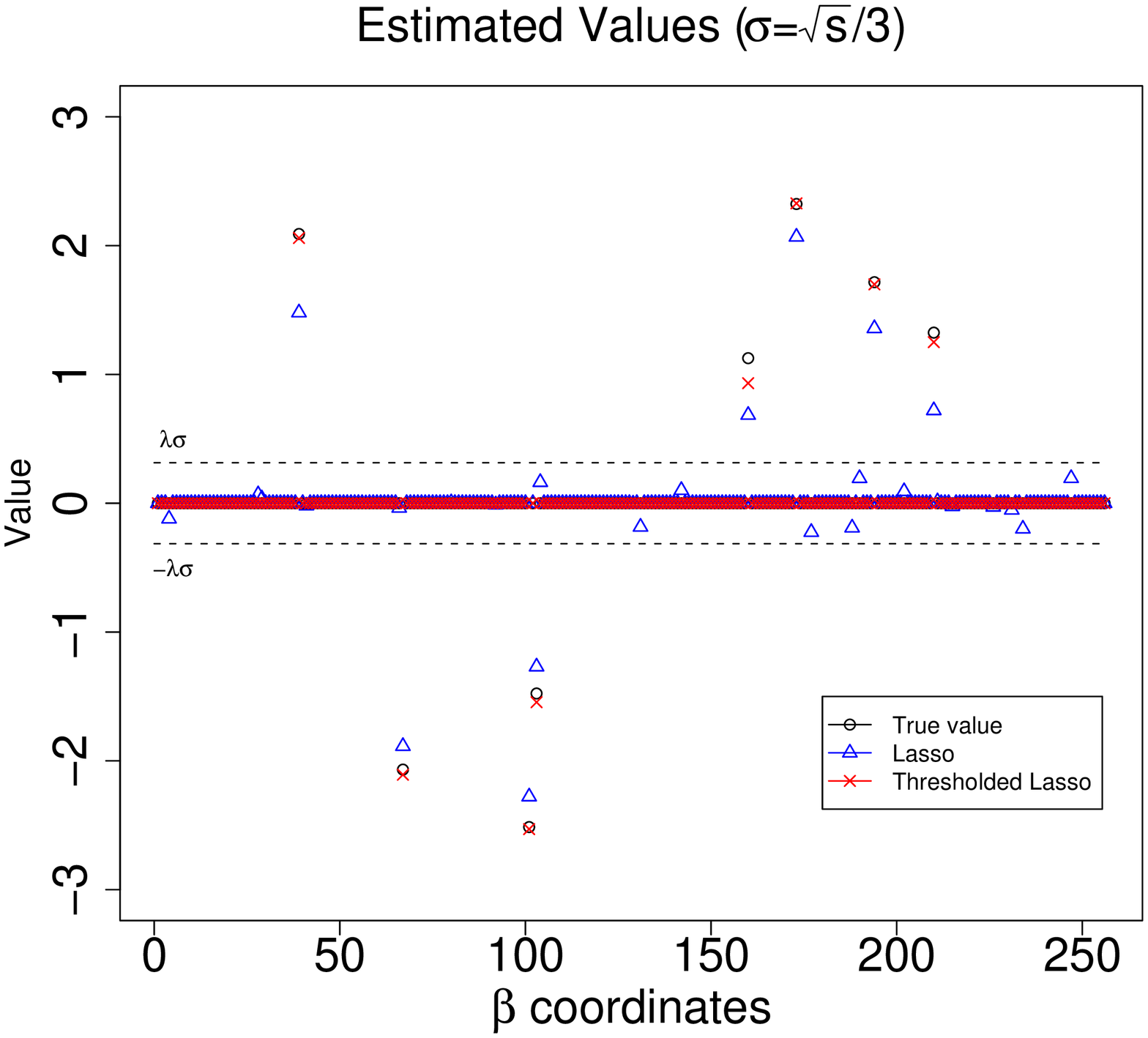} 
\end{tabular} \\
(a) & (b) \\
\end{tabular}
\caption{Illustrative example: i.i.d. Gaussian ensemble; 
$p=256$, $n=72$, $s=8$, and $\sigma = \sqrt{s}/3$.
(a) compare with the Lasso estimator $\tilde{\beta}$ which minimizes 
$\ell_2$ loss. Here $\tilde{\beta}$ has only 3 FPs, but $\rho^2$ is
large with a value of $64.73$. 
(b) Compare with the $\beta_{\init}$ obtained using $\lambda_n$. 
The dotted lines show the thresholding level $t_0$. 
The $\beta_{\init}$ has 15 FPs, all of which were cut after the 1st step; 
resulting  $\rho^2= 12.73$. After refitting with OLS in the 2nd step, 
for the $\hat{\beta}$, $\rho^2$ is further reduced to $0.51$.}
\label{fig:example}
\end{center}
\end{figure}
\section{Numerical experiments}
\label{sec:experiments}
In this section, we present results from numerical simulations designed to 
validate the theoretical analysis presented in previous sections. 
In our Thresholded Lasso implementation
(we plan to release the implementation as an R package), we use a 
{\em Two-step} procedure as described in
Section~\ref{sec:introduction}: we use the Lasso as the initial estimator, 
and OLS in the second step after thresholding.
Specifically, we carry out the Lasso using procedure $\lars(Y, X)$ that implements
the LARS algorithm~\cite{EHJ04} to calculate the full regularization path. 
We then use $\lambda_n$, whose expression is fixed throughout the experiments as follows,
\ben
\label{eq::pen-exp}
\lambda_n = 0.69 \lambda \sigma, \; \text{ where }  \lambda= \sqrt{2 \log p/n},
\text{ in~\eqref{eq::origin} }
\een
to select a $\beta_{\init}$ from this output path as our initial estimator.
We then threshold the $\beta_{\init}$ using a value $t_0$ typically chosen between 
$0.5 \lambda \sigma$ and $\lambda \sigma$. 
See each experiment for the actual value used.
Given that columns of $X$ being normalized to have $\ell_2$
norm $\sqrt{n}$, for each input parameter $\beta$, we compute its 
SNR as follows: 
$$
SNR :=  \twonorm{\beta}^2 / \sigma^2.
$$
To evaluate $\hat{\beta}$, we use metrics defined  in Table~\ref{tab:fpfn};
we also compute the ratio between squared $\ell_2$ error and the 
ideal mean squared error, known as the $\rho^2$;
see Section~\ref{sec:exp-ell2-errors} for details.
\subsection{Illustrative example}
\label{subsec:illus}
In the first example, we run the following experiment with a setup similar to
what was used in~\cite{CT07} to conceptually compare the behavior of 
the Thresholded Lasso with the Gauss-Dantzig selector: 

\begin{enumerate}
\item
Generate an {\it i.i.d. Gaussian ensemble} $X_{n \times p}$, 
where $X_{ij} \sim N(0, 1)$ are independent, which is then 
normalized to have column $\ell_2$-norm $\sqrt{n}$.
\item
Select a support set $S$ of size $|S| = s$ uniformly at random, 
and sample a vector $\beta$ with independent and identically 
distributed entries on $S$ as follows, $\beta_i = \mu_i (1 + |g_i|),$ 
where $\mu_i = \pm 1$ with probability 1/2 and $g_i \sim N(0,1)$.
\item 
Compute $Y = X \beta + \epsilon$, where the noise 
$\epsilon \sim N(0, \sigma^2 I_n)$ is generated with $I_n$ being 
the  $n \times n$ identity matrix. 
Then feed $Y$ and $X$ to the Thresholded Lasso with thresholding 
parameter being $t_0$ to recover $\beta$ using $\hat{\beta}$. 
\end{enumerate}
In Figure~\ref{fig:example}, we set $p=256$, $n=72$, $s=8$, 
$\sigma = \sqrt{s}/3$ and $t_0 = \lambda\sigma$.  
We compare the Thresholded Lasso estimator $\hat{\beta}$
with the Lasso, where the full LARS regularization path is searched to 
find the {\em optimal} $\tilde{\beta}$ that has the minimum $\ell_2$ error.

\subsection{Type I/II errors}
\label{subsec:type12}
We now evaluate the Thresholded Lasso estimator by comparing
Type I/II errors under different values of $t_0$ and SNR.
We consider Gaussian random matrices for the design $X$ with 
both diagonal and  Toeplitz covariance. We refer to the former
as {\it i.i.d. Gaussian ensemble} and the latter as {\it Toeplitz ensemble}.
In the Toeplitz case, the covariance is given by 
$T(\gamma)_{i,j} = \gamma^{|i-j|}$ where $0< \gamma < 1.$
We run under two noise levels: $\sigma = \sqrt{s}/3$ and $\sigma = \sqrt{s}$.
For each $\sigma$, we vary the threshold $t_0$ from 
$0.01 \lambda \sigma$ to $1.5 \lambda \sigma$. 
For each $\sigma$ and $t_0$ combination, 
we run the following experiment: First  we generate $X$ as in Step 1 above. 
After obtaining $X$, we keep it fixed and then repeat Steps $2-3$ for $200$ 
times with a new $\beta$ and $\epsilon$ generated each time and 
we count the number of Type I and II errors in $\hat{\beta}$.
We compute the  average at the end of 200 runs, which will correspond to
one data point on the curves in Figure~\ref{fig:type12} (a) and (b).

For both types of designs, similar behaviors are observed.
For $\sigma=\sqrt{s}/3$, FNs increase slowly; hence there is a wide
range of values from which $t_0$ can be chosen such that 
FNs and FPs are both zero. In contrast, when $\sigma=\sqrt{s}$, 
FNs increase rather quickly as $t_0$ increases due to the low SNR.
It is clear that the low SNR and high correlation combination makes
it the most challenging situation for variable selection, as predicted 
by our theoretical analysis and others. See discussions in 
Section~\ref{SEC:GENERAL-RE}. In (c) and (d), we run additional 
experiments for the low SNR case for Toeplitz ensembles. 
The performance is improved by increasing the
sample size or lowering the correlation factor.
\begin{table}[h]
\begin{center}
\caption{Metrics for evaluating $\hat{\beta}$}
\label{tab:fpfn}
\begin{tabular}{l|l} 
\hline
Metric & Definition \\ \hline
Type I errors or False Positives (FPs) & \# of incorrectly selected non-zeros in $\hat{\beta}$ \\
Type II errors or False Negatives (FNs) & \# of non-zeros in $\beta$ that are not selected in $\hat{\beta}$ \\
True positives (TPs) & \# of correctly selected non-zeros \\ 
True Negatives (TNs) & \# of zeros in $\hat{\beta}$ that are also zero in $\beta$ \\
False Positive Rate (FPR) & $ FPR = FP / (FP + TN) = FP/(p-s) $ \\
True Positive Rate (TPR) & $ TPR = TP/ (TP+FN) = TP/ s $ \\ \hline
\end{tabular}
\end{center}
\end{table}

\subsection{$\ell_2$ loss}
\label{sec:exp-ell2-errors}
We now compare the performance of the Thresholded Lasso with 
the ordinary Lasso by examining the metric $\rho^2$ defined as follows:
$$
\rho^2 = \frac{\sum_{i=1}^p (\hat{\beta}_i - \beta_i)^2}{\sum_{i=1}^p \min(\beta_i^2, \sigma^2/n)}.
$$

\begin{figure}
\begin{center}
\begin{tabular}{cc}
\begin{tabular}{c}
\includegraphics[width=0.33\textwidth,angle=270]{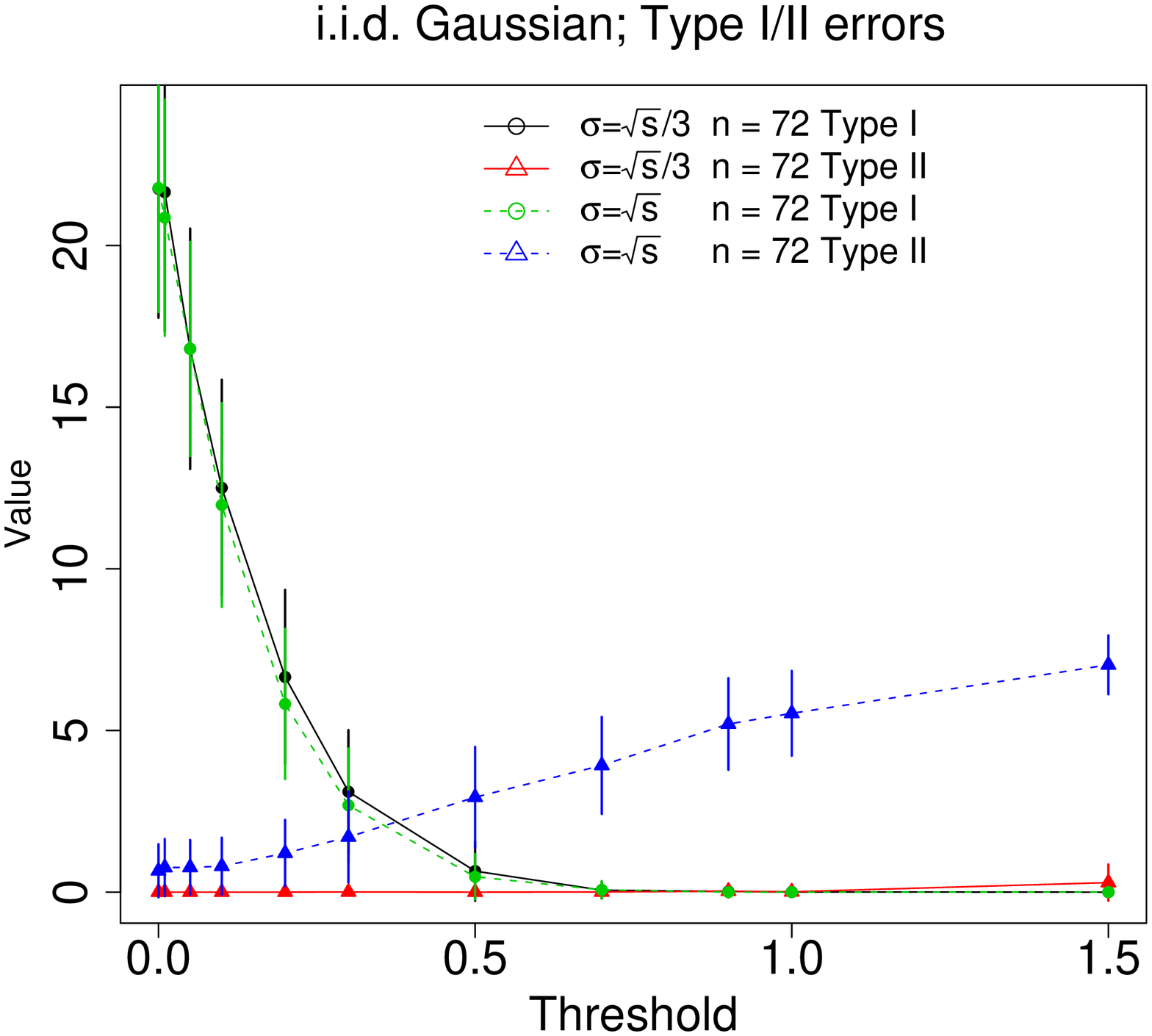}
\end{tabular}& 
\begin{tabular}{c}
\includegraphics[width=0.33\textwidth,angle=270]{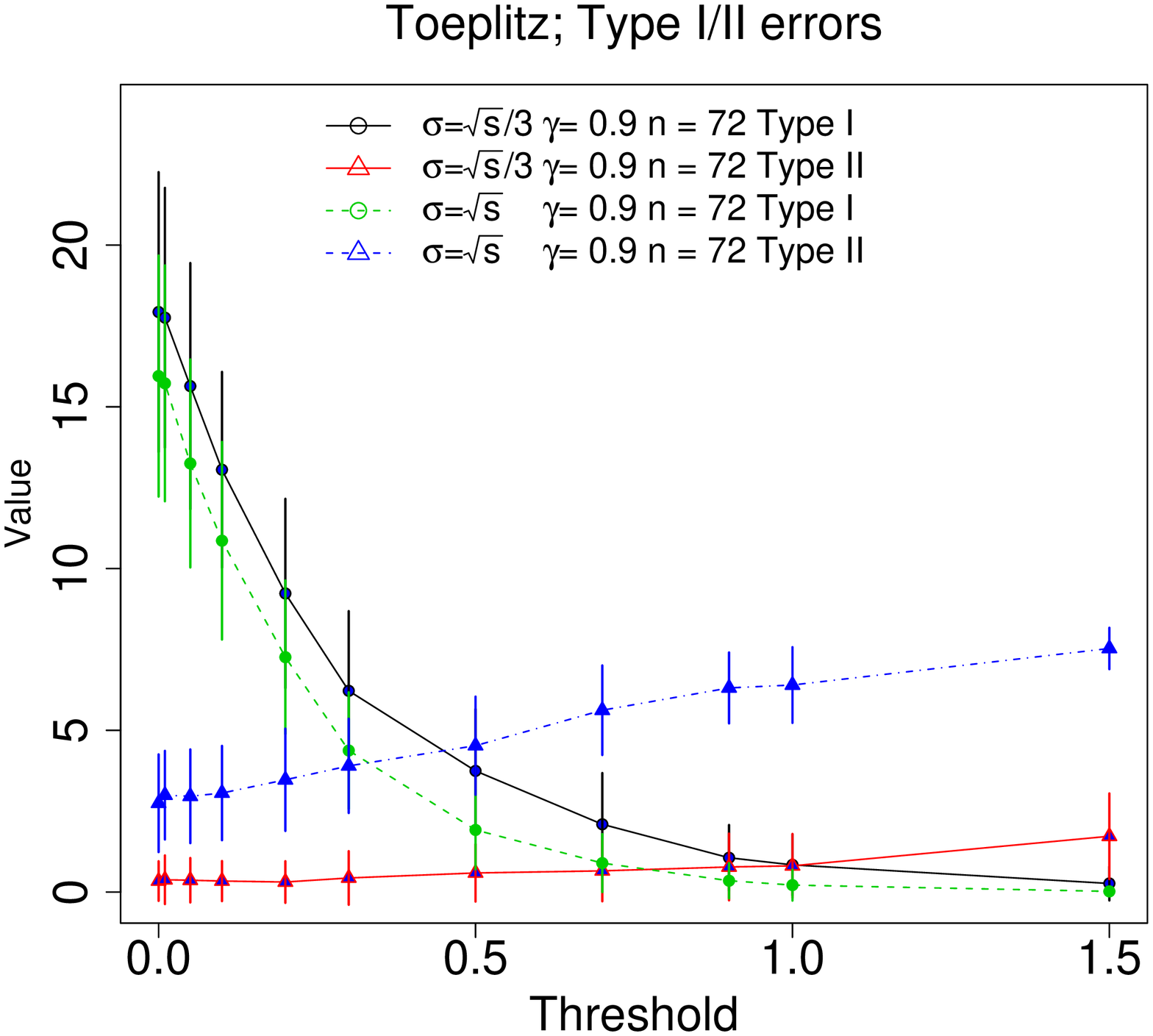} \\
\end{tabular} \\
(a) & (b) \\
\begin{tabular}{c}
\includegraphics[width=0.33\textwidth,angle=270]{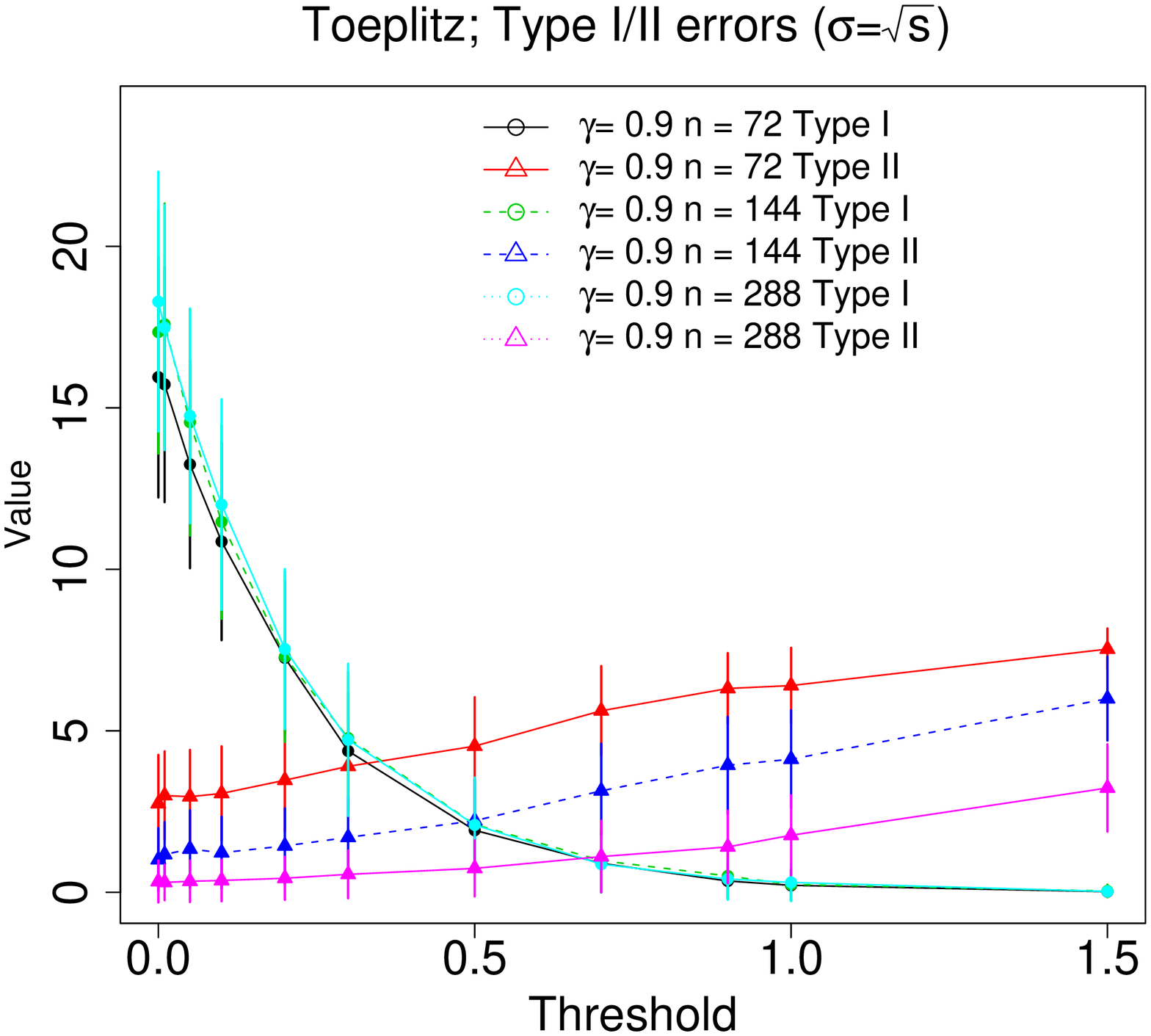}
\end{tabular}& 
\begin{tabular}{c}
\includegraphics[width=0.33\textwidth,angle=270]{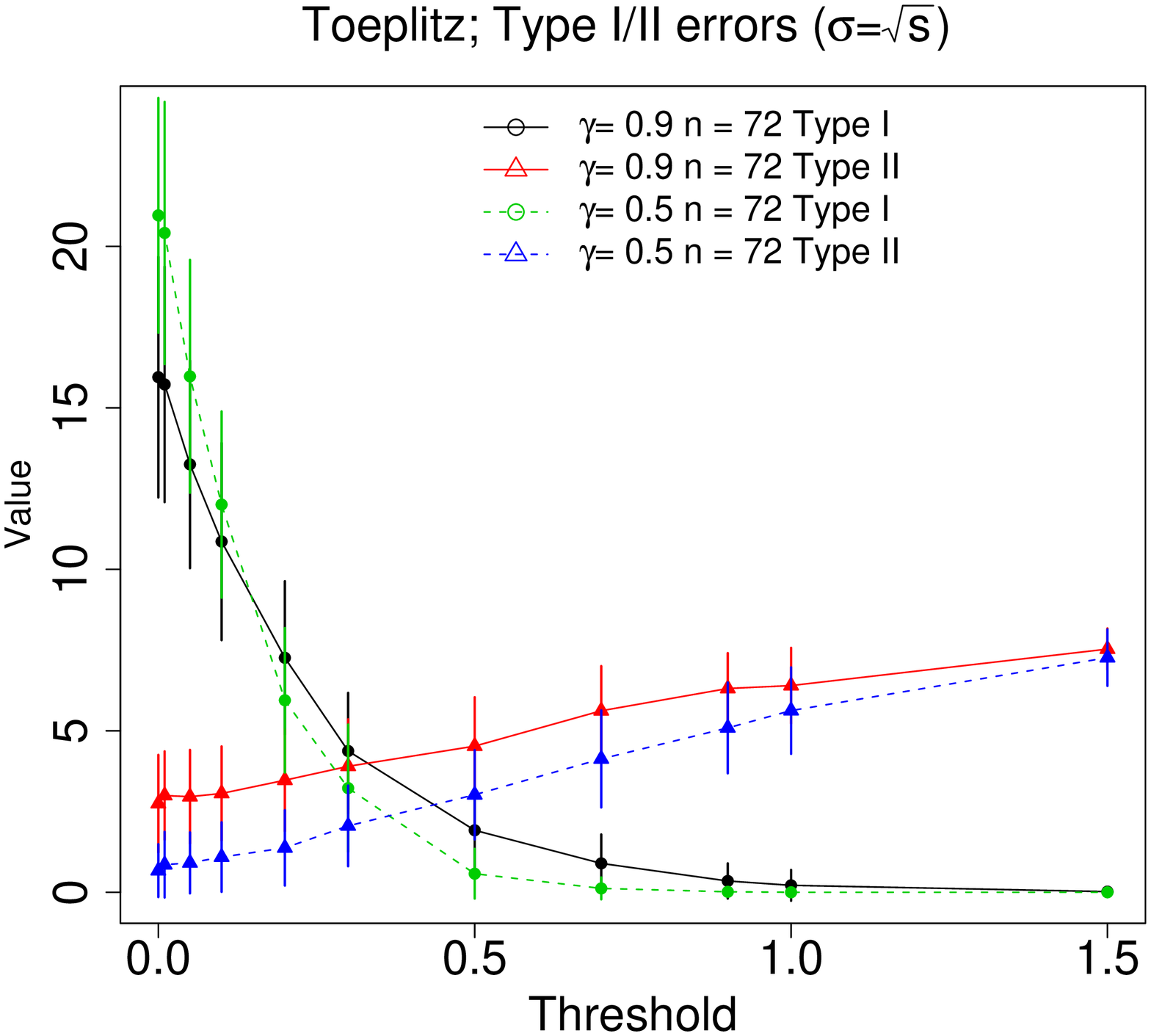} \\
\end{tabular} \\
(c) & (d) \\
\begin{tabular}{c}
\includegraphics[width=0.33\textwidth,angle=270]{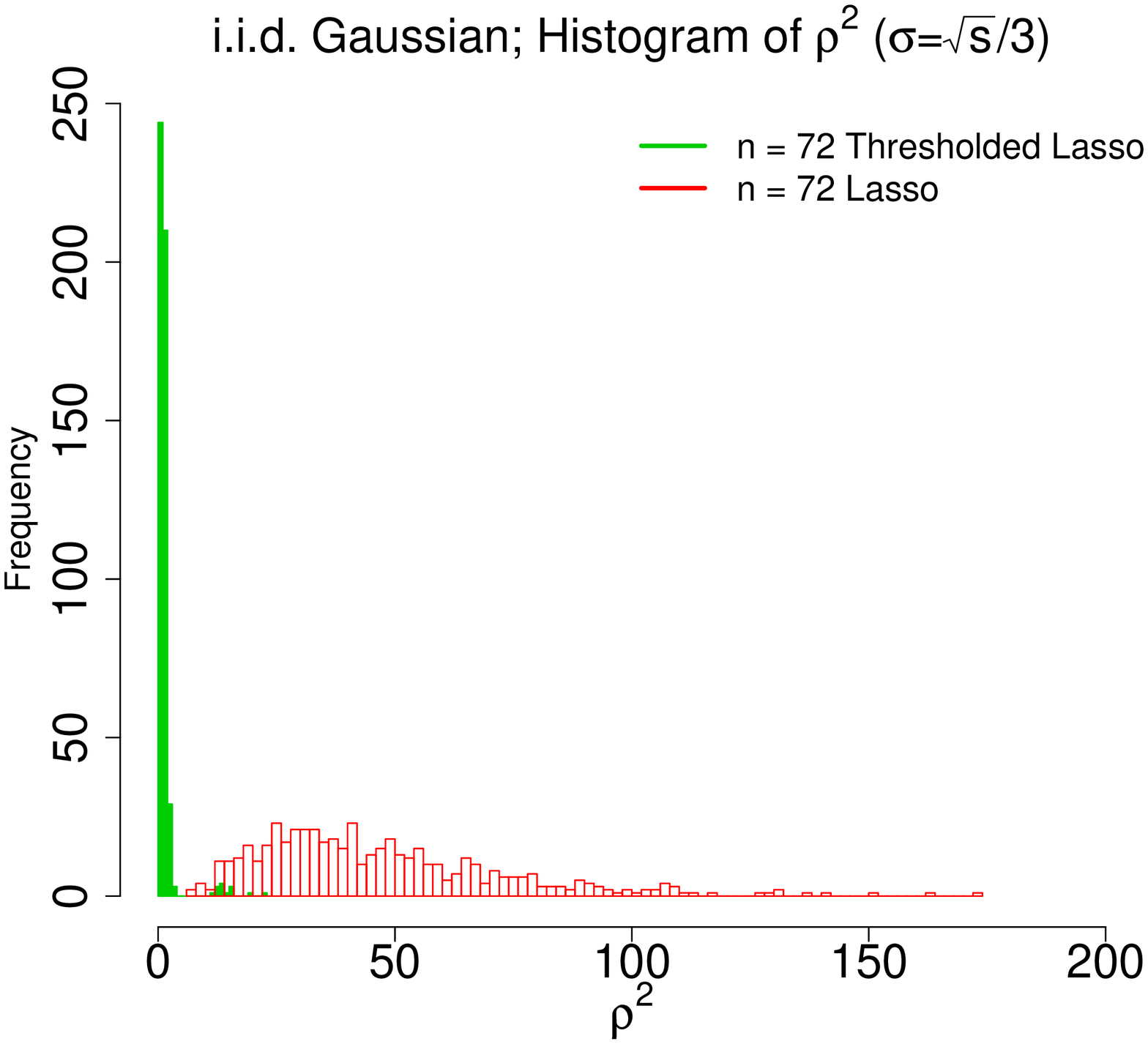}
\end{tabular}& 
\begin{tabular}{c}
\includegraphics[width=0.33\textwidth,angle=270]{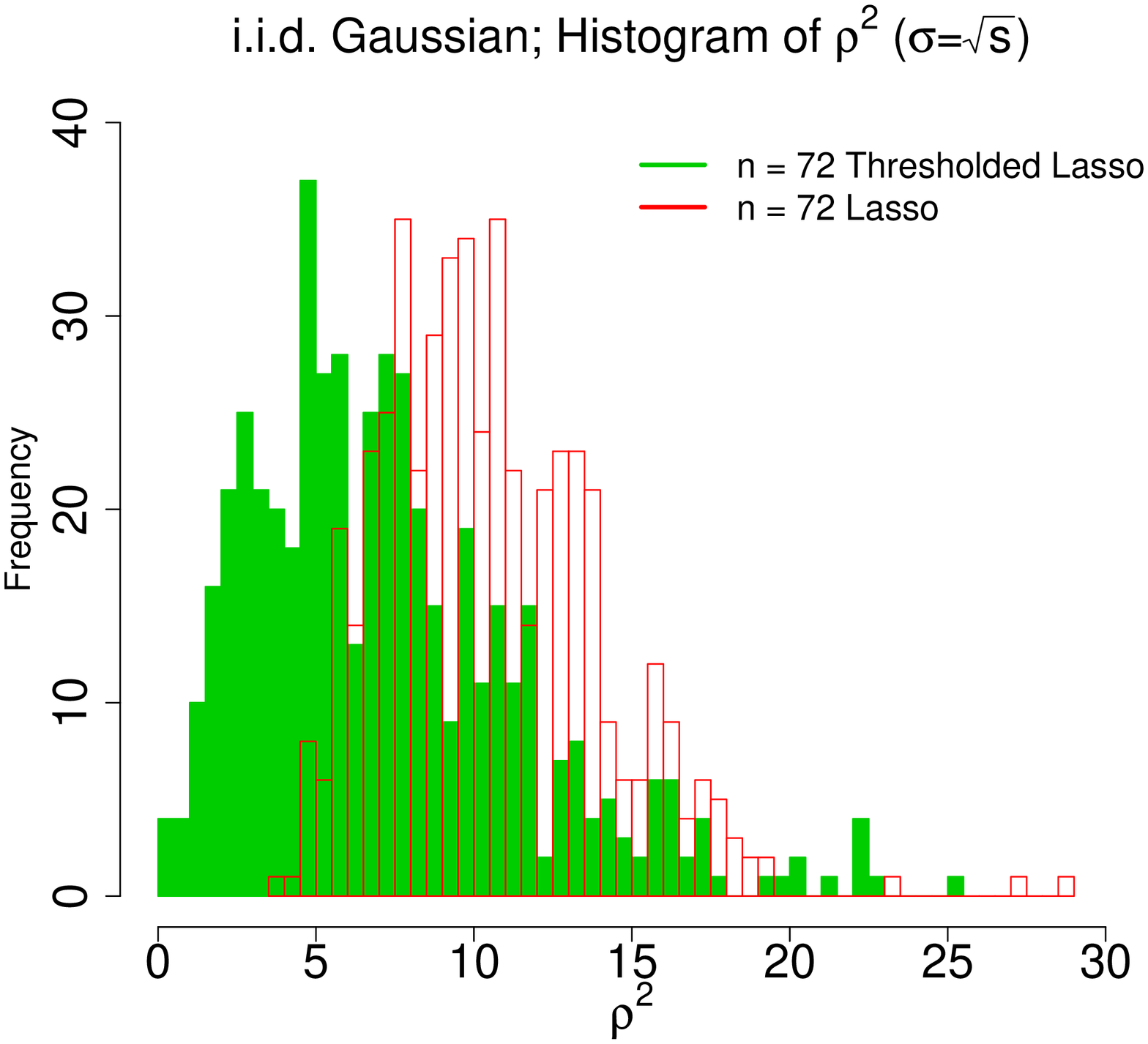} 
\end{tabular} \\
(e) & (f) \\
\end{tabular}
\caption{$p=256$ $s=8$. 
(a) (b) Type I/II errors for i.i.d. Gaussian and Toeplitz ensembles.
Each vertical bar represents $\pm 1$ std. The unit of $x$-axis is
in $\lambda \sigma$. For both types of design matrices,
FPs decrease and FNs increase as the threshold increases. 
For Toeplitz ensembles, in (c) with fixed correlation $\gamma$, 
FNs decrease with more samples, and in (d) with fixed sample size, 
FNs decrease as the correlation $\gamma$ decreases.
(e) (f) Histograms of $\rho^2$ under i.i.d Gaussian ensembles 
from 500 runs.}
\label{fig:type12}
\end{center}
\end{figure}

We first run the above experiment using i.i.d. Gaussian ensemble
under the following thresholds: $t_0 = \lambda \sigma$ for $\sigma=\sqrt{s}/3$,
and $t_0 = 0.36 \lambda \sigma$ for $\sigma=\sqrt{s}$. These are chosen based on the desire 
to have low errors of both types (as shown in Figure~\ref{fig:type12} (a)). 
Naturally, for low SNR cases,  small $t_0$ will reduce Type II errors.
In practice, we suggest using cross-validations to choose the exact 
constants in front of $\lambda \sigma$.
We plot the histograms of $\rho^2$ in Figure~\ref{fig:type12} (e) and (f).
In (e), the mean and median are $1.45$ and $1.01$ for the Thresholded 
Lasso, and  $46.97$ and $41.12$ for the Lasso. 
In (f), the corresponding values are $7.26$ and $6.60$ for 
the Thresholded Lasso and $10.50$ and $10.01$ for the Lasso.
With high SNR, the Thresholded Lasso performs extremely well;
with low SNR, the improvement of the Thresholded Lasso over the
ordinary Lasso is less prominent; this is in close correspondence with the
Gauss-Dantzig selector's behavior as shown by~\cite{CT07}.

Next we run the above experiment under different sparsity values of $s$. 
We again use i.i.d. Gaussian ensemble with 
$p=2000$, $n=400$, and $\sigma =\sqrt{s}/3$. 
The threshold is set at $t_0 = \lambda\sigma$. 
The SNR for different $s$ is fixed at around $32.36$. 
Table~\ref{tab:rho-square-snr} shows the mean of the $\rho^2$ for the 
Lasso and the Thresholded Lasso estimators. 
The Thresholded Lasso performs consistently better than the ordinary Lasso 
until about $s=80$, after which both break down.
For the Lasso, we always choose from the full regularization path
the {\em optimal} $\tilde{\beta}$ that has the minimum $\ell_2$ loss.

\begin{table}[h]
\begin{center}
\caption{$\rho^2$ under different sparsity and fixed SNR. 
Average over 100 runs for each $s$.
}
\label{tab:rho-square-snr}
\begin{tabular}{cccccccc} 
\hline
s      & 5    & 18   & 20   & 40   & 60   & 80   & 100  \\ \hline % & 120\\ \hline
SNR & 34.66 &  32.99 &  32.29&   32.08 &  32.28 &  32.56 &  32.54  \\ \hline %&  32.39  \\ \hline
Lasso  & 17.42 &  22.01 &  44.89 &  52.68 &  31.88 &  29.40 &  47.63 \\ \hline % &  53.27 \\ \hline
Thresholded Lasso  & 1.02& 0.96& 1.11& 1.54& 10.32& 29.38 & 53.81 \\ \hline % & 63.58  \\ \hline
\end{tabular}
\end{center}
\end{table}

\begin{figure}
\begin{center}
\begin{tabular}{cc}
\begin{tabular}{c}
\includegraphics[width=0.33\textwidth,angle=270]{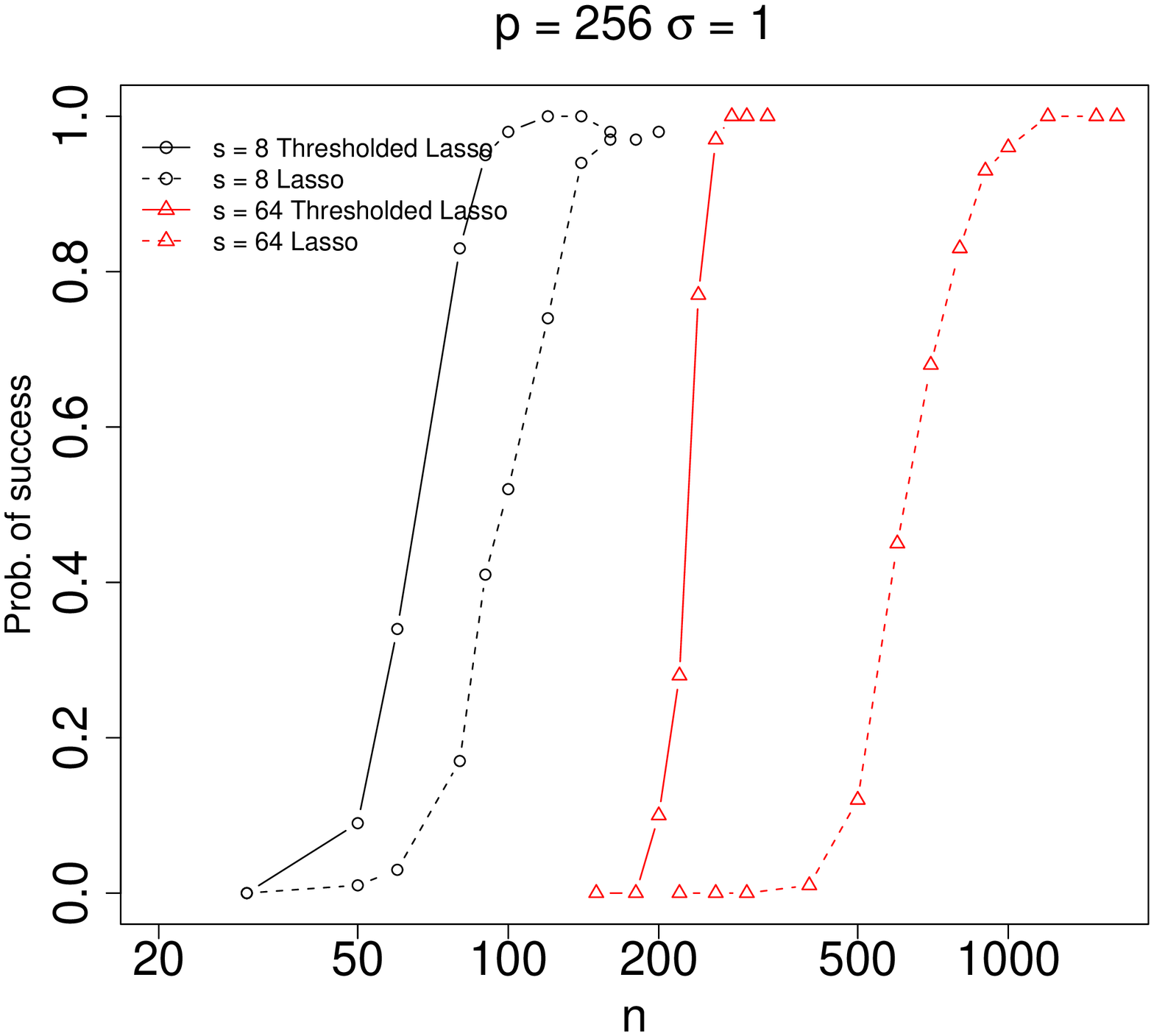}
\end{tabular}&
\begin{tabular}{c}
\includegraphics[width=0.33\textwidth,angle=270]{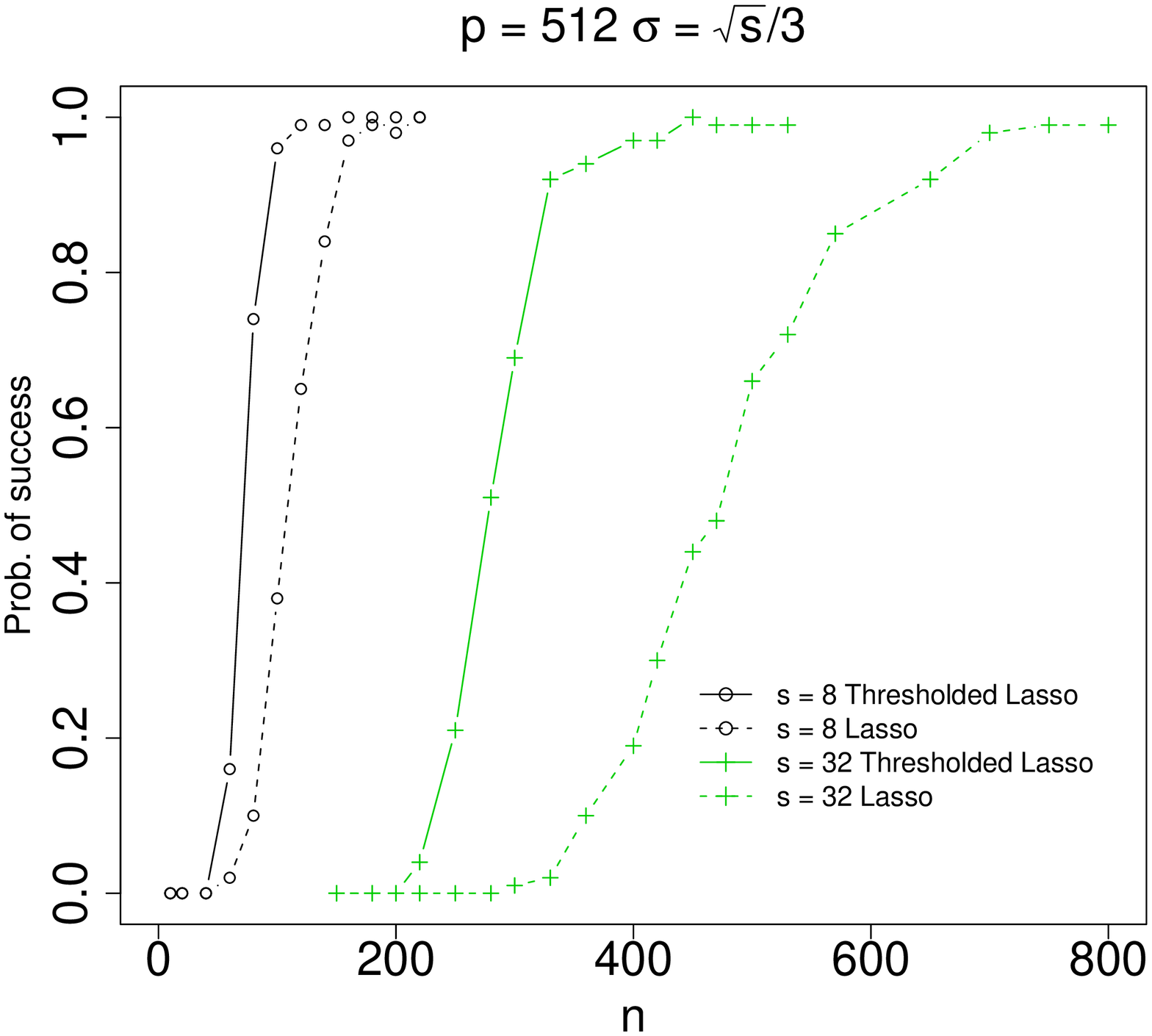}
\end{tabular} \\
(a)&(b) \\
\begin{tabular}{c}
\includegraphics[width=0.33\textwidth,angle=270]{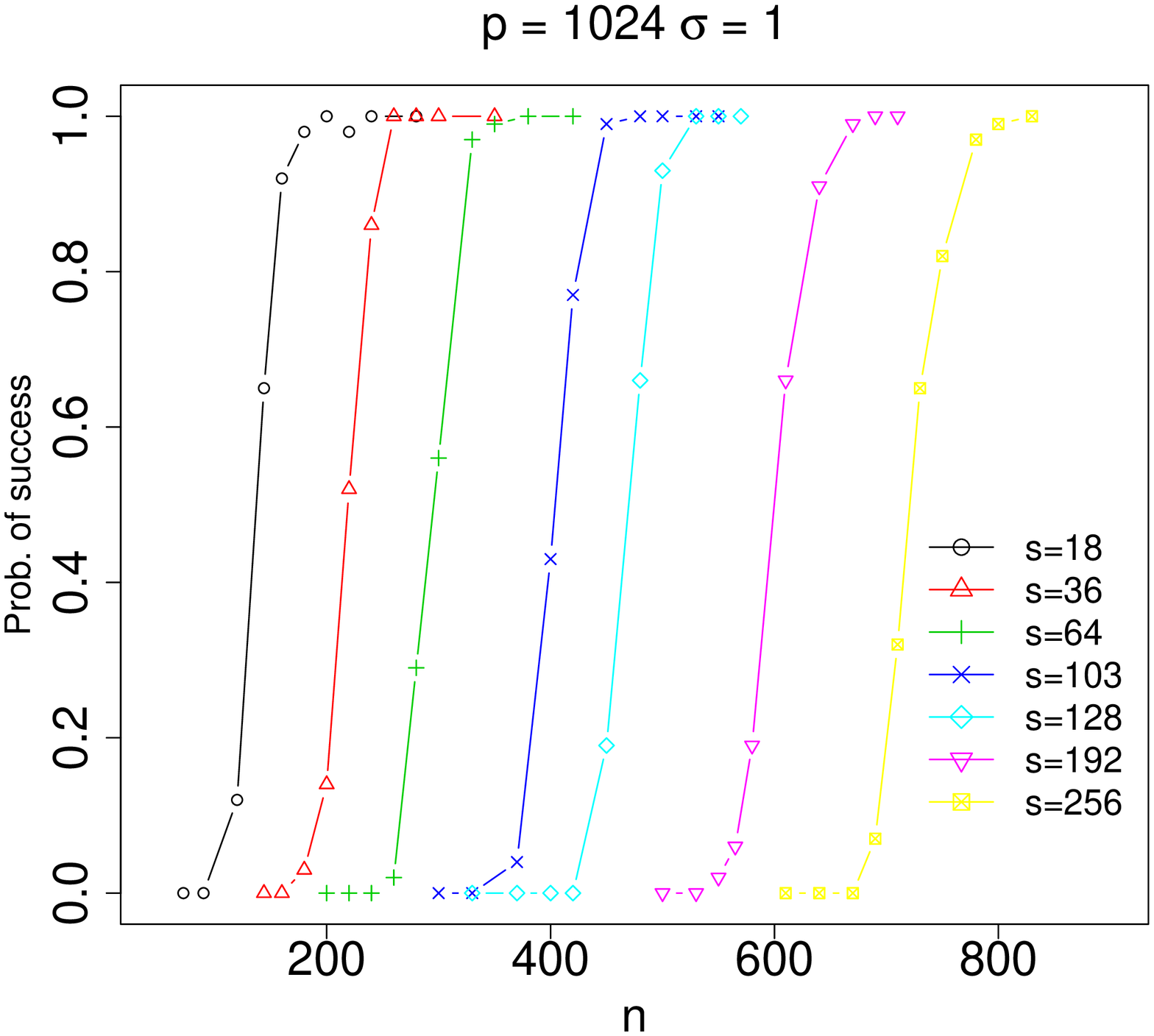}
\end{tabular}&
\begin{tabular}{c}
\includegraphics[width=0.33\textwidth,angle=270]{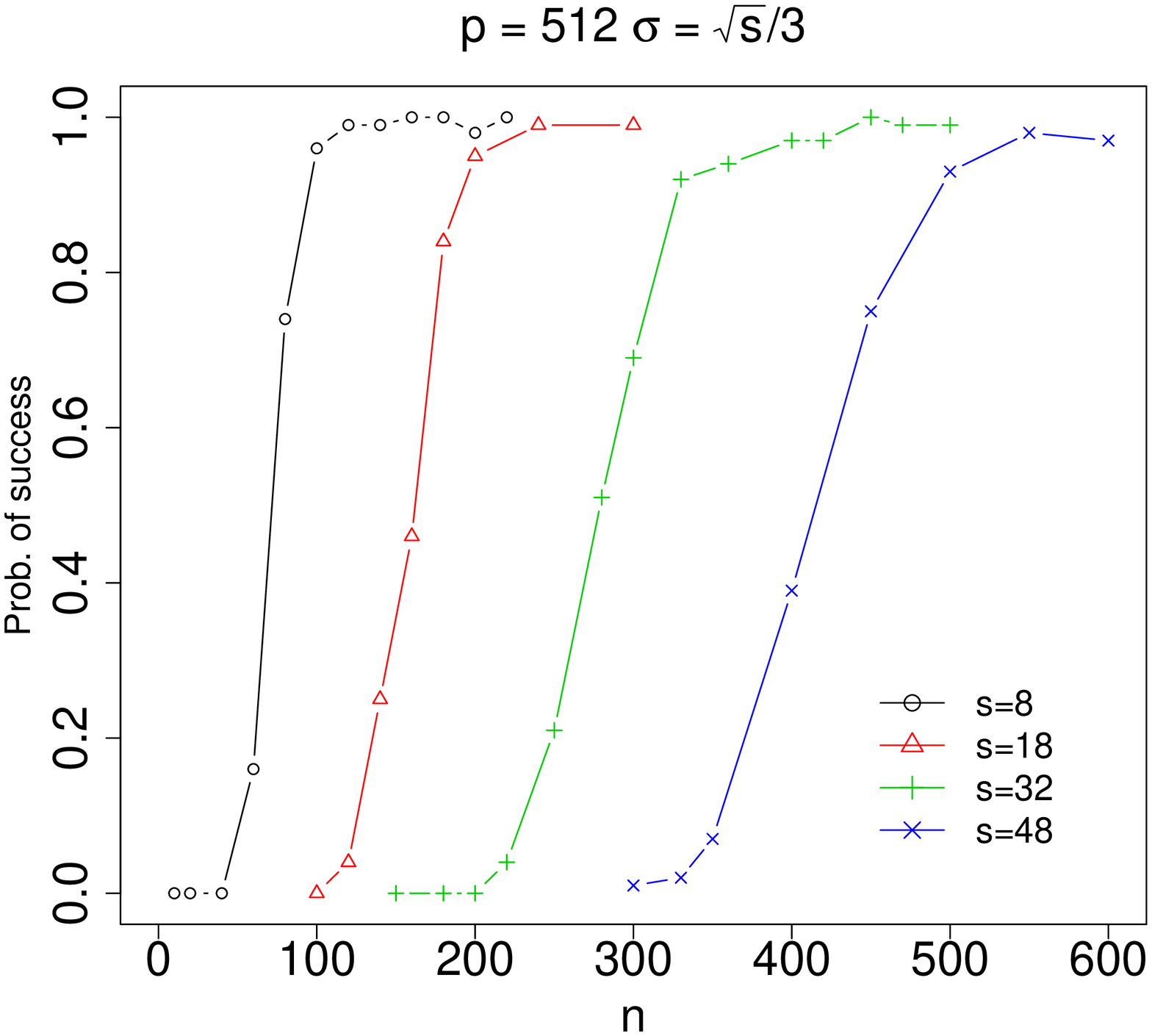}
\end{tabular} \\
(c)&(d) \\
\begin{tabular}{c}
\includegraphics[width=0.33\textwidth,angle=270]{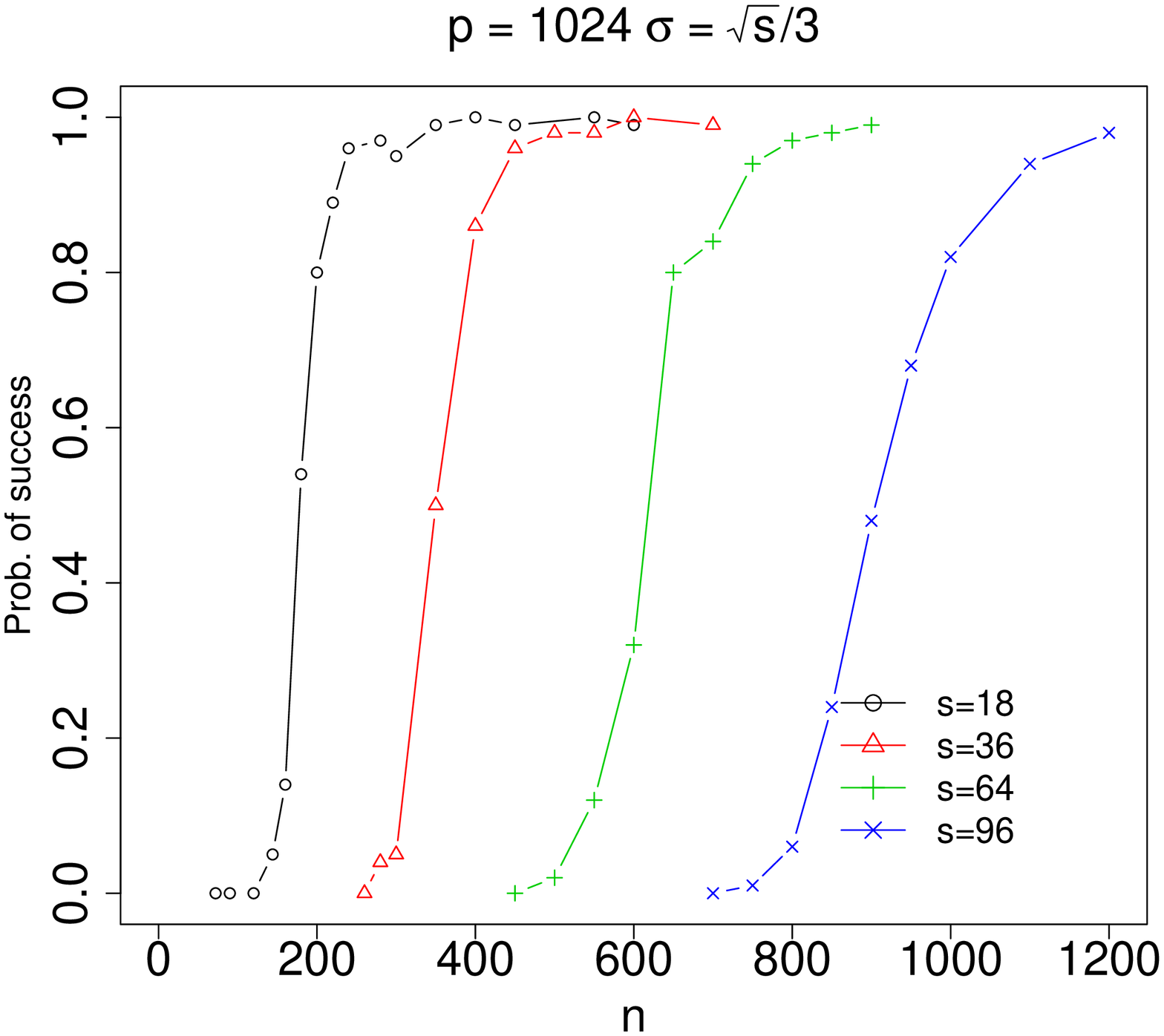}
\end{tabular}&
\begin{tabular}{c}
\includegraphics[width=0.33\textwidth,angle=270]{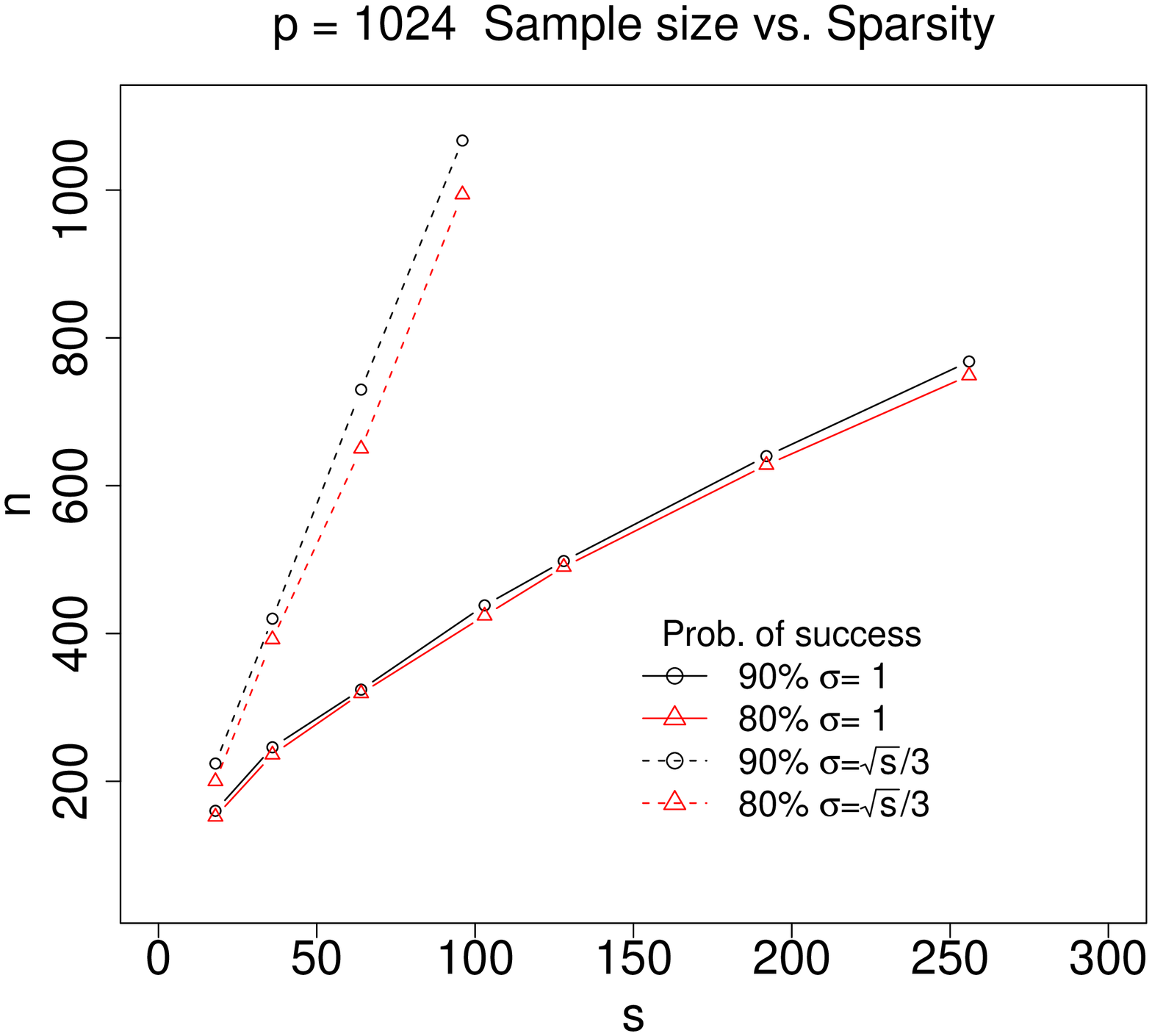}
\end{tabular} \\
(e)&(f) \\
\end{tabular}
\caption{(a) (b)
Compare the probability of success for $p = 256$ and $p=512$ under two 
noise levels. The Thresholded Lasso estimator requires much fewer samples
than the ordinary Lasso.
(c) (d) (e) show the probability of success of the Thresholded Lasso under 
different levels of sparsity and noise levels when $n$ increases for
$p =512$ and $1024$. (f) The number of samples $n$ increases almost linearly with 
$s$ for p = 1024. More samples are required to achieve the same level of success 
when $\sigma=\sqrt{s}/3$ due to the relatively low SNR.
}
\label{fig:succ-p256-four}
\end{center}
\end{figure}

\subsection{Linear Sparsity}
We next present results demonstrating that the Thresholded Lasso 
recovers a sparse model using a small number of samples per 
non-zero component in $\beta$ when $X$ is a subgaussian ensemble.  
We run under three cases of $p = 256, 512, 1024$; for each $p$, 
we increase the sparsity $s$ by roughly equal steps from 
$s= {0.2 p}/{\log (0.2 p)}$ to $p/4$.  For each $p$ and $s$, we run with different
sample size $n$.  For each tuple $(n, p, s)$, we run an experiment similar to
the one described in Section~\ref{subsec:type12} with an i.i.d. Gaussian 
ensemble $X$ being fixed while repeating Steps $2-3$ 100 times.
In Step 2, each randomly selected non-zero coordinate of $\beta$ is 
assigned a value of $\pm 0.9$ with probability $1/2$.
After each run, we compare $\hat{\beta}$ with the true $\beta$; if all
components match in signs, we count this experiment as a success. 
At the end of the 100 runs, we compute the percentage of successful runs 
as the probability of success.  We compare with the ordinary Lasso, 
for which we search over the full regularization path of LARS and
choose the $\breve{\beta}$ that best matches $\beta$ in terms of support.

We experiment with $\sigma =1$ and $\sigma = \sqrt{s}/3$.  
For $\sigma = 1$, we set $t_0 = f_t \sqrt{|\hat{S}_0|} \lambda \sigma$,
where $\hat{S}_0 =  \left\{j: \beta_{j, \init} \geq 0.5 \lambda_{n} = 0.35
\lambda \sigma \right\}$ for $\lambda_n$ as in~\eqref{eq::pen-exp}, 
and $f_t$ is chosen from the range of $[0.12, 0.24]$ 
(cf. Section~\ref{eq::large-beta}).  
For $\sigma=\sqrt{s}/3$, we set $t_0 = 0.7 \lambda \sigma$ with SNR 
being fixed.  The results are shown in Figure~\ref{fig:succ-p256-four}. 
We observe that under both noise levels, 
the Thresholded Lasso estimator requires much fewer samples than 
the ordinary lasso in order to conduct exact recovery of the sparsity 
pattern of the true linear model when all non-zero components are 
sufficiently large.  When $\sigma$ is fixed as $s$ increases, the
SNR is increasing; the experimental results illustrate the behavior of 
sparse recovery when it is  close to the noiseless setting.  
Given the same sparsity,  more samples are required for the low SNR case 
to reach the same level of success rate. 
Similar behavior was also observed for Toeplitz and Bernoulli 
ensembles with i.i.d. $\pm 1$ entries.

%\frac{\sqrt{12}}{5} \sigma \sqrt{{2 \log p}/{n}} \approx 
%We carry out the lasso using procedure {\tt lars}$(Y, X)$ that implements 
%the LARS algorithm of~\cite{EHJ04} to calculate the full regularization path.
%We then use $\lambda_n$ to retrieve the $\beta_{\init}$ from the LARS output.

\subsection{ROC comparison}
We now compare the performance of the Thresholded Lasso estimator
with the Lasso and the Adaptive Lasso by examining their ROC curves.
Our parameters are $p=512$, $n=330$, $s=64$ and we run under two cases:
$\sigma = \sqrt{s}/3$ and $\sigma = \sqrt{s}$.
In the Thresholded Lasso, we vary the threshold level from 
$0.01 \lambda \sigma$ to $1.5\lambda \sigma$. 
For each threshold, we run the experiment described in 
Section~\ref{subsec:type12} with an i.i.d. Gaussian ensemble $X$ being fixed
while repeating Steps $2-3$ 100 times. After each run, we compute the FPR
and TPR of the $\hat{\beta}$, and compute their averages after 100 runs
as the FPR and TPR for this threshold.  For the Lasso, 
we compute the FPR and TPR for each output vector along its entire
regularization path.  
For the Adaptive Lasso, we use the {\em optimal} output $\tilde{\beta}$ in terms 
of $\ell_2$ loss from the initial Lasso penalization path as the input to its 
second step, that is, we set $\beta_{\init} :=  \tilde{\beta}$
and use $w_j = 1/\beta_{\init, j}$ to compute the weights 
for penalizing those non-zero components in $\beta_{\init}$ in
the second step, while all zero components of $\beta_{\init}$ are now 
removed. We then compute the FPR and TPR for each vector
that we obtain from the second step's LARS output.
We implement the algorithms as given in ~\cite{Zou06}, the details of which 
are omitted here as its implementation has become standard.
The ROC curves are plotted in 
Figure~\ref{fig:roc-gaussian}. The Thresholded Lasso performs better than
both the ordinary Lasso and the Adaptive Lasso; its advantage is 
more apparent when the SNR is high.

\begin{figure}
\begin{center}
\begin{tabular}{cc}
\begin{tabular}{c}
\includegraphics[width=0.35\textwidth,angle=270]{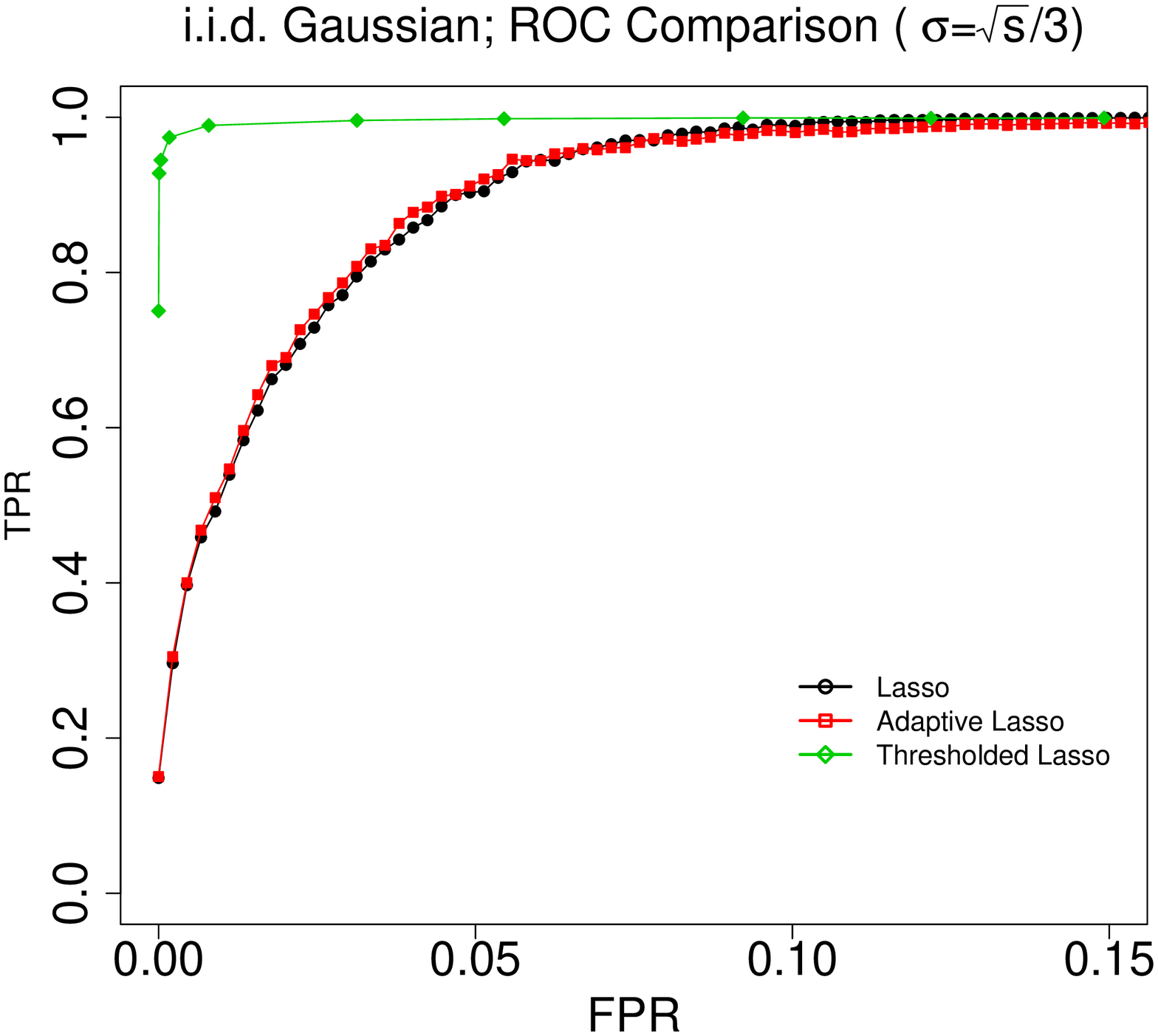}
\end{tabular}& 
\begin{tabular}{c}
\includegraphics[width=0.35\textwidth,angle=270]{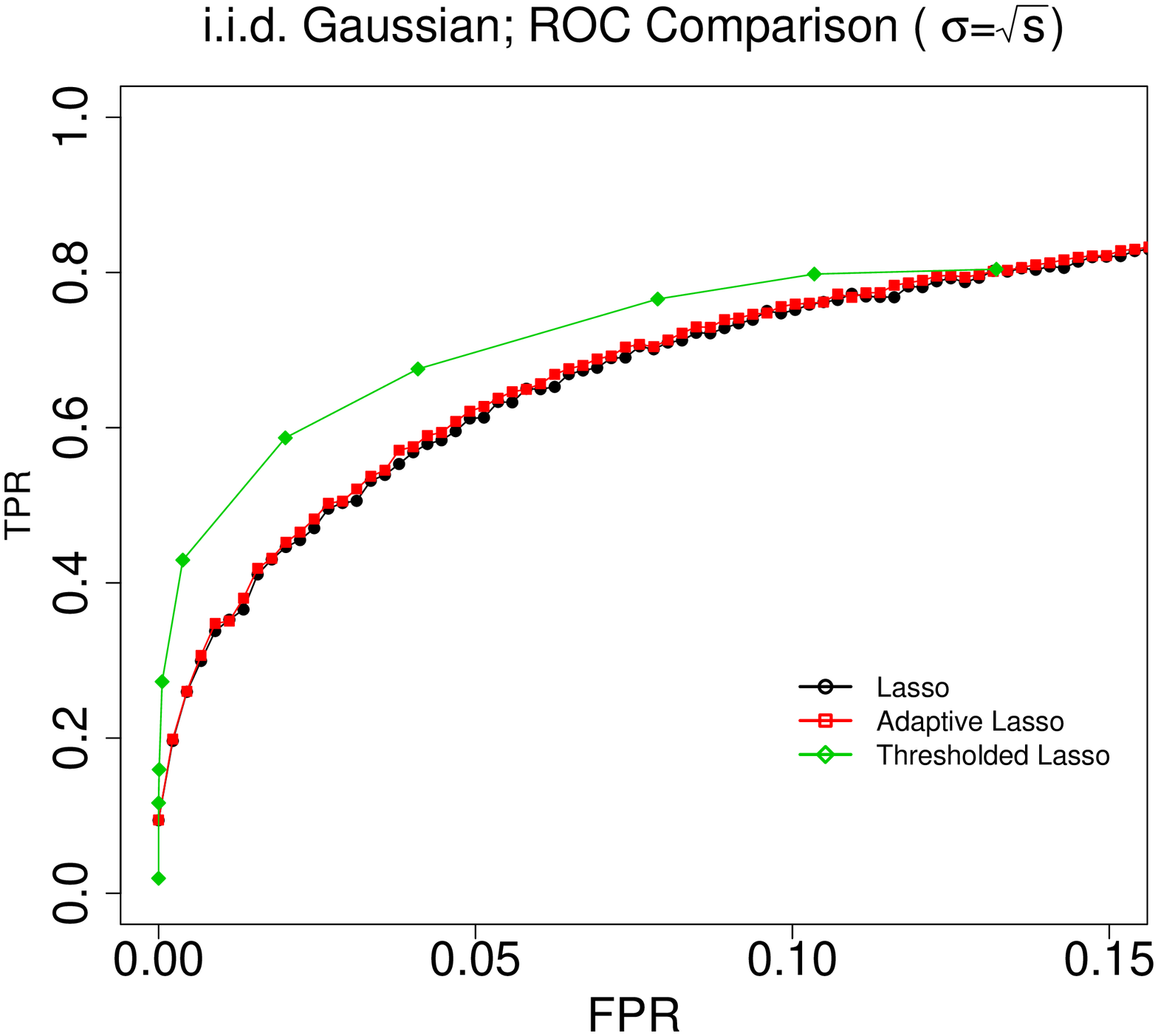} \\
\end{tabular} \\
(a) & (b) \\
\end{tabular}
\caption{$p=512$ $n=330$ $s=64$.  ROC for the Thresholded Lasso, ordinary Lasso and Adaptive
Lasso. The Thresholded Lasso clearly outperforms the ordinary Lasso and the
Adaptive Lasso for both high and low SNRs.
}
\label{fig:roc-gaussian}
\end{center}
\end{figure}

\section{Conclusion}
In this paper,
we show that the thresholding method is effective in 
variable selection and accurate in statistical estimation. 
It improves the ordinary Lasso in significant ways.
For example, we allow very significant  number of non-zero elements in the 
true parameter,  for which the ordinary Lasso would have failed.
On the theoretical side, we show that if $X$ obeys the RE
condition and if the true parameter is sufficiently sparse, the Thresholded 
Lasso achieves the $\ell_2$ loss within a logarithmic factor of the 
{\it ideal mean square error} one would achieve with an oracle, 
while selecting a sufficiently sparse model $I$.
This is accomplished when threshold level is at about 
$\sqrt{2\log p/n} \sigma$, assuming that columns of $X$ have 
$\ell_2$ norm $\sqrt{n}$.
We also report a similar result on the Gauss-Dantzig selector 
under the UUP, built upon results from~\cite{CT07}.
% which guarantees that in the worst case,
%we include $s_0$ irrelevant variables in the model, where $s_0 \leq s$
%is the smallest integer such that for $\lambda = \sqrt{2 \log p/n}$,
When the SNR is high, almost exact recovery of the 
non-zeros in $\beta$ is possible as shown in our theory; exact recovery 
of the support of $\beta$ is shown in our simulation study when $n$ 
is only linear in $s$ for several Gaussian and Bernoulli random ensembles.
When the SNR is relatively low, the inference task is difficult for any
estimator. In this case, we show that Thresholded Lasso tradeoffs 
Type I and II errors nicely: we recommend choosing the thresholding 
parameter conservatively. Algorithmic issues such as how to get an 
estimate on $\sigma$ and  parameters related to the incoherence 
conditions is left as future work. While the current focus is on $\ell_2$ 
loss, we are also interested in exploring the {\em sparsity oracle inequalities} for the 
Thresholded Lasso under the RE condition as studied in~\cite{BRT09} in our future
work.

\appendix
\section{Proof of Theorem~\ref{THM:RE}}
\label{sec:proof-thm-RE}
\begin{proofof2}
Proving  Theorem~\ref{THM:RE} involves showing that the Lasso and 
the Dantzig selector satisfy~\eqref{eq::2-normS-generic}.
These have been proved in~\cite{BRT09}.
Theorem~\ref{THM:RE} is then an immediate corollary of 
Theorem~\ref{THM:GENERIC} under assumptions therein.
We note that on $\T_a$, it holds that 
$\norm{\upsilon_{\init,\Sc}}_1 \leq k_0 \norm{\upsilon_{\init,S}}_1$, where 
$k_0 = 1$ for the  Dantzig selector when $\lambda_n \geq \basepen$ 
%(cf. Proposition~\ref{prop:initial-bound-DS}) 
and $k_0 =3$ for the Lasso, when $\lambda_n \geq 2 \basepen$ for 
the Lasso.
%(cf. Proposition~\ref{prop:initial-bound-DS}) 
Then on $\T_a$ as in~\eqref{eq::low-noise},~\eqref{eq::2-normS-generic}
% and \eqref{eq::1-norm-generic} 
holds with $B_0 =  4 K^2(s, 3)$ and $B_1 = 3 K^2(s, 3)$ for 
Lasso under $RE(s, 3, X)$ and \eqref{eq::2-normS-generic} 
%and \eqref{eq::1-norm-generic} 
holds with $B_0 = B_1 = 4 K^2(s, 1)$ 
for the Dantzig selector  under $RE(s, 1, X)$;
See~\cite{Zhou09c} for deriving the exact constants here.
\end{proofof2}

\section{Proof of Theorem~\ref{THM:PREDITIVE-ERROR}}
\begin{proofof}{Theorem~\ref{THM:PREDITIVE-ERROR}}
It is clear by construction that under $\T_a$, 
$X \hat{\beta}_I = P_{I} Y$ and $|I| \leq 2 s_0$.
Hence
\bens
\twonorm{X \hat{\beta}_I - X \beta} /\sqrt{n}
& = & 
\twonorm{(P_I - \Id) X \beta + P_{I} \epsilon} /\sqrt{n} \\
 & \leq &  
\twonorm{X_{I^c} \beta_{\drop}} /\sqrt{n}+ \twonorm{P_{I} \epsilon} /\sqrt{n} \\
& \leq & 
\sqrt{\Lambda_{\max}(s)} \twonorm{\beta_{\drop}} + 
\frac{\sqrt{|I|(1+ a) \Lambda_{\max}(|I|)} \lambda \sigma}
{\Lambda_{\min}(|I|)}
\eens
where we have on $\T_a$, for $\basepen = \sqrt{1+a} \lambda \sigma$, 
where $\lambda = \sqrt{2\log p/n}$,
\bens
\lefteqn{
\twonorm{X_I (X_{I}^T X_{I})^{-1}X_{I}^T \epsilon} /\sqrt{n}\leq
\twonorm{X_I (X_{I}^T X_{I}/n)^{-1}/\sqrt{n}} \twonorm{X_{I}^T \epsilon/n}} \\
& \leq &
\frac{\sqrt{\Lambda_{\max}(|I|)}\sqrt{|I|}\basepen }{\Lambda_{\min}(|I|)} 
 \leq 
\frac{\sqrt{|I| (1+a) \Lambda_{\max}(|I|)} \lambda \sigma}{\Lambda_{\min}(|I|)}
\eens
Now by Lemma~\ref{lemma:threshold-DS}
and~\ref{lemma:threshold-RE}, we have 
$\twonorm{\beta_{\drop}} \leq C \sqrt{s_0} \lambda \sigma$ 
for some constant $C$.
% the depends on $t_0$ and $\twonorm{h_{T_0}}$
%where $T_0 = \{1, \ldots, s_0\}$ and $h = \beta_{\init} - \beta_{T_0}$.
\end{proofof}

\section{Proof of Proposition~\ref{PROP:COUNTING-S0}}
\begin{proofof2}
%{\textnormal{Proposition~\ref{PROP:COUNTING-S0}}}
Recall that $|\beta_j| \leq \lambda \sigma$ for all $j > a_0$ as 
defined in~\eqref{eq::define-a0}; hence
for $\lambda = \sqrt{2 \log p/n}$, we have by~\eqref{eq::SR-range},
$\sum_{i> a_0}^p \min(\beta_i^2, \lambda^2 \sigma^2)= 
\sum_{i> a_0}^s \beta_i^2 \leq  (s_0 - a_0) \lambda^2 \sigma^2$;
hence
\bens
\size{ \{j \in A_0^c: |\beta_j| \geq  \sqrt{\log p/(c' n)} \sigma \}}
& \leq &  2 c' (s_0  - a_0) \text{ where }  \size{T_0 \setminus A_0} = s_0 - a_0.
\eens
Now given that $\beta_i \geq \beta_j$ for all $i \in T_0, j \in T_0^c$,
the proposition holds.
\end{proofof2}

\section{Proof of Theorem~\ref{THM:GENERIC}}
\label{appendix:thm:generic}
We first state two lemmas. Define
$\upsilon_{\init} = \beta_{\init} - \beta$ and $\upsilon^{(i)} = \hat\beta^{(i)} - \beta$.
\begin{lemma}
\label{prop:beta-min}
Under assumptions in Theorem~\ref{THM:GENERIC}, suppose on 
$\T_a \cap Q_b$,
\ben 
\label{eq::beta-min-general}
\beta_{\min} \geq \Xi + \Gamma \; \text{ where} 
\; \; \Xi := \max_{i = 0, 1} \norm{\upsilon^{(i)}_S}_{\infty} \; 
\text{ and }  \; \Gamma := \max_{i = 0, 1} t_i.
\een
Then $S \subseteq \hat{S}_2 \subseteq  \hat{S}_1$.
\end{lemma}

\begin{proof}
We have $\forall j \in S$
$\beta_{\init, j} \geq  \beta_{\min} - \norm{\upsilon_{\init, S}}_{\infty} 
\geq  \beta_{\min} - \Xi \geq \Gamma = t_0$ and \\
$\hat{\beta}^{(1)}_{j} \geq
\beta_{\min} - \norm{\upsilon^{(1)}_{S}}_{\infty} 
 \geq \beta_{\min} - \Xi \geq \Gamma \geq t_1.$
Thus the lemma holds by definition of $\hat{S}_{i}$, for $i = 0, 1, 2$.
\end{proof}
The following lemma follows from Lemma~\ref{prop:MSE-missing},
by plugging in $\twonorm{\beta_{\drop}}  = 0$.
\begin{lemma}
{\textnormal{({\bf $\ell_2$-loss for the OLS estimators})}}
\label{prop:MSE}
Suppose that $I \supseteq S$ and $|I| \leq 2s$, then
the OLS estimator $\hat{\beta}_I := (X_I^T X_{I})^{-1} X_{I}^T Y$ 
satisfies on $\T_a$,
$\shtwonorm{\hat{\beta}_{I} - \beta} 
\leq  {\basepen \sqrt{|I|}}/{\Lambda_{\min}(|I|)}$
which satisfies~\eqref{eq::2-norm-generic} with 
$B_2 = 1/(B \Lambda_{\min}(2s))$.
\end{lemma}
\silent{
\begin{proof}
%{\textnormal{Lemma~\ref{prop:MSE}}}
By definition of $\hat{\beta}_{I}$, we have 
\bens
\hat{\beta}_{I}
& = & (X_{I}^T X_{I})^{-1}X_{I}^T (X \beta + \epsilon) 
= (X_{I}^T X_{I})^{-1} X_{I}^T X_{I} \beta_{I} + 
(X_{I}^T X_{I})^{-1}X_{I}^T \epsilon  \\
& = &
\beta + (X_{I}^T X_{I})^{-1}X_{I}^T \epsilon
\eens
given that $I \supseteq S$; hence on $\T_a$, where $a \geq 0$, 
we have by~\eqref{eq::eigen-cond},
\bens
\twonorm{\hat{\beta}^{(i)} - \beta} = 
\twonorm{(X_{I}^T X_{I})^{-1}X_{I}^T \epsilon} \leq
\twonorm{(X_{I}^T X_{I}/n)^{-1}} 
\twonorm{X_{I}^T \epsilon/n}
\leq \frac{\sqrt{|I|}  \basepen}{\Lambda_{\min}(|I|)}
\eens
and thus~\eqref{eq::2-norm-MSE} holds.
\end{proof}
}
\begin{proofof}{\textnormal{Theorem~\ref{THM:GENERIC}}}
It is clear by construction that  
\ben
\label{eq::closure}
\hat{S}_2 \subseteq \hat{S}_1 \subseteq \hat{S}_0.
\een
Recall that $\hat{S}_0$ is obtained by thresholding 
$\beta_{\init}$ with $4 \lambda_n$, 
hence by~\eqref{eq::2-normS-generic},
%~\eqref{eq::1-norm-generic}, 
we have 
$$|\hat{S}_0 \setminus S| \leq \frac{\norm{\upsilon_{\init, S^c}}_1}{4 \lambda_n}
\; \leq \; \frac{B_1 \lambda_n s}{4 \lambda_n} \leq  \frac{B_1 s}{4}.$$
\begin{enumerate}
\item
If $B_1 \leq 4$, we have that $|\hat{S}_0| \leq 2s$;
\item
Otherwise, we have 
$|\hat{S}_0| \leq s + {B_1 s}/{4} \leq {B_1 s}/{2}.$
\end{enumerate}
Hence for 
$t_i = 4 \lambda_n \sqrt{|\hat{S}_{i}|}, \forall i =0,1$ and 
$\Gamma$ as in \eqref{eq::beta-min-general},
it holds by~\eqref{eq::closure} that
\ben
\label{eq::bound-Gamma}
\Gamma = t_0 = 4 \lambda_n \sqrt{|\hat{S}_{0}|} 
\leq
\lambda_n \sqrt{s} \max\left(2\sqrt{2B_1}, 4\sqrt{2} \right).
\een
Now given~\eqref{eq::beta-min-specific} and~\eqref{eq::2-normS-generic},
we have $\forall j \in S$,
\bens
\label{eq::S-bleiben-init}
\beta_{\init, j} & \geq & \beta_{\min} - \norm{\upsilon_{\init, S}}_{\infty} 
\geq \beta_{\min} -  \twonorm{\upsilon_{\init, S}} \geq \Gamma = t_0,
\eens
and hence it holds that $S \subseteq \hat{S}_1 \subseteq \hat{S}_0$
by construction of $\hat{S}_1$, and hence $t_0 \geq 4 \lambda_n \sqrt{s}$.
Now by~\eqref{eq::2-normS-generic}, we have for $s \geq {B_1^2}/{16}$,  
%~\eqref{eq::1-norm-generic}, 
\ben
\label{eq::T-1-general}
& & 
|\hat{S}_1 \setminus S| < \frac{\norm{\upsilon_{\init, \Sc}}_1}{t_0}
\leq  \frac{B_1 \lambda_n s}{4 \lambda_n \sqrt{s}} <
\frac{B_1 \sqrt{s}}{4} < s; \; \text{ and } \; |\hat{S}_1| < 2s.
\een
For  the OLS estimator $\hat{\beta}^{(1)}$ with $I =  \hat{S}_1$, by 
Lemma~\ref{prop:MSE}, we have on $\T_a$
\bens
\twonorm{\hat{\beta}^{(1)} - \beta} \leq 
\frac{\basepen \sqrt{s_1}}{\Lambda_{\min}(s_1)}
\leq \frac{\lambda_n \sqrt{s_1}}{B \Lambda_{\min}(2s)} 
\leq B_2 \lambda_n \sqrt{2s}, \text{ where } \; s_1 := \size{\hat{S}_1}
\eens
where $\lambda_n \geq B \basepen$, for $\basepen$ as 
in~\eqref{eq::low-noise}, and $B_2 = 1/(B \Lambda_{\min}(2s))$.
Clearly we have by definition of $\Xi$ in~\eqref{eq::beta-min-general},
\bens
\label{eq::bound-Xi}
\Xi & \leq & \max_{i = 0, 1} \norm{\upsilon^{(i)}_S}_{2} 
\; \leq \; \max\{B_0, \sqrt{2} B_2\} \lambda_n \sqrt{s}
\eens
and thus $\beta_{\min} \geq \Xi + \Gamma$ holds 
given~\eqref{eq::beta-min-specific} and~\eqref{eq::bound-Gamma}.
By Lemma~\ref{prop:beta-min}, we have 
$\hat{S}_i \supseteq S, \forall i = 0, 1, 2$.
It remains to show~\eqref{eq::extra-variables} 
and~\eqref{eq::2-norm-generic};
Upon thresholding $\hat{\beta}^{(1)}$ with $t_1$, we have
for $s_1 := \size{\hat{S}_1}$ and $\lambda_n \geq B \basepen$,
\bens
|\hat{S}_2 \setminus S| \leq 
{\twonorm{\upsilon^{(1)}}^2}/{t_1^2}  \leq 
\left(\frac{\basepen \sqrt{s_1}}{\Lambda_{\min}(s_1)}
\cdot \inv{4 \lambda_n \sqrt{s_1}} \right)^2 \leq
\inv{16 B^2 \Lambda_{\min}^2(s_1)}.
\eens
Now for the final estimator in~\eqref{eq::final-estimator},
we have on $\T_a \cap Q_b$ by  Lemma~\ref{prop:MSE},\\
$\twonorm{\hat{\beta}^{(2)} - \beta}  =  \twonorm{\hat{\beta} - \beta}  
= {\basepen \sqrt{|\hat{S}_2|}}/{\Lambda_{\min}(|\hat{S}_2|)}
 \leq \lambda_n B_2 \sqrt{2s}$.
\end{proofof}

\section{Proofs for the Gauss-Dantzig selector}
\label{sec:append-guass}
Recall $\beta_{\init}$ is the solution to the Dantzig selector.
We write $\beta = \beta^{(1)} + \beta^{(2)}$ where
\bens
\beta_{j}^{(1)} =  \beta_{j} \cdot 1_{1 \leq j \leq s_0} \; \text{ and } \;
\beta_{j}^{(2)} =   \beta_{j} \cdot 1_{j> s_0}.
\eens
Let $h = \beta_{\init} - \beta^{(1)}$, where $\beta^{(1)}$ is hard-thresholded 
version of $\beta$, localized to 
$T_0 = \{1, \ldots, s_0\}$.
Let $T_1$ be the $s_0$ largest positions of $h$ outside of $T_0$; 
Let $T_{01} = T_0 \cup T_1$.
\silent{
 hence
$$\twonorm{h_{T_{01}}}^2 = 
\twonorm{h_{T_{0}}}^2 +  \twonorm{h_{T_{1}}}^2
= \twonorm{\beta_{T_0, \init} - \beta_{T_0}}^2 + 
\twonorm{\beta_{T_1, \init}}^2.$$}
The proof of Proposition~\ref{prop:DS-oracle}(cf.~\cite{CT07}) yields 
the following:
% inequalities,
\ben
\label{eq::DS-T01-2-bound}
\twonorm{h_{T_{01}}} & \leq &  C'_0 \lambda_{p, \tau} \sigma \sqrt{s_0}, 
\text{ for }\; C'_0 \; \text{ as in}~\eqref{eq::DS-constants} \\
\label{eq::DS-T0c-1-bound}
\norm{h_{T_0^c}}_1 & \leq & C_1 \lambda_{p, \tau} \sigma s_0, \; 
\text{ where } \; C_1 = \left(C'_0 + \frac{1 + \delta}{1 - \delta - \theta}\right), 
\; \text{ and } \\
\twonorm{h_{T_{01}^c}}
& \leq & 
\norm{h_{T_0^c}}_1/\sqrt{s_0} \; \leq \; C_1 \lambda_{p, \tau} \sigma \sqrt{s_0},
\; \text{ (cf. Lemma~\ref{lemma::h01-bound-CT}).}
\een
\silent{
where in

We first prove Lemma~\ref{lemma:threshold-DS},
which applies to the Gauss-Dantzig selector. We then prove 
Lemma~\ref{prop:MSE-missing}, which is used  also for the 
Thresholded Lasso.}

\begin{proofof}{\textnormal{Lemma~\ref{lemma:threshold-DS}}}
Consider the set 
$I \cap T_0^c := \{j \in T_0^c: \size{\beta_{j, \init}} > t_0\}$.
It is clear by definition of $h = \beta_{\init} - \beta^{(1)}$ 
and~\eqref{eq::DS-T0c-1-bound} that
\ben
\label{eq::set-count-T0-c}
\size{I \cap T_0^c} \leq \norm{\beta_{T_0^c, \init}}_1/t_0 
= \norm{h_{T_0^c}}_1/t_0 < s_0,
\een
where $t_0 \geq C_1 \lambda_{p, \tau} \sigma$.
Thus $|I| = |I \cap T_0| + |I \cap T_0^c| \leq 2s_0$;
Now~\eqref{eq::ideal-size} holds given~\eqref{eq::set-count-T0-c} and
$|I  \cup S| = |S| + |I \cap \Sc| \leq s +  |I \cap T_0^c| < s+ s_0.$
We now bound $\twonorm{\beta_{\drop}}^2$.
By~\eqref{eq::DS-T01-2-bound} and~\eqref{eq::off-beta-norm-bound-2},
where $\drop_{11} \subset T_0$,
% and $0 \leq a_0 \leq s_0$, 
we have for $t_0 < C_4 \lambda_{p, \tau} \sigma \sqrt{s_0}$, \\
$\twonorm{\beta_{\drop}}^2  
%& \leq  & (s_0 - a_0) \lambda^2 \sigma^2 + 
%(t_0 \sqrt{a_0} + \twonorm{h_{\drop_{11}}})^2 \\
\leq  (s_0 - a_0) \lambda^2 \sigma^2 + 
(t_0 \sqrt{s_0} + \twonorm{h_{T_0}})^2 \leq 
((C_4 + C_0')^2 + 1)  \lambda_{p, \tau}^2 \sigma^2.$
\end{proofof}
\silent{
\begin{remark}
It is worth interpreting the bounds as above. By definition
\bens
h_{T_0} =  \beta_{\init, T_0} - \beta_{T_0} \; \text{ and } \; 
h_{T_0^c} =  \beta_{T_0^c, \init}.
\eens
hence a bound on $\norm{h_{T_0}}_2 = \norm{h_{T_0}}_2$ and 
$\norm{h_{T_0^c}}_1 = \norm{\beta_{T_0^c, \init}}_1$
will allow us to threshold the ``non-significant'' set of variables (which includes
those in $\Sc$) in $T_0^c$ effectively, without erasing too many 
variables from $T_0$. 
\end{remark}
Note that we can also provide a slightly stronger bound 
than~\eqref{eq::set-count-T0-c} on the number of variables 
that we select from $T_{01}$; this is obtained using the bound 
in~\eqref{eq::DS-T01-2-bound} on $h_{T_{01}}$:
\bens
|I  \cap T_{01}|  & = & |I  \cap T_0|  + |I  \cap T_1| 
\leq   s_0 + \frac{\twonorm{\beta_{T_1, \init}}^2}{t_0^2}  
\leq s_0 + \frac{\twonorm{h_{T_{01}}}^2}{t_0^2} \\ 
& \leq &  
s_0 + \frac{(C'_0 \lambda_{p, \tau} \sigma)^2 s_0}
{(C_1 \lambda_{p, \tau} \sigma)^2} \leq  s_0 (1 + (C'_0/C_1)^2 ),
\eens
which implies that we can never select every variable in $T_{01}$; 
recall, $T_1$ is the largest $s_0$ positions of $\beta_{\init}$ outside of $T_0$, 
while $T_0$ corresponds to the  largest $s_0$ positions of $\beta$ itself.

Let $\drop_1 := \drop \cap T_0$ and $\drop_2 := \drop \cap T_0^c$.
Thus 
$$\beta_{\drop}^{(1)} := (\beta_j)_{j \in T_0 \cap \drop}$$
consists of  coefficients that are significant relative to $\lambda \sigma$, 
but are dropped as $\beta_{j, \init} < t_0$; and $\beta_{\drop}^{(2)}$ 
consists of those below $\lambda \sigma$ in magnitude that are dropped. 
Hence 
\ben
\label{eq::drop-decompose}
\twonorm{\beta_{\drop}}^2 = \twonorm{\beta_{\drop}^{(1)}}^2 + 
\twonorm{\beta_{\drop}^{(2)}}^2.
\een
%$\beta_{\drop} = \beta_{\drop}^{(1)} +  \beta_{\drop}^{(2)}$, where
%\bens
%\beta_{j, \drop}^{(1)} \; = \; \beta_{j} \cdot 1_{j \in \drop_1}, \;  
%\text{ and}
%\; \beta_{j, \drop}^{(2)} \; = \;  \beta_{j} \cdot 1_{j \in \drop_2}.
%\eens
Now it is clear $\beta_{\drop}^{(2)}$ is bounded given~\eqref{eq::beta-2-small}, 
indeed, we have for $\lambda = \sqrt{\log p/n}$, 
\ben
\label{eq::drop-bound-2} 
\twonorm{\beta_{\drop}^{(2)}}^2 \leq \twonorm{\beta^{(2)}}^2 
 =  \sum_{j > s_0} \beta_j^2 = \sum_{j > s_0} \min(\beta_j^2, \lambda^2 \sigma^2)
\leq s_0 \lambda^2 \sigma^2,
\een
where the second equality is by~\eqref{eq::beta-2-small} and the 
last inequality is by definition of $s_0$ as in~\eqref{eq::define-s0}.
We now focus on $\beta_{\drop}^{(1)}$, where $|\drop_1| < s_0$; 
we have by the triangle inequality,
\ben
\nonumber
\twonorm{\beta_{\drop}^{(1)}}
& \leq & 
\nonumber
\twonorm{\beta_{\drop_1, \init}} + 
\twonorm{\beta_{\drop_1, \init} - \beta_{\drop}^{(1)}} \\
& \leq &
\label{eq::drop-bound-1}
t_0 \sqrt{\size{\drop_1}} +  \twonorm{h_{T_0}} \leq 
(C_4 + C_0') \lambda_{p, \tau} \sigma \sqrt{s_0}.
\een
where we used the fact that 
$\twonorm{\beta_{\drop_1, \init} - \beta_{\drop}^{(1)}} 
:= \twonorm{h_{\drop_1}} \leq \twonorm{h_{T_0}}$ 
and~\eqref{eq::DS-T01-2-bound};
Hence~\eqref{eq::off-beta-norm-bound} holds
given~\eqref{eq::drop-decompose}
~\eqref{eq::drop-bound-1} and~\eqref{eq::drop-bound-2}.
}
\begin{proofof}{\textnormal{Lemma~\ref{prop:MSE-missing}}}
Note that $X_{I^c} \beta_{I^c} = X_{\dropS} \beta_{\dropS}$.
We have 
%first obtain an expression for $\hat{\beta}_I$.
\ben
\nonumber
& & 
\hat\beta_{I} = (X_{I}^T X_{I})^{-1}X_{I}^T Y 
 = (X_{I}^T X_{I})^{-1} X_{I}^T (X_I \beta_I + X_{I^c} \beta_{I^c} + \epsilon)  \\
\nonumber
&  & \; \; \; \; =
\beta_I + (X_{I}^T X_{I})^{-1} X_{I}^T X_{\dropS} \beta_{\dropS} +
(X_{I}^T X_{I})^{-1}X_{I}^T \epsilon; \\ \nonumber
\lefteqn{\text{ Hence } \; \;
\twonorm{\hat{\beta}_I - \beta_{I}}  =  
\twonorm{
(X_{I}^T X_{I})^{-1} X_{I}^T X_{\dropS} \beta_{\dropS} +
(X_{I}^T X_{I})^{-1}X_{I}^T \epsilon} } \\
\label{eq::appen-two-terms}
& & \; \; \; \leq 
\twonorm{(X_{I}^T X_{I})^{-1} X_{I}^T X_{\dropS} \beta_{\dropS}} +
\twonorm{(X_{I}^T X_{I})^{-1}X_{I}^T \epsilon},
\een
where the second term is bounded as Lemma~\ref{prop:MSE}:
we have on $\T_a$,
\ben
\label{eq::noise-two-norm-bound-new}
\; \; \; \twonorm{(X_{I}^T X_{I})^{-1}X_{I}^T \epsilon}
\leq 
\twonorm{\left(\frac{X_{I}^T X_{I}}{n}\right)^{-1}} 
\twonorm{\frac{X_{I}^T \epsilon}{n}} \leq
\frac{\sqrt{|I|}}{\Lambda_{\min}(|I|)} \basepen 
%:= \frac{\sqrt{|I|} \sqrt{1+a} \lambda \sigma }{\Lambda_{\min}(|I|)} 
\een 
by~\eqref{eq::eigen-cond}, where
$\basepen = \sqrt{1+a} \lambda \sigma$ for  $\lambda = \sqrt{\log p/n}$.
We now focus on bounding the first term in~\eqref{eq::appen-two-terms}.
Let $P_I$ denote the orthogonal projection onto $I$. Let 
$$c = (X_{I}^T X_{I})^{-1} X_{I}^T X_{\dropS} \beta_{\dropS}, 
\text{ hence } X_I c = P_I  X_{\dropS} \beta_{\dropS}.$$ 
By the disjointness of $I$ and $\dropS$, 
we have for $P_I  X_{\dropS} \beta_{\dropS} := X_I c$,
\ben
\nonumber
\twonorm{P_I  X_{\dropS} \beta_{\dropS}}^2 & = & 
\ip{P_I X_{\dropS} \beta_{\dropS}, X_{\dropS} {\beta_{\dropS}}} = 
\ip{X_I c, X_{\dropS} {\beta_{\dropS}}} \\
\nonumber
& \leq &
n \theta_{|I|, |\dropS|} \twonorm{c} \twonorm{\beta_{\dropS}} \; \text{ where } \\
\label{eq::c-bound0}
\twonorm{c} & \leq & \frac{\twonorm{X_I c}}{\sqrt{n \Lambda_{\min}(|I|)}} 
\leq \frac{\twonorm{P_I  X_{\dropS} \beta_{\dropS}}}
{\sqrt{n \Lambda_{\min}(|I|)}}; \text{ Hence } \\
\label{eq::c-bound}
\twonorm{P_I  X_{\dropS} \beta_{\dropS}} & \leq &
\frac{\sqrt{n} \theta_{|I|, |\dropS|}}{\sqrt{\Lambda_{\min}(|I|)}} 
\twonorm{\beta_{\dropS}} \;
\text { where } \twonorm{\beta_{\dropS}} = \twonorm{\beta_{\drop}}
\een
and $\twonorm{c} \leq {\theta_{|I|, |\dropS|}
\twonorm{\beta_{\drop}}}/{\Lambda_{\min}(|I|)}$.
%\frac{\twonorm{X_I c}}{\sqrt{\Lambda_{\min}(|I|)}}
%\frac{\twonorm{P_I X_{\dropS} \beta_{\dropS}}}{\sqrt{n \Lambda_{\min}(|I|)}} 
Now we have on $\T_a$, by~\eqref{eq::noise-two-norm-bound-new},
%~\eqref{eq::c-bound0},~\eqref{eq::c-bound} and
\bens
\twonorm{\hat{\beta}_I - \beta_{I}} 
& \leq & 
\twonorm{(X_{I}^T X_{I})^{-1} X_{I}^T X_{\dropS} \beta_{\dropS}} +
\twonorm{(X_{I}^T X_{I})^{-1}X_{I}^T \epsilon} \\
& \leq & 
\frac{\theta_{|I|, |\dropS|}}{\Lambda_{\min}(|I|)} \twonorm{\beta_{\drop}}
+
\frac{\sqrt{|I|}}{\Lambda_{\min}(|I|)}\basepen.
\eens
Now the lemma holds given 
$\twonorm{\hat{\beta}_I - \beta}^2  = 
\twonorm{\hat{\beta}_I - \beta_{I}}^2 + \twonorm{\beta_{I} - \beta}^2$.
\end{proofof}
%& = & 
%\frac{\theta_{|I|, |\dropS|}  \twonorm{\beta_{\dropS}} +  
%\basepen \sqrt{|I|}}{\Lambda_{\min}(|I|)} 
% \leq
%\left(\frac{\theta_{|I|, |\dropS|}  \twonorm{\beta_{\dropS}} + 
%\basepen \sqrt{|I|}}{\Lambda_{\min}(|I|)}\right)^2 + 
%\twonorm{\beta_{\dropS}}^2.
%\eens

\begin{proofof}{\textnormal{Theorem~\ref{thm:ideal-MSE-prelude}}}
It holds by definition of $S_{\drop}$ that $I \cap S_{\drop} = \emptyset$.
It is clear by Lemma~\ref{lemma:threshold-DS} that
$|\dropS| < s$ and $|I| \leq 2s_0$ and
$|I \cup S_{\drop}| \leq |I \cup S| \leq s + s_0 \leq 2s$;
Thus for $\hat\beta_{I} = (X_I^T X_{I})^{-1} X_{I}^T Y$, we have 
for  $\lambda = \sqrt{2 \log p/n}$, and by~\eqref{eq::s0-upper-bound}
\bens
\label{eq::re-exp}
\lefteqn{
\twonorm{\hat{\beta}_I - \beta}^2
 \leq 
\twonorm{\beta_{\drop}}^2
\left(1 + \frac{2 \theta^2_{|I|, |\dropS|}}{\Lambda_{\min}^2(|I|)}\right) +
\frac{2 |I|}{\Lambda_{\min}^2(|I|)}\basepen^2 } \\
%& \leq &
%\twonorm{\beta_{\drop}}^2
%\left(1 + \frac{2 \theta^2_{s, 2s_0}}{\Lambda_{\min}^2(2s_0)}\right) +
%\frac{4 s_0}{\Lambda_{\min}^2(2s_0)}\basepen^2 \\
& \leq &
\lambda^2 \sigma^2 s_0
\left((\sqrt{1+a} + \tau^{-1})^2((C_0' + C_4)^2 + 1) 
\left(1 + \frac{2 \theta^2_{s, 2s_0}}{\Lambda_{\min}^2(2s_0)}\right)
+ \frac{4(1+a)}{\Lambda_{\min}^2(2s_0)}\right).
%& \leq &
%2 C_3^2 \log p \left(\sigma^2/n + \sum_{i=1}^p \min(\beta_i^2, \sigma^2/n)
%\right)
\eens
Thus the theorem holds for $C_3$ as in~\eqref{eq::DS-constants-3} 
by~\eqref{eq::s0-upper-bound}, where it holds for $\tau > 0$ that
\bens
\frac{\theta_{s, 2s_0}}{\Lambda_{\min}(2s_0)}
 \leq
\frac{\theta_{s, 2s}}{\Lambda_{\min}(2s_0)}  
 \leq
\frac{1 - \delta_{2s} - \tau}{\Lambda_{\min}(2s)} < 1
\eens
given that 
$\theta_{s, 2s} < 1 - \tau - \delta_{2s} < \Lambda_{\min}(2s)$ for $\tau > 0$.
\end{proofof}

\section{Oracle properties of the Lasso}
We first show 
%the proof for Lemma~\ref{lemma:T0-pre-loss}.
Lemma~\ref{lemma:T0-pre-loss}, which gives us the prediction error
using $\beta_{T_0}$.
\begin{lemma}
\label{lemma:T0-pre-loss}
Suppose that~\eqref{eq::eigen-max} holds. We have
for $\lambda = \sqrt{(2 \log p)/n}$.
\ben
\label{eq::T0-pre-loss}
\twonorm{X \beta - X \beta_{T_0}}/\sqrt{n}  & \leq & 
\sqrt{\Lambda_{\max}(s- s_0)} \lambda \sigma \sqrt{s_0}.
\een
\end{lemma}
\begin{proof}
%{\textnormal{Lemma~\ref{lemma:T0-pre-loss}}}
The lemma holds given that 
$\twonorm{\beta_{T_0^c}} \leq \lambda \sigma \sqrt{s_0}$, and 
$\twonorm{X \beta - X \beta_{T_0}}/\sqrt{n} 
= \twonorm{X \beta_{T_0^c}}/\sqrt{n}
\leq \sqrt{\Lambda_{\max}(s- s_0)} \twonorm{\beta_{T_0^c}}.
$
\end{proof}
We then state Lemma~\ref{lemma::h01-bound-CT}, followed by 
the proof of Theorem~\ref{thm:RE-oracle}, where 
we do not focus on obtaining the best constants.
\silent{
\begin{proofof}{\textnormal{Lemma~\ref{lemma:parallel}}}
It is sufficient to show that~\eqref{eq::parallel} holds for
$\twonorm{c}  = \twonorm{c'} = 1$.
\ben
\label{eq::parallel}
\frac{\abs{{\ip{X_I c, X_{S_{\drop}} {c'}}}}}{n} \leq  
\frac{(\Lambda_{\max}(2s) - \Lambda_{\min}(2s))}{2} \cdot \twonorm{c}  \twonorm{c'}.
\een
Indeed, by~\eqref{eq::eigen-admissible-s} and~\eqref{eq::eigen-max}
\bens
2 \Lambda_{\min}(2s) & \leq &
\frac{ \twonorm{X_I c +  X_{S_{\drop}} {c'}}^2 }{n}  
\; \leq \; 2 \Lambda_{\max}(2s) \\
2 \Lambda_{\min}(2s) & \leq & \frac{\twonorm{X_I c -  X_{S_{\drop}} {c'}}^2 }{n}  
\; \leq \; 2 \Lambda_{\max}(2s) 
\eens
Hence~\eqref{eq::parallel} follows from the parallelogram identity,
\bens
\abs{{\ip{X_I c, X_{S_{\drop}} {c'}}}} =
\frac{\twonorm{X_I c +  X_{S_{\drop}} {c'}}^2  - \twonorm{X_I c -  X_{S_{\drop}} {c'}}^2 }{4}    
\eens
\end{proofof}
}
Lemma~\ref{lemma::h01-bound-CT} is the same ( up to normalization) 
as Lemma 3.1 in~\cite{CT07}.
We note that in their original statement, the UUP condition is assumed;
a careful examination of their proof shows that it is a sufficient
but not necessary  condition; indeed we only need to assume that 
$\Lambda_{\min}(2s_0) > 0$ and $\theta_{s_0, 2s_0} < \infty$, as we show 
below. The proof is included by the end of this section 
for the purpose of a self-complete presentation.
\begin{lemma}
\label{lemma::h01-bound-CT}
Suppose $\Lambda_{\min}(2s_0) > 0$ and $\theta_{s_0, 2s_0} < \infty$.
Then
\bens
\twonorm{h_{T_{01}}} & \leq &  
\inv{\sqrt{\Lambda_{\min}(2s_0)} \sqrt{n}} 
\twonorm{X h} + \frac{\theta_{s_0, 2s_0}}{ \sqrt{s_0}\Lambda_{\min}(2s_0)}
\norm{h_{T_0^c}}_1 \\
\label{eq::RE-T0c-2-bound}
\twonorm{h_{T_{01}^c}}^2 & \leq &
\norm{h_{T_0^c}}_1^2 \sum_{k \geq s_0 + 1} 1/k^2  \; \leq \; 
\norm{h_{T_0^c}}_1^2/s_0 \; \; \text { and thus} \\
\twonorm{h} & \leq & \twonorm{h_{T_{01}}}^2 + s_0^{-1} \norm{h_{T_0^c}}^2_1
\eens
\end{lemma}
\silent{
\begin{remark}
Alternatively, we could have replaced the above bounds with
\bens
\twonorm{h_{T_{01}}} & \leq & 
\inv{\sqrt{\Lambda_{\min}(2s_0)} \sqrt{n}} \twonorm{X h} +  
\frac{\sqrt{\Lambda_{\max}(s_0)}}{\sqrt{\Lambda_{\min}(2s_0)} \sqrt{s_0}} 
\norm{h_{T_0^c}}_1.
\eens
\end{remark}
}

\begin{proofof}{\textnormal{Theorem~\ref{thm:RE-oracle}}}
Throughout this proof, we assume that $\T_a$ holds.
We use $\hat{\beta} := \beta_{\init}$ to represent the solution
to the Lasso estimator in~\eqref{eq::origin}; By the optimality of $\hat{\beta}$,
we have
\begin{eqnarray}
\label{eq::opt-lasso} 
\; \; \; \;\; 
\inv{2n} \norm{Y- X  \hat\beta}^2_2  - \inv{2n} \norm{Y- X \beta_{T_0} }^2_2 
 \leq 
\lambda_{n} \norm{\beta_{T_0}}_1 - \lambda_{n} \norm{\hat\beta}_1, \text{ where} \\
\nonumber
\; \; \; 
\twonorm{Y-X \hat \beta}^2 =
\twonorm{X \beta -X \hat \beta + \epsilon}^2 = 
\twonorm{ X \hat \beta - X \beta }^2
+ 2(\beta - \hat\beta )^T X^T \e + \norm{\e}^2_2 
\end{eqnarray} 
and similarly, we have for $\beta_0 = \beta_{T_0}$,
\begin{eqnarray*}
\label{eq::beta0-explain} 
\twonorm{Y-X  \beta_0}^2
 =  
\twonorm{X \beta -X \beta_0 + \epsilon}^2 = 
\twonorm{ X \beta - X \beta_0 }^2
+ 2(\beta - \beta_0)^T X^T \e + \norm{\e}^2_2;
\eens
Let $h = \hat{\beta} - \beta_0$.
Thus by~\eqref{eq::opt-lasso} and the triangle inequality,
we have on $\T_a$
\bens
\frac{\twonorm{ X \hat \beta - X \beta }^2}{n}
& \leq  & 
\frac{\twonorm{ X \beta - X \beta_0 }^2}{n}
+ \frac{2 h^T  X^T \e}{n} + 2 \lambda_{n}
( \norm{\beta_{0}}_1 -  \norm{h + \beta_0}_1) \\
& \leq  & 
\frac{\twonorm{ X \beta - X \beta_0 }^2}{n}
+ 2 \norm{h}_1 \norm{ \frac{X^T \e}{n}}_{\infty} + 
2 \lambda_{n} (\norm{h_{T_0}}_1 -  \norm{h_{T_0^c}}_1 )\\
& \leq  & 
\frac{\twonorm{ X \beta - X \beta_0 }^2}{n}
+ 3 \lambda_n \norm{h_{T_0}}_1 -  \lambda_n \norm{h_{T_0^c}}_1,
\eens
where we have used the fact that $\lambda_n \geq 2 \basepen$ for 
$a \geq 0$; Thus we have on $\T_a$,  
\ben
\label{eq::long-shot}
& & 
\twonorm{ X \hat \beta - X \beta }^2/n
+ \lambda_n \norm{h_{T_0^c}}_1
\leq \twonorm{ X \beta - X \beta_0 }^2/n + 3 \lambda_n  \norm{h_{T_0}}_1,
\een
which is also the starting point of our analysis on the 
oracle inequalities of the Lasso estimator.
Now we differentiate between two cases.
\begin{enumerate}
\item
Suppose that on $\T_a$,
$\twonorm{ X \beta - X \beta_0 }^2/n \geq 3 \lambda_n  \norm{h_{T_0}}_1$.
We then have that 
\ben
\label{eq::T0-error-II}
\twonorm{ X \hat \beta - X \beta }^2/n
+ \lambda_n \norm{h_{T_0^c}}_1 \leq 2 \twonorm{ X \beta - X \beta_0 }^2/n 
\een
and hence for $\lambda_n  = d_0 \lambda \sigma$, where $d_0 \geq 2$,
we have by Lemma~\ref{lemma:T0-pre-loss},
\bens
\norm{h_{T_0^c}}_1 \leq 
2 \Lambda_{\max}(s- s_0) \lambda \sigma s_0/d_0
\leq \Lambda_{\max}(s- s_0) \lambda \sigma s_0.
\eens
Now by~\eqref{eq::long-shot}, we have
\bens
 \norm{h}_1
& \leq  & 
\twonorm{ X \beta - X \beta_0 }^2/(n \lambda_n) + 4  \norm{h_{T_0}}_1 \\
& \leq  & 
7 \twonorm{ X \beta - X \beta_0 }^2/( 3 n \lambda_n)
\leq 7 \Lambda_{\max}(s- s_0) \lambda \sigma s_0/ (3d_0) \text{ and clearly} \\
\twonorm{ X h} & \leq & 
\twonorm{X \hat \beta - X \beta}  + \twonorm{X \beta - X \beta_0 }
\leq (\sqrt{2} + 1)  \twonorm{X \beta - X \beta_0 }.
\eens
By Lemma~\ref{lemma::h01-bound-CT}, we have on $\T_a$,
\bens
\twonorm{h_{T_{01}}}
& \leq &
\inv{\sqrt{n} \sqrt{\Lambda_{\min}(2s_0) }} \twonorm{X h}
+ \frac{\theta_{s_0, 2s_0}}{\Lambda_{\min}(2s_0) \sqrt{s_0}}
\norm{h_{T_0^c}}_1 \\
& \leq &
%\inv{\sqrt{n \Lambda_{\min}(2s_0) }}
%\left(\twonorm{X \hat \beta - X \beta}  + \twonorm{X \beta - X \beta_0 }\right)
%+ \frac{\theta_{s_0, 2s_0}}{\Lambda_{\min}(2s_0) \sqrt{s_0}}
%\norm{h_{T_0^c}}_1 \\
%& \leq &
\lambda \sigma \sqrt{s_0}
\frac{ \sqrt{\Lambda_{\max}(s - s_0)}}
{\sqrt{\Lambda_{\min}(2s_0)}}
\left((\sqrt{2} + 1)
+
\frac{\theta_{s_0, 2s_0} \sqrt{\Lambda_{\max}(s - s_0)}}
{\sqrt{\Lambda_{\min}(2s_0)}}\right)   \\
& = & D \lambda \sigma \sqrt{s_0}, \text{ for } \;
D =(\sqrt{2} + 1) \frac{\sqrt{\Lambda_{\max}(s - s_0)}}{\sqrt{\Lambda_{\min}(2s_0)}}
+ \frac{\theta_{s_0, 2s_0} \Lambda_{\max}(s - s_0)}{\Lambda_{\min}(2s_0)}.
\eens
\item
Otherwise, suppose on $\T_a$, we have
$\twonorm{ X \beta - X \beta_0 }^2/n \leq 3 \lambda_n  \norm{h_{T_0}}_1$;
thus
\bens
\twonorm{ X \hat \beta - X \beta}^2/(n \lambda_n) +  
\norm{h_{T_0^c}}_1  \leq   6 \norm{h_{T_0}}_1 
\eens
and $\norm{h_{T_0^c}}_1  \leq  6 \norm{h_{T_0}}_1$, which
under the $RE(s_0, 6, X)$ condition immediately implies that 
\ben
\label{eq::T0-error-I}
\twonorm{h_{T_0}} \leq K(s_0, 6) \twonorm{X h}/\sqrt{n}.
\een
The rest of the proof is devoted to this second case.
\end{enumerate}
%Now suppose that~\eqref{eq::T0-error-I} holds.
We use $K := K(s_0, 6)$ as a shorthand below. 
By~\eqref{eq::long-shot}, we have on $\T_a$,
\ben
\nonumber
\lefteqn{
\twonorm{ X \hat \beta - X \beta }^2/n
+ \lambda_n \norm{h_{T_0^c}}_1 - \twonorm{ X \beta - X \beta_0 }^2/n } \\
& & 
\nonumber
\leq 3 \lambda_n  \norm{h_{T_0}}_1 
\leq 3 \lambda_n  \sqrt{s_0} \twonorm{h_{T_0}}  \leq
\frac{3 K \lambda_n  \sqrt{s_0}}{\sqrt{n}} \twonorm{X h} 
\; \; (by~\eqref{eq::T0-error-I}) \\
& &
\label{eq::T0-error-exp}   
3  K \lambda_n  \sqrt{s_0}
\twonorm{X \hat \beta - X \beta}  
+ 3 K \lambda_n  \sqrt{s_0} \twonorm{X \beta - X \beta_0}/\sqrt{n} \\
& & 
\nonumber 
\leq  3 K \lambda_n  \sqrt{s_0} \twonorm{X \beta - X \beta_0} /\sqrt{n}
+ \twonorm{X \hat \beta - X \beta}^2/n + (3 K  \lambda_n  \sqrt{s_0})^2,
\een
from which the following immediately follows:
for $\lambda_n = d_0 \lambda \sigma \geq 2 \basepen$,  we have
\bens
\norm{h_{T_0^c}}_1
& \leq &  
\twonorm{ X \beta - X \beta_0 }^2/(n \lambda_n)  + 
3 K \sqrt{s_0} \twonorm{X \beta - X \beta_0}/\sqrt{n} + (3K/2)^2 \lambda_n  s_0\\
& = & 
\left(\twonorm{ X \beta - X \beta_0 }/\sqrt{n \lambda_n}  +  (3K /2)  \sqrt{\lambda_n  s_0} \right)^2  :=  D'_1 \lambda \sigma s_0 
\eens
where 
$D'_1 = (\sqrt{\Lambda_{\max}(s- s_0)/d_0} + 3 K(s_0, 6) \sqrt{d_0}/2)^2$.
% \Lambda_{\max}(s- s_0)/d_0 +  
%3 K(s_0, 6) \sqrt{\Lambda_{\max}(s- s_0)}   + (9 d_0 /4)  K^2(s_0, 6) 
Similarly, we can derive a bound on $\norm{h}_1$ from~\eqref{eq::long-shot}; we have on $\T_a$,
\bens
\nonumber
\lefteqn{
\twonorm{ X \hat \beta - X \beta }^2/n
+ \lambda_n \norm{h_{T_0^c}}_1 + \lambda_n \norm{h_{T_0}}_1 
- \twonorm{ X \beta - X \beta_0 }^2/n 
\leq 4 \lambda_n  \norm{h_{T_0}}_1 } \\
& \leq 4 &  \lambda_n  \sqrt{s_0} \twonorm{h_{T_0}} \leq  
4  K \lambda_n  \sqrt{s_0} \twonorm{X h}/\sqrt{n} \; \; (by~\eqref{eq::T0-error-I}) \\
%& \leq  &  
%4  \lambda_n  \sqrt{s_0} K \left(\twonorm{X \hat \beta - X \beta}  
%+ \twonorm{X \beta - X \beta_0} \right)/\sqrt{n} \\
& \leq  &  
4 K  \lambda_n  \sqrt{s_0}  \twonorm{X \beta - X \beta_0} /\sqrt{n}
+ \twonorm{X \hat \beta - X \beta}^2/n + (2 K \lambda_n  \sqrt{s_0})^2
\eens
Hence it is clear that for $\lambda_n = d_0 \lambda \sigma \geq 2 \basepen$, 
we have by  Lemma~\ref{lemma::h01-bound-CT}, on $\T_a$,
\bens
\norm{h}_1
& \leq &  
\twonorm{ X \beta - X \beta_0 }^2/(n \lambda_n)  + 
4 K \sqrt{s_0} \twonorm{X \beta - X \beta_0}/\sqrt{n} + 4 K^2  \lambda_n  s_0 \\
& = & 
\left(\twonorm{ X \beta - X \beta_0 }/\sqrt{n \lambda_n}  +  2K  \sqrt{\lambda_n  s_0} \right)^2  = D_2 \lambda \sigma s_0 
\eens
where $D_2 = (\sqrt{\Lambda_{\max}(s- s_0)/d_0} + 2 K(s_0, 6) \sqrt{d_0})^2$.
Now we derive a bound for $\twonorm{ X \hat \beta - X \beta }^2/n$; 
our starting point is~\eqref{eq::T0-error-exp}, from which by shifting items 
around and adding $(3 K \lambda_n  \sqrt{s_0})^2$ to both sides, we obtain 
\bens
\lefteqn{
 \twonorm{ X \hat \beta - X \beta }^2/n - 
3 K \lambda_n  \sqrt{s_0} \twonorm{X \hat \beta - X \beta} /\sqrt{n} + 
(3 K \lambda_n  \sqrt{s_0})^2 + \lambda_n \norm{h_{T_0^c}}_1} \\
& \leq &  \twonorm{ X \beta - X \beta_0 }^2/n +
 3K  \lambda_n  \sqrt{s_0}  \twonorm{X \beta - X \beta_0}/\sqrt{n} +
(3 K \lambda_n  \sqrt{s_0}/2)^2 
\eens
Thus we have for $\lambda_n = d_0 \lambda \sigma \geq 2 \basepen$,
\bens
\left(\inv{\sqrt{n}}{\twonorm{ X \hat{\beta} - X \beta }}
- \frac{3 K \lambda_n  \sqrt{s_0}}{2}\right)^2 
+ \lambda_n \norm{h_{T_0^c}}_1  \leq 
\left(\twonorm{ X \beta - X \beta_0 }/\sqrt{n}   
+ 3K  \lambda_n  \sqrt{s_0}/2\right)^2 
\eens
and hence
\ben
\label{eq::Xh-bound}   
\twonorm{ X \hat{\beta} - X \beta }/\sqrt{n}
& \leq &
\twonorm{ X \beta - X \beta_0 }/\sqrt{n}  + 3  K \lambda_n  \sqrt{s_0} \\
& \leq &
\nonumber
\lambda \sigma \sqrt{s_0}  \left(\sqrt{\Lambda_{\max}(s- s_0)}  + 3 d_0 K(s_0, 6)
\right).
\een
by Lemma~\ref{lemma::h01-bound-CT}.
Under $RE(s_0, 6, X)$ condition, 
we have by~\eqref{eq::Xh-bound}     
\bens
\twonorm{h_{T_0}} & \leq & K(s_0, 6) \twonorm{X h}/\sqrt{n}  \leq 
\frac{K(s_0, 6)}{\sqrt{n}}  
\left(\twonorm{X \hat \beta - X \beta}  + 
\twonorm{X \beta - X \beta_0 }\right)\\
& \leq & 
K(s_0, 6)  
\left(2 \twonorm{ X \beta - X \beta_0 }/\sqrt{n}  
+ 3  K(s_0, 6)  \lambda_n  \sqrt{s_0} \right) \\
& \leq & 
\lambda \sigma \sqrt{s_0}  K(s_0, 6) (2 \sqrt{\Lambda_{\max}(s- s_0)}  + 
3 d_0 K(s_0, 6)).
\eens
Let $T_1$ be the $s_0$ largest positions of $h$ outside of $T_0$; 
Now by a property as derived in~\cite{Zhou09c} (Proposition A.1), we also have
for $K := K(s_0, 6)$.
\bens
\twonorm{h_{T_{01}}} \leq \sqrt{2/n} K(s_0, 6) \twonorm{X h}
\leq 
\lambda \sigma \sqrt{2 s_0} K
(2 \sqrt{\Lambda_{\max}(s- s_0)}  + 3 d_0 K) \leq D_0 \lambda \sigma \sqrt{s_0} 
\eens
Moreover, we have by Lemma~\ref{lemma::h01-bound-CT},
\bens
\twonorm{\hat\beta - \beta}^2 
& \leq & 2 \twonorm{\hat \beta - \beta_{T_0}}^2 
+ 2 \twonorm{\beta - \beta_{T_0}}^2 
\leq 2 \twonorm{h}^2 + 2 \lambda^2 \sigma^2 s_0 \\
& \leq & 
2 (\twonorm{h_{T_{01}}}^2 + \norm{h_{T_0^c}}^2_1/s_0) + 
2 \lambda^2 \sigma^2 s_0
\leq 
2 \lambda^2 \sigma^2  s_0(D_0^2 + D_1^2 + 1)
\eens
We note that~\eqref{eq::pred-error-gen}
holds given~\eqref{eq::T0-error-II}  and~\eqref{eq::Xh-bound}. 
\end{proofof}
\begin{remark}
We could have bounded $\twonorm{h_{T_{01}}}$ for the second case also by 
Lemma~\ref{lemma::h01-bound-CT}; we take the form here for simplicity.
\end{remark}
\begin{proofof}{\textnormal{Lemma~\ref{lemma::h01-bound-CT}}}
Decompose $h_{T_{01}^c}$ into $h_{T_2}$, \ldots, $h_{T_K}$ 
such that $T_2$ corresponds to locations of the $s_0$ largest coefficients of 
$h_{T_{01}^c}$ in absolute values, and $T_3$ corresponds to locations of the 
next $s_0$ largest coefficients of $h_{T_{01}^c}$ in absolute values, 
and so on. 
%Hence we have  $T_0^c = \bigcup_{k=1}^K T_k$, where $K \geq 1, 
%|T_k| = s_0, \forall k = 1, \ldots, K-1$, and $|T_K| \leq s_0$. 
Let $V$ be the span of columns of $X_j$, where $j \in T_{01}$, and 
$P_V$ be the orthogonal projection onto $V$.
Decompose $P_V X h$:
\bens
P_V X h & = & P_V X h_{T_{01}} + \sum_{j \geq 2} P_V X h_{T_j}
= X h_{T_{01}} + \sum_{j \geq 2} P_V X h_{T_j} , \text{ where } \\
\twonorm{P_V X h_{T_j}} 
& \leq & 
\frac{\sqrt{n} \theta_{s_0, 2s_0}}{\Lambda_{\min}(2s_0)}\twonorm{h_{T_j}}
\text{ and }
\sum_{j \geq 2} \twonorm{h_{T_j}} \leq  \norm{h_{T_0^c}}_1 /\sqrt{s_0}
\eens
see~\cite{CT07}) for details; Thus we have
\bens
\twonorm{X h_{T_{01}}} & = & 
\twonorm{P_V X h - \sum_{j \geq 2} P_V X h_{T_j} }  \leq 
\twonorm{P_V X h} + \twonorm{\sum_{j \geq 2} P_V X h_{T_j} } \\
& \leq & 
\twonorm{X h} + \sum_{j \geq 2} \twonorm{P_V X h_{T_j} } 
\leq
\twonorm{X h} +  \frac{\sqrt{n} \theta_{s_0, 2s_0}}
{\sqrt{\Lambda_{\min}(2s_0)} \sqrt{s_0}} \norm{h_{T_0^c}}_1,
\eens
where  we used the fact that $\twonorm{P_V} \leq 1$.
Hence the lemma follows given 
$\twonorm{h_{T_{01}}} \leq 
\inv{\sqrt{\Lambda_{\min}(2s_0)} \sqrt{n}}  \twonorm{X h_{T_{01}}}$.
For other bounds,
the fact that the $k$th largest value of $h_{T_0^c}$ obeys
$\size{h_{T_{0}^c}}_{(k)} \leq  \norm{h_{T_0^c}}_1 / k$ has been used; 
see~\cite{CT07}.
\end{proofof}

\section{Proofs for Lemmas in Section~\ref{SEC:GENERAL-RE}}
Let $\lambda = \sqrt{ 2 \log p/n}$. 
By definition of $s_0$ as in~\eqref{eq::define-s0}, we have
$\sum_{i=1}^p \min(\beta_i^2, \lambda^2 \sigma^2)  \leq  
s_0 \lambda^2 \sigma^2$.
We write $\beta = \beta^{(11)} +  \beta^{(12)} + \beta^{(2)}$ where
\bens
\beta_{j}^{(11)} =  \beta_{j} \cdot 1_{1 \leq j \leq a_0}, \; 
\beta_{j}^{(12)} =   \beta_{j} \cdot 1_{a_0 < j \leq s_0}, \;\text{ and } \; 
\beta^{(2)} =    \beta_{j} \cdot 1_{j > s_0}.
\eens
Now it is clear that $\sum_{j \leq a_0} \min(\beta_j^2, \lambda^2 \sigma^2) =   
a_0 \lambda^2 \sigma^2$ and hence
\ben
\label{eq::SR-range}
& & 
\sum_{j > a_0} \min(\beta_j^2, \lambda^2 \sigma^2) = 
\twonorm{ \beta^{(12)} + \beta^{(2)}}^2 \leq  (s_0 - a_0) \lambda^2 \sigma^2.
\een
%correspondingly,
%$\beta_{\drop}^{(2)}$ and $\beta_{\drop}^{(12)}$ consist of those below 
%$\lambda \sigma$ in magnitude that are dropped. 
\begin{proofof}{\textnormal{Lemma~\ref{lemma:threshold-general}}}
It is clear for $\drop_{11} = \drop \cap A_0$, we have 
$\drop_{11} \subset A_0 \subset T_0 \subset S$. Let
$\beta_{\drop}^{(11)} := (\beta_j)_{j \in A_0 \cap \drop}$ consist of  
coefficients of $\beta$ that are above $\lambda \sigma$ in their 
absolute values but are dropped as $\beta_{j, \init} < t_0$.
Now by~\eqref{eq::SR-range}, we have
\bens
\label{eq::drop-decompose-new}
\; \; \; \; \; 
\twonorm{\beta_{\drop}}^2 \leq  \twonorm{\beta_{\drop}^{(11)}}^2 + 
\twonorm{ \beta^{(12)} + \beta^{(2)}}^2 
\leq \twonorm{\beta_{\drop}^{(11)}}^2 + (s_0 - a_0) \lambda^2 \sigma^2,
\eens
where  $|\drop_{11}| \leq a_0$ and thus we have by the triangle inequality,
\ben
\nonumber
\twonorm{\beta_{\drop}^{(11)}}
& \leq & 
\twonorm{\beta_{\drop_{11}, \init}} + \twonorm{\beta_{\drop_{11}, \init} - 
\beta_{\drop}^{(11)}} \leq 
t_0 \sqrt{\size{\drop_{11}}} +  \twonorm{h_{\drop_{11}}}  \\
& \leq &
\label{eq::drop-bound-3}
t_0 \sqrt{a_0} +  \twonorm{h_{\drop{11}}};
%\leq t_0 \sqrt{a_0} +  \twonorm{h_{A_0}};
\een
Thus \eqref{eq::off-beta-norm-bound-2} holds. Now
we replace the bound of $|\drop_{11}| \leq a_0$ with
$|\drop_{11}| \leq 
\frac{\twonorm{h_{\drop_{11}}}^2}{|\beta_{\min,A_0} - t_0|^2}$
in~\eqref{eq::drop-bound-3} to obtain
$$\twonorm{\beta_{\drop}^{(11)}} 
\leq t_0 \frac{\twonorm{h_{\drop_{11}}}}{\beta_{\min,A_0} - t_0}
+  \twonorm{h_{\drop_{11}}} =
\twonorm{h_{\drop_{11}}}
\frac{\beta_{\min,A_0}}{\beta_{\min,A_0} - t_0}$$
which proves~\eqref{eq::off-beta-norm-bound-alt}.
\end{proofof}

\begin{proofof}{\textnormal{Lemma~\ref{lemma:threshold-general-II}}}
Suppose $\T_a \cap Q_c$ holds.
It is clear by the choice of $t_0$ in~\eqref{eq::ideal-t0} and by 
\eqref{eq::betaA-min-cond} that
$\min_{i \in A_0} \hat{\beta}_{i} \geq
\beta_{\min, A_0}  -  \norm{h_{A_0}}_{\infty} \geq t_0$
and $\drop_{11} = \emptyset$.
Thus by~\eqref{eq::ideal-t0}, we can bound
$|I \cap T_0^c|$, depending on which one is applicable, by
$|I \cap T_0^c| \leq
 {\norm{\beta_{T_0^c, \init}}_1}/{t_0} \leq \breve{s}_0$ or by
$|I \cap T_0^c|  \leq 
 {\twonorm{\beta_{T_0^c, \init}}^2}/{t_0^2} \leq \breve{s}_0.$
Moreover, 
the bounds on $\twonorm{\hat{\beta}_I - \beta}^2$ follows immediately
from Lemma~\ref{prop:MSE-missing}, on event $\T_a$, where 
$\theta^2_{|I|, |\dropS|}$ is bounded in Lemma~\ref{lemma:parallel}
given $|I| + |\dropS| \leq s + |I \cap T_0^c| \leq s + \breve{s}_0 \leq 2s$.
\end{proofof}

{\bf Acknowledgment.}
The author thanks Larry Wasserman and Peter B\"{u}hlmann 
for helpful discussions at the early stage of this work, and 
Sara van de Geer for her positive feedback when I presented this work in 
the Workshop of the DFG-SNF Research Group at University of Bern, 
in September 2009. The author would like to thank the NIPS 2009 reviewers 
for their constructive comments, and John Lafferty and Jun Gao for their 
encouragements and support.

\bibliography{./local}

\end{document}